\documentclass[a4paper]{amsart}
\usepackage{graphicx}
\usepackage{amssymb}
\usepackage{amsmath}
\usepackage{amsthm}
\usepackage[all]{xy}
\usepackage{ragged2e}
\usepackage{makecell}
\usepackage{enumitem}
\usepackage{lscape}
\usepackage{leftindex}
\usepackage{mathtools}
\usepackage{hyperref}

\newtheorem*{thmmA}{\bf Theorem A}
\newtheorem*{thmmB}{\bf Theorem B}
\newtheorem{thm}{Theorem}[section]

\newtheorem{lem}[thm]{Lemma}
\newtheorem{exm}[thm]{Example}
\newtheorem{conj}[thm]{Conjecture}
\newtheorem{prop}[thm]{Proposition}
\theoremstyle{definition}
\newtheorem{defn}[thm]{Definition}
\theoremstyle{remark}
\newtheorem{rem}[thm]{\bf Remark}
\newtheorem{nota}[thm]{\bf Notation}
\numberwithin{table}{section}

\newcommand{\xr}{\mathbb{R}}
\newcommand{\xh}{\mathbb{H}}
\newcommand{\xc}{\mathbb{C}}

\newcommand{\im}{\mathbf{i}}
\newcommand{\jm}{\mathbf{j}}
\newcommand{\km}{\mathbf{k}}
\newcommand{\lm}{\mathbf{l}}

\begin{document}
	
	\title{The locally complexified-gentle algebras}
	\author[J. Li]{Jie Li}
    \address{School of Mathematics, 
	Hefei University of Technology, Hefei, Anhui, PR China}
    \email{lijie0@hfut.edu.cn}

     \author[C. Zhang]{Chao Zhang}
     \address{Department of Mathematics,
     School of Mathematics and Statistics,
    Guizhou University,
   550025, Guiyang,
    PR China}
   \email{zhangc@amss.ac.cn}
	
	\begin{abstract}
		We call an $\xr$-algebra locally complexified-gentle if it becomes a locally gentle $\xc$-algebra up to Morita equivalence after complexification. We use modulated quivers to introduce two types of locally complexified-gentle algebras and show that they are Morita equivalent to some semilinear clannish algebras.
	\end{abstract}


    
	\maketitle
	
	\section{Introduction}
	Gentle algebras were introduced by Assem and Skowro\'{n}ski \cite{AS1987}, which are described by certain combinatorial conditions on quivers and ideals of their path algebras; see Definition~\ref{usual defn}. Research of these algebras and their generalizations has become highly active since they appear in many other subjects; see \cite{GP1999, LD2009}.
	
	In representation theory, gentle algebras are well-studied. Indecomposable objects in both module and derived categories of gentle algebras were classified; see \cite{BR1987} for module categories and \cite{BD2017, BM2003} for derived categories. Some homological properties are shown; see \cite{ALP2016, CB1989, GR2005, KM2015, KY2018}. Recently, research on geometric models for gentle algebras and their relation with other subjects has attracted much attention; see \cite{ABCP2010, BC2021, BC2023, HKK2017, HZZ2023, OPS2018, QZZ2022}.
	
	Gentle algebras were originally defined over algebraically closed fields, and it is natural to explore such ``gentle" algebras over more general fields, such as the real number field. Base field change is a well-studied approach in representation theory. Many properties are found to be compatible with certain field extensions, such as homological dimensions \cite{ERZ1957, JL1982} and the representation type \cite{JL1982, KS2001, LZ2022}. Our motivation is to study and describe algebras that become gentle algebras after certain base field extensions. We expect these algebras to have properties similar to gentle algebras. We can also use descent theory to generalize the known results of gentle algebras. 
	
	\medskip
	
	In this article, we focus on the extension from the real number field $\xr$ to the complex number field $\xc$. We call an $\xr$-algebra \textbf{(locally) complexified-gentle} if it becomes a (locally) gentle $\xc$-algebra up to Morita equivalence after complexification; see Definition \ref{defn of cmp-gen alg}. We show that (locally) complexified-gentle algebras have many properties similar to those of (locally) gentle algebras: they are of generically tame representation type and Gorenstein; see Proposition \ref{prop of cmp-gen alg}.
	
	By Gabriel's theorem, when the base field is algebraically closed, say $\xc$, each finite-dimensional algebra has a quiver presentation. In this case, the definition of gentle algebras works properly. However, if we consider a perfect but not algebraically closed base field, say $\xr$, we need modulated quivers to include more algebras; see \cite{DR1976}. Hence, we study locally complexified-gentle algebras in the form of $T(Q,\mathcal{M})/I$, where $Q$ is a quiver with modulation $\mathcal{M}$ and $I$ is some ideal of the tensor algebra $T(Q,\mathcal{M})$; see Section 2.1. 
	
	To describe some complexified-gentle algebras $T(Q,\mathcal{M})/I$, we first define gentle vertices in $T(Q,\mathcal{M})/I$ locally. They are ordinarily gentle vertices (with $\xr$, $\xh$ or $\xc$) and specially gentle vertices; see Section~3.2. Next, we give two types of locally complexified-gentle algebras. 
	
	\begin{thmmA}[Theorem \ref{main}]
		Let $T(Q,\mathcal{M})/I$ be an $\xr$-algebra satisfying one of the following conditions: \begin{enumerate}
			\item each vertex is ordinarily gentle with $\xr$ (or with $\xh$) in $T(Q,\mathcal{M})/I$ and the ideal $I$ is generated by $I_v, v\in Q_0$,
			\item each vertex is ordinarily gentle with $\xc$ or specially gentle in $T(Q,\mathcal{M})/I$ and the ideal $I$ is generated by $I_v, v\in Q_0$,
		\end{enumerate} where $I_v$ is some bimodule associated with those arrows staring or ending at $v$. Then $T(Q,\mathcal{M})/I$ is locally complexified-gentle.
	\end{thmmA}
	
	Those algebras that satisfy condition (1) are called \textbf{locally complexified-gentle of uniform type} with $\xr$ (or with $\xh$). Those satisfying condition (2) are called \textbf{locally complexified-gentle of special type}.
	
	We conjecture that these two types of locally complexified-gentle algebras are all locally complexified-gentle algebras up to Morita equivalence. Evidence and examples are given in Section 3.3.
	\medskip
	
	Recently, Bennett-Tennenhaus and Crawley-Boevey have introduced in \cite{BC2024} semilinear clannish algebras over division rings and described their finite-dimensional indecomposable representations under some mind conditions. If we take the base ring to be $\xc$, then a $\xc$-semilinear clannish algebra is an $\xr$-algebra. To some extent, semilinear algebras provide us a way to explore algebras over more general fields. In this article, we consider the relation between semilinear clannish algebras and locally complexified-gentle algebras.

	Locally complexified-gentle algebras of uniform type are exactly gentle algebras over $\xr$ or $\xh$, and thus are semilinear clannish. By \cite{BC2024}, their finite-dimensional indecomposable representations can be classified directly.
	
	Representations of locally complexified-gentle algebras of special type seem complicated to describe by modulated quivers; see \cite{DR1976, DR1978}. For these algebras, we study a special kind of $\xc$-semilinear clannish algebras, called \textbf{$\xc$-semilinear clannish algebras of gentle type}; see Definition~\ref{csca of gentle}. By giving their modulated quiver presentations, we prove our second main result:
	
	\begin{thmmB}[Theorem \ref{main 2}]
		An $\xr$-algebra is locally complexified-gentle of special type if and only if it is Morita equivalent to some $\xc$-semilinear clannish algebra of gentle type. Therefore, locally complexified-gentle algebras are semilinear clannish algebras up to Morita equivalence.
	\end{thmmB} 
	
	We explain why we choose the field extension $\xc/\xr$. Firstly, over $\xr$, all division rings and simple bimodules over these rings are clear and provide typical situations which may occur in other field extensions. Moreover, the complexificaion of an $\xr$-algebra is clear in aspect of (modulated) quiver presentations; see \cite{LJ2023}. There are many relative studies; see \cite{DR1976, DR1978}. Finally, the Galois group of $\xc/\xr$ is $C_2$, the cyclic group of order $2$, which is closely related to gentle algebras. For example, a skew-gentle algebra is Morita equivalent to a skew-group algebra $C_2* A$, where $A$ is some gentle algebra; see \cite{GP1999}.
	
	The article is organized as follows. 
	
	In Section~2, we recall the modulated quiver presentation $T(Q,\mathcal{M})/I$ of $\xr$-algebras and the complexified quiver presentation $\xc\Gamma/J$ of $T(Q,\mathcal{M})/I\otimes_{\xr}\xc$. 
	
	In Section~3.1, we define locally complexified-gentle algebras and prove some properties. In Section~3.2, we define gentle vertices in $T(Q,\mathcal{M})/I$ and prove that they can be gentle in $\xc\Gamma/J$. Then, we give two types of locally complexified-gentle algebras in Section~3.3 using gentle vertices, and Theorem A is proved. In Section~3.4, we give some examples of complexified-gentle algebras and exclude some examples that are not complexified-gentle. 
	
	We consider in Section~4.1 algebras $\xc_\sigma Q/\langle S\rangle$, certain quotients of $\xc$-semilinear path algebras and give their modulated quiver presentation $T(Q^\mathrm{b},\mathcal{M})$ in Section~4.2. In Section~4.3, we introduce $\xc$-semilinear clannish algebras of gentle type $\xc_\sigma Q/\langle S\cup Z\rangle$ and describe their presentations $T(Q^\mathrm{b},\mathcal{M})/I$. In Section 4.4 we prove Theorem B. In Section 4.5, we explain how to classify finite-dimensional indecomposable representations of locally complexified-gentle algebras of uniform or special type.

	\section{The \texorpdfstring{$\xr$}{}-algebras and their complexifications}
	
	In this section, we recall some knowledge about $\xr$-algebras, their modulated quiver presentations, their complexifications and complexified quiver presentations.

	\subsection{Tensor algebras of modulated quivers over \texorpdfstring{$\xr$}{}}
	
	Let $Q=(Q_0,Q_1,s,t)$ be a finite quiver, where $Q_0$ is a finite set of vertices, $Q_1$ is a finite set of arrows, $s\colon Q_1\rightarrow Q_0$ maps an arrow to its starting vertex, and $t\colon Q_1\rightarrow Q_0$ maps an arrow to its terminal vertex.
	
	By the Frobenius theorem, the (isomorphism classes of) finite-dimensional division rings over the real number field are only the field $\mathbb{R}$ itself, the field $\mathbb{C}$ of complex numbers and the quaternion algebra $\mathbb{H}$ over $\xr$. Denote the conjugate of $c\in\xc$ by $\overline{c}$. We view $\mathbb{H}=\xr 1_{\xh}\oplus\xr\jm\oplus\xr\km\oplus\xr\lm$ as $$\{\begin{bmatrix} a & b \\ -\overline{b} & \overline{a} \end{bmatrix}
	|a,b\in \mathbb{C} \},$$ an $\mathbb{R}$-subalgebra of $M_2(\mathbb{C})$ (the full matrix algebra of the $2\times2$ matrices over $\mathbb{C}$) by identifying $1_\xh$ as $\begin{bmatrix} 1 & 0 \\ 0 & 1 \end{bmatrix}$, $\jm$ as $\begin{bmatrix} \im & 0 \\ 0 & -\im \end{bmatrix}$, $\km$ as $\begin{bmatrix} 0 & 1 \\  -1 & 0 \end{bmatrix}$ and $\lm$ as $\begin{bmatrix} 0 & \im \\  \im& 0 \end{bmatrix}$. The bimodules between these division rings up to isomorphisms are clear; see the construction of $\Gamma_1$ in Section 2.2 or \cite[Section 2]{LJ2023} for details. In this paper, bimodule structures on matrices are given by multiplication of matrices.
	
	A \textbf{modulation} $\mathcal{M}$ of $Q$ (over $\xr$) is a map given below: for each vertex $i$, the image $\mathcal{M}(i)$ is a division ring $\xr$, $\xh$ or $\xc$; for each arrow $\alpha$, the image $\mathcal{M}(\alpha)$ is one of the simple $\mathcal{M}(t(\alpha))$-$\mathcal{M}(s(\alpha))$-bimodules listed in Section 2.2; see Remark \ref{all sim bim} or \cite[Section~2]{LJ2023}. In this case, we call $(Q,\mathcal{M})$ a \textbf{modulated quiver} (over $\xr$). 
	
	\begin{nota}
		We write the modulation on a quiver to denote a modulated quiver. For example, the modulated quiver $(\xymatrix@C=4ex@R=-1ex{u\ar@<.4ex>[r]^{\alpha}\ar@<-.4ex>[r]_{\beta} & v}, \, \mathcal{M} :u\mapsto \xr, v\mapsto \xc, \alpha\mapsto \xc,\beta\mapsto\xc)$ can be denoted by $\xymatrix@C=6ex@R=-1ex{\xr_u\ar@<.4ex>[r]^{\xc_\alpha,\,\xc_{\beta}}\ar@<-.4ex>[r]& \xc_v}$, by $\xymatrix@C=6ex@R=-1ex{\xr\ar@<.4ex>[r]^{\xc_\alpha,\,\xc}\ar@<-.4ex>[r]& \xc_v}$ if we only care about vertex $v$ and arrow $\alpha$, or simply by $\xymatrix@C=5ex@R=-1ex{\xr\ar@<.4ex>[r]^{\xc,\,\xc}\ar@<-.4ex>[r]& \xc}$.
	\end{nota}

    \medskip
	The tensor algebra of a modulated quiver $(Q,\mathcal{M})$ is the tensor algebra of the bimodule $\underset{\alpha\in Q_1}{\oplus}\mathcal{M}(\alpha)$ over the ring $\underset{i\in Q_0}{\prod}\mathcal{M}(i)$, where the bimodule structure is compatible with the quiver structure. We denote it by $$T(Q,\mathcal{M}):=(\underset{i\in Q_0}{\prod}\mathcal{M}(i))\oplus(\underset{\alpha\in Q_1}{\oplus}\mathcal{M}(\alpha))\oplus(\underset{\alpha\in Q_1}{\oplus}\mathcal{M}(\alpha))^{\otimes2}\oplus\cdots.$$
	
	For a path $p=\alpha_n\alpha_{n-1}\cdots\alpha_1$ in $Q$, set $$\mathcal{M}(p):=\mathcal{M}(\alpha_n)\otimes_{\mathcal{M}(s(\alpha_n))}\mathcal{M}(\alpha_{n-1})\otimes_{\mathcal{M}(s(\alpha_{n-1}))}\cdots\otimes_{\mathcal{M}(s(\alpha_2))}\mathcal{M}(\alpha_1).$$ 
	
	\medskip
	\subsection{Complexified quiver presentations}
	In this subsection, we recall some notation from \cite{LJ2023}. 
    
    Given a modulated quiver $(Q,\mathcal{M})$, its \textbf{complexified quiver} $\Gamma=(\Gamma_0,\Gamma_1,s,t)$ is constructed as follows.

	\textbf{Vertices of $\Gamma$}:
	\begin{enumerate}
		\item Each $i\in Q_0$ with $\mathcal{M}(i)=\mathbb{R}$ or $\mathbb{H}$ gives a vertex $i$ in $\Gamma_0$.
		\item Each $i\in Q_0$ with $\mathcal{M}(i)=\mathbb{C}$ gives two vertices $i$ and $\overline{i}$ in $\Gamma_0$.
	\end{enumerate}
	
	\textbf{Arrows of $\Gamma$}: Let $\alpha:i\rightarrow j$ be an arrow in $Q_1$, modulated with a simple bimodule $\mathcal{M}(\alpha)=S$.
	\begin{enumerate}
		\item If $S=\xr\in\mathbb{R}\mbox{-}\mathbb{R}\mbox{-mod}$, it gives $\alpha:i\rightarrow j$ in $\Gamma_1$.
		\item If $S=\xh\in\mathbb{H}\mbox{-}\mathbb{H}\mbox{-mod}$, it gives $\alpha:i\rightarrow j$ in $\Gamma_1$.
		\item If $S=\xc\in\mathbb{R}\mbox{-}\mathbb{C}\mbox{-mod}$, it gives $\alpha:i\rightarrow j$ and $\overline{\alpha}:\overline{i}\rightarrow j$ in $\Gamma_1$.
		\item If $S=\xc\in\mathbb{C}\mbox{-}\mathbb{R}\mbox{-mod}$, it gives $\alpha:i\rightarrow j$ and $\overline{\alpha}:i\rightarrow\overline{j}$ in $\Gamma_1$. \item If $S=\xh\in\mathbb{R}\mbox{-}\mathbb{H}\mbox{-mod}$, it gives $\alpha:i\rightarrow j$ and $\overline{\alpha}:i\rightarrow j$ in $\Gamma_1$.
		\item If $S=\xh\in\mathbb{H}\mbox{-}\mathbb{R}\mbox{-mod}$, it gives $\alpha:i\rightarrow j$ and $\overline{\alpha}:i\rightarrow j$ in $\Gamma_1$.
		\item If $S=\begin{bmatrix}
			\xc\\ \xc
		\end{bmatrix}\in\mathbb{H}\mbox{-}\mathbb{C}\mbox{-mod}$, it gives $\alpha:i\rightarrow j$ and $\overline{\alpha}:\overline{i}\rightarrow j$ in $\Gamma_1$.
		\item If $S=\begin{bmatrix}
			\xc& \hspace{-2mm}\xc
		\end{bmatrix}\in\mathbb{C}\mbox{-}\mathbb{H}\mbox{-mod}$, it gives $\alpha:i\rightarrow j$ and $\overline{\alpha}:i\rightarrow\overline{j}$ in $\Gamma_1$.
		\item If $S=\mathbb{C}\in\mathbb{C}\mbox{-}\mathbb{C}\mbox{-mod}$, it gives $\alpha:i\rightarrow j$ and $\overline{\alpha}:\overline{i}\rightarrow\overline{j}$ in $\Gamma_1$.
		\item If $S=\overline{\mathbb{C}}\in\mathbb{C}\mbox{-}\mathbb{C}\mbox{-mod}$, it gives $\alpha:\overline{i}\rightarrow j$ and $\overline{\alpha}:i\rightarrow\overline{j}$ in $\Gamma_1$.
	\end{enumerate}
	
	\begin{rem}\label{all sim bim}
		All the simple bimodules over division rings up to isomorphisms are listed above. To see (7) and (8), notice that $\xh\otimes_{\xr}\xc\overset{\theta}{\simeq}M_2(\xc)$; see Section 2.3.
	\end{rem}
	
	The quiver $\Gamma=(\Gamma_0,\Gamma_1,s,t)$ is well-defined and there is an \textbf{automorphism} $\tau$ of $\Gamma$ defined as follows.
	\begin{enumerate}
		\item  For each $i$ in $\Gamma_0$ given by $i\in Q_0$, if $\overline{i}$ exists, $\tau(i)=\overline{i}$ and $\tau(\overline{i})=i$; if not, $\tau(i)=i$.
		\item  For each $\alpha$ in $\Gamma_1$ given by $\alpha\in Q_1$, if $\overline{\alpha}$ exists, $\tau(\alpha)=\overline{\alpha}$ and $\tau(\overline{\alpha})=\alpha$; if not, $\tau(\alpha)=\alpha$.
	\end{enumerate}

	We also have a surjection $$\pi\colon\Gamma\rightarrow Q,$$ where $\pi(x)=x=\pi(\overline{x})$ (if $\overline{x}$ exists), $\forall x\in Q_0\cup Q_1$. We say that $x$ in $\Gamma_0\cup \Gamma_1$ is a \textbf{fiber} of $y$ in $Q_0\cup Q_1$ if $\pi(x)=y$. A path $p'$ in $\Gamma$ is called a \textbf{fiber} of a path $p$ in $Q$ if the arrows in $p'$ are fibers of those in $p$ in orders. 
	
	We call $\Gamma$ the complexified quiver of $(Q,\mathcal{M})$ because of the following results. Here $\xc\Gamma$ is the path algebra of $\Gamma$ over $\xc$. For convenience, we identify $\xc\Gamma$ with the tensor algebra $T(\prod_{i\in\Gamma_0}\xc e_i, \bigoplus_{\alpha\in\Gamma_1}\xc \alpha)$ via the canonical isomorphism between them, where $e_i$ denotes the trivial path corresponding to $i\in\Gamma_0$.
	
	\begin{prop}\cite[Proposition 3.5, Remark 3.6]{LJ2023}\label{ios after complexification}
		Let $\Gamma$ be the complexified quiver of a modulated quiver $(Q,\mathcal{M})$. Then there is an isomorphism of $\xc$-algebras
		$$\Psi\colon \mathbf{e}(T(Q, \mathcal{M})\otimes_{\xr}\xc)\mathbf{e}\longrightarrow\xc \Gamma,$$ where $\mathbf{e}$ is a full idempotent of $T(Q, \mathcal{M})\otimes_{\xr}\xc$. Moreover, for each path $p$ in $Q$, $$\Psi( \mathbf{e}(\mathcal{M}(p)\otimes\xc)\mathbf{e})=\oplus_{i=1}^{n}\xc q_i,$$ where $q_1,\dots,q_n$ are all the paths in $\Gamma$ that are fibers of $p$.
	\end{prop}
	
	In Section 2.3, we briefly recall the isomorphism $\Psi$ and the idempotent $\mathbf{e}$ for the reader's convenience; one can see \cite{LJ2023} for details. 
	
	We fix the isomorphism $\Psi$ above. For each ideal $I$ of $T(Q,\mathcal{M})$, set $$J:=\Psi(\mathbf{e}(I\otimes_{\mathbb{R}}\xc)\mathbf{e})\subset\xc \Gamma.$$ Then $J$ is an ideal of $\xc \Gamma$. Hence, we have 
	
	\begin{thm}\cite[Theorem 3.7]{LJ2023}
		The algebra $T(Q, \mathcal{M})/I\otimes_{\mathbb{R}}\xc$ is Morita equivalent to $\xc \Gamma/J$. Moreover, $I\subseteq (\underset{\alpha\in Q_1}{\bigoplus}\mathcal{M}(\alpha))^{\otimes 2}$ if and only if $J\subseteq (\underset{\alpha\in \Gamma_1}{\bigoplus}\xc\alpha)^{\otimes 2}$, and $I$ is an admissible ideal if and only if so is $J$.
	\end{thm}
	
	We call $\xc \Gamma/J$ the complexifed quiver presentation of $T(Q, \mathcal{M})/I$.
	
	\begin{exm}\label{skgen exm}
		Assume that $(Q,\mathcal{M})=\xr_u\overset{\xc_\beta}{\rightarrow} \xc_v\overset{\xc_\alpha}{\rightarrow} \xr_w$ and $I=\langle 1\otimes 1\rangle\subset\xc\otimes_{\xc}\xc=\mathcal{M}(\alpha)\otimes_{\xc}\mathcal{M}(\beta)$. Then 
		$$\Gamma=\makecell{\xymatrix@C=5ex@R=-1ex{ & v \ar[dr]^{\alpha}& \\ u\ar@<.5ex>[ur]^{\beta}\ar@<-0.5ex>[dr]_{\overline{\beta}}&  & w \\ & \overline{v} \ar[ur]_{\overline{\alpha}}&}}\mbox{, and } J=\langle\alpha\beta+\overline{\alpha}\overline{\beta}\rangle;$$ see the proof of \cite[Lemma 6.4]{LJ2023}.
		The algebra $\xc\Gamma/J$ is isomorphic to a skew-gentle algebra $\xc\Gamma'/J'$ (see the definition in \cite[Section 4]{GP1999}), where
		$$\Gamma'=\makecell{\xymatrix@C=5ex@R=-1ex{ u'\ar[r]^{\beta'}& v'\ar@(ul,ur)_{s} \ar[r]^{\alpha'} & w' }}\mbox{, and } J'=\langle\alpha'\beta',ss-s\rangle.$$
		The isomorphism is induced by $f\colon\xc\Gamma'\rightarrow\xc\Gamma$, with $f(e_{u'})=e_u$, $f(e_{v'})=e_v+e_{\overline{v}}$, $f(e_{w'})=e_w$, $f(s)=e_{\overline{v}}$, $f(\alpha')=\alpha+\overline{\alpha}$ and $f(\beta')=\beta-\overline{\beta}$.
	\end{exm}
	
	\medskip
	
	\subsection{The isomorphism \texorpdfstring{$\Psi$}{}} We recall the isomorphism $\Psi$ which will be used for the computations in Table \ref{psi(ab)}; see \cite[Section 2 and 3]{LJ2023} for more details. 
	
	The idempotent $\mathbf{e}$ of $T(Q,\mathcal{M})\otimes_{\xr}\xc$ is given by
	$$\mathbf{e}=(\mathbf{e}_{i})_{i\in Q_0}\in\underset{i\in Q_0}{\prod}(\mathcal{M}(i)\otimes_{\xr}\xc),$$ where
	\begin{equation*}
		\mathbf{e}_{i}=\begin{cases}
			1_{\mathcal{M}(i)}\otimes 1 & \text{if } \mathcal{M}(i)=\xr \text{ or } \xc\\ (1_{\xh}\otimes1-\jm\otimes\im)/2
			&\text{if } \mathcal{M}(i)=\xh
		\end{cases}.\end{equation*}
	
	The isomorphism $\Psi$ is given as the composition of the following isomorphisms:
	\begin{align*}
		\mathbf{e}(T(Q,\mathcal{M})\otimes_{\xr}\xc)\mathbf{e} &=\mathbf{e}(T(\underset{i\in Q_0}{\prod}\mathcal{M}(i),\underset{\alpha\in Q_0}{\oplus}\mathcal{M}(\alpha))\otimes_{\xr}\xc)\mathbf{e}\\
		&\overset{\psi_1}{\simeq}\mathbf{e}T(\underset{i\in Q_0}{\prod}(\mathcal{M}(i)\otimes_{\xr}\xc),\underset{\alpha\in Q_0}{\oplus}(\mathcal{M}(\alpha)\otimes_{\xr}\xc))\mathbf{e}\\
		&\overset{\theta}{\simeq}\theta(\mathbf{e})T(\underset{i\in Q_0}{\prod}\theta(\mathcal{M}(i)\otimes_{\xr}\xc),\underset{\alpha\in Q_0}{\oplus}\theta(\mathcal{M}(\alpha)\otimes_{\xr}\xc))\theta(\mathbf{e})\\
		&\overset{\psi_2}{\simeq} T(\underset{i\in Q_0}{\prod}\theta(\mathbf{e}(\mathcal{M}(i)\otimes_{\xr}\xc)\mathbf{e}),\underset{\alpha\in Q_0}{\oplus}\theta(\mathbf{e}(\mathcal{M}(\alpha)\otimes_{\xr}\xc)\mathbf{e}))\\
		&\overset{\psi_3}{\simeq} \xc\Gamma.
	\end{align*}
	
	(1) The isomorphism $\psi_1$ is given in a canonical way as in the isomorphism:
	$$(\mathcal{M}\otimes_{\mathcal{A}}\mathcal{N})\otimes_{\xr}\xc\simeq(\mathcal{M}\otimes_{\xr}\xc)\otimes_{\mathcal{A}\otimes_{\xr}\xc}(\mathcal{N}\otimes_{\xr}\xc),$$
	where $\mathcal{A}$ is an $\xr$-algebra, $\mathcal{M}$ and $\mathcal{N}$ are $\mathcal{A}$-bimodules, which is given by 
	$$(m\otimes n)\otimes c\mapsto (m\otimes c)\otimes (n\otimes 1).$$
	
	(2) Now we give $\theta$ and $\psi_3$. For vertices, we give the isomorphism of $\xc$-algebras $$\psi_3\circ\theta:\underset{i\in Q_0}{\prod}\mathbf{e}(\mathcal{M}(i)\otimes_{\xr}\xc)\mathbf{e}\simeq \underset{i\in Q_0}{\prod}\theta(\mathbf{e}(\mathcal{M}(i)\otimes_{\xr}\xc)\mathbf{e})\simeq\xc\Gamma_0$$ as below. Let $i$ be a vertex in $Q$.
	
	$\cdot$ If $\mathcal{M}(i)=\xr$, then $\mathbf{e}(\xr\otimes_{\xr}\xc)\mathbf{e}\simeq\xc\simeq \xc e_i$ , given by $$a\otimes c\mapsto ac\mapsto ace_i \mbox{, where } \mathbf{e}_i=1_{\xr}\otimes1_{\xc}\mbox{ and } \theta(\mathbf{e}_i)=1_\xc.$$
	
	$\cdot$ If $\mathcal{M}(i)=\xh$, then $\mathbf{e}(\xh\otimes_{\xr}\xc)\mathbf{e}\simeq\theta(\mathbf{e})M_2(\xc)\theta(\mathbf{e})\simeq \xc e_i$, given by $$\mathbf{e}_i(\begin{bmatrix} a & b \\ -\overline{b} & \overline{a} \end{bmatrix}\otimes c)\mathbf{e}_i\mapsto\begin{bmatrix}1&0\\0&0\end{bmatrix}\begin{bmatrix} ac & bc \\ -\overline{b}c & \overline{a}c \end{bmatrix}\begin{bmatrix}1&0\\0&0\end{bmatrix}\mapsto ace_i.$$
	
	$\cdot$ If $\mathcal{M}(i)=\xc$, then $\mathbf{e}(\xc\otimes_{\xr}\xc)\mathbf{e}\simeq\xc\times\xc\simeq \xc e_i\oplus\xc e_{\overline{i}}$, given by $$a\otimes c\mapsto (ac, \overline{a}c)\mapsto ace_i+\overline{a}ce_{\overline{i}}\mbox{, where } \mathbf{e}_i=1_{\xc}\otimes1_{\xc}\mbox{ and } \theta(\mathbf{e}_i)=1_{\xc\times\xc}.$$
	
	For arrows, the isomorphism $$\psi_3\circ\theta:\underset{\alpha\in Q_1}{\oplus}\mathbf{e}(\mathcal{M}(\alpha)\otimes_{\xr}\xc)\mathbf{e}\simeq\underset{\alpha\in Q_1}{\oplus}\theta(\mathbf{e}(\mathcal{M}(\alpha)\otimes_{\xr}\xc)\mathbf{e})\simeq \xc\Gamma_1$$ is given as below. For each $\alpha\in Q_1$, via the isomorphisms of algebras, we list below isomorphisms of $\xc e_{t(\alpha)}$-$\xc e_{s(\alpha)}$-bimodules.
	
	$\cdot$ If $\mathcal{M}(\alpha)={_{\xr}\xr_{\xr}}$, then $\mathbf{e}(\xr\otimes_{\xr}\xc)\mathbf{e}\simeq\xc\simeq \xc \alpha$, given by $$a\otimes c\mapsto ac\mapsto ac\alpha.$$
	
	$\cdot$ If $\mathcal{M}(\alpha)={_{\xh}\xh_{\xh}}$, then $\mathbf{e}(\xh\otimes_{\xr}\xc)\mathbf{e}\simeq\theta(\mathbf{e})M_2(\xc)\theta(\mathbf{e})\simeq \xc \alpha$, given by $$\mathbf{e}_{t(\alpha)}(\begin{bmatrix} a & b \\ -\overline{b} & \overline{a} \end{bmatrix}\otimes c)\mathbf{e}_{s(\alpha)}\mapsto\begin{bmatrix}1&0\\0&0\end{bmatrix}\begin{bmatrix} ac & bc \\ -\overline{b}c & \overline{a}c \end{bmatrix}\begin{bmatrix}1&0\\0&0\end{bmatrix}\mapsto ac\alpha.$$
	
	$\cdot$ If $\mathcal{M}(\alpha)={_{\xr}\xc_{\xc}}$, then $\mathbf{e}(\xc\otimes_{\xr}\xc)\mathbf{e}\simeq\xc\oplus\xc\simeq \xc \alpha\oplus\xc\overline{\alpha}$, given by $$a\otimes c\mapsto (ac,\overline{a}c)\mapsto ac\alpha+\overline{a}c\overline{\alpha}.$$
	
	$\cdot$ If $\mathcal{M}(\alpha)={_{\xc}\xc_{\xr}}$, then $\mathbf{e}(\xc\otimes_{\xr}\xc)\mathbf{e}\simeq\xc\oplus\xc\simeq \xc \alpha\oplus\xc\overline{\alpha}$, given by $$a\otimes c\mapsto (ac,\overline{a}c)\mapsto ac\alpha+\overline{a}c\overline{\alpha}.$$
	
	$\cdot$ If $\mathcal{M}(\alpha)={_{\xr}\xh_{\xh}}$, then
	$\mathbf{e}(\xh\otimes_{\xr}\xc)\mathbf{e}\simeq M_2(\xc)\theta(\mathbf{e})\simeq \xc \alpha\oplus\xc\overline{\alpha}$, given by $$(\begin{bmatrix} a & b \\ -\overline{b} & \overline{a} \end{bmatrix}\otimes c)\mathbf{e}_{s(\alpha)}\mapsto\begin{bmatrix} ac & bc \\ -\overline{b}c & \overline{a}c \end{bmatrix}\begin{bmatrix}1&0\\0&0\end{bmatrix}\mapsto ac\alpha-\overline{b}c\overline{\alpha}.$$
	
	$\cdot$ If $\mathcal{M}(\alpha)={_{\xh}\xh_{\xr}}$, then $\mathbf{e}(\xh\otimes_{\xr}\xc)\mathbf{e}\simeq\theta(\mathbf{e})M_2(\xc)\simeq \xc \alpha\oplus\xc\overline{\alpha}$, given by $$\mathbf{e}_{t(\alpha)}(\begin{bmatrix} a & b \\ -\overline{b} & \overline{a} \end{bmatrix}\otimes c)\mapsto\begin{bmatrix}1&0\\0&0\end{bmatrix}\begin{bmatrix} ac & bc \\ -\overline{b}c & \overline{a}c \end{bmatrix}\mapsto ac\alpha+bc\overline{\alpha}.$$
	
	$\cdot$ If $\mathcal{M}(\alpha)=\leftindex_{\xh}{\begin{bmatrix}\xc\\ \xc\end{bmatrix}}_{\xc}$, then 
	$\mathbf{e}(\begin{bmatrix}\xc\\ \xc\end{bmatrix}\otimes_{\xr}\xc)\mathbf{e}\simeq\theta(\mathbf{e}) M_2(\xc) \simeq\xc \alpha\oplus\xc\overline{\alpha}$, given by $$\mathbf{e}_{t(\alpha)}(\begin{bmatrix} a\\ b \end{bmatrix}\otimes c)\mapsto\begin{bmatrix}1&0\\0&0\end{bmatrix}\begin{bmatrix} ac & -\overline{b}c \\ bc & \overline{a}c \end{bmatrix} \mapsto ac\alpha-\overline{b}c\overline{\alpha}.$$
	
	$\cdot$ If $\mathcal{M}(\alpha)=\leftindex_{\xc}{\begin{bmatrix}\xc & \xc\end{bmatrix}}_{\xh}$, then 
	$\mathbf{e}(\begin{bmatrix}\xc & \xc\end{bmatrix}\otimes_{\xr}\xc)\mathbf{e}\simeq M_2(\xc) \theta(\mathbf{e})\simeq\xc \alpha\oplus\xc\overline{\alpha}$, given by $$(\begin{bmatrix}a & b \end{bmatrix}\otimes c)\mathbf{e}_{s(\alpha)}\mapsto \begin{bmatrix}ac & bc\\ -\overline{b}c &\overline{a}c \end{bmatrix}\begin{bmatrix}1&0\\0&0\end{bmatrix}\mapsto ac\alpha-\overline{b}c\overline{\alpha}.$$
	
	$\cdot$ If $\mathcal{M}(\alpha)={_{\xc}\xc_{\xc}}$, then $\mathbf{e}(\xc\otimes_{\xr}\xc)\mathbf{e}\simeq\xc\oplus\xc\simeq \xc \alpha\oplus\xc\overline{\alpha}$, given by $$a\otimes c\mapsto (ac,\overline{a}c)\mapsto ac\alpha+\overline{a}c\overline{\alpha}.$$
	
	$\cdot$ If $\mathcal{M}(\alpha)={_{\xc}\overline{\xc}_{\xc}}$, then $\mathbf{e}(\overline{\xc}\otimes_{\xr}\xc)\mathbf{e}\simeq\xc\oplus\xc\simeq \xc \alpha\oplus\xc\overline{\alpha}$, given by $$a\otimes c\mapsto (ac,\overline{a}c)\mapsto ac\alpha+\overline{a}c\overline{\alpha}.$$

	(3) The map $\psi_2$ can be reduced to the following isomorphism:
	\begin{align*}
		&\theta(\mathcal{M}(\alpha)\otimes_{\xr}\xc)\otimes_{\theta(\mathcal{M}(i)\otimes_{\xr}\xc)}\theta(\mathcal{M}(\alpha)\otimes_{\xr}\xc)\\
		\simeq&\theta((\mathcal{M}(\alpha)\otimes_{\xr}\xc)\mathbf{e})\otimes_{\theta(\mathbf{e}(\mathcal{M}(i)\otimes_{\xr}\xc)\mathbf{e})}\theta(\mathbf{e}(\mathcal{M}(\alpha)\otimes_{\xr}\xc)),
	\end{align*}
	where $\alpha,\beta\in Q_1$ with $s(\beta)=t(\alpha)=i$. We still denote it by $\psi_2$.
	
	If $\mathcal{M}(i)\neq\xh$, we set $\psi_2$ as the identity map since $\mathbf{e}(\mathcal{M}(i)\otimes_{\xr}\xc)\mathbf{e}=\mathcal{M}(i)\otimes_{\xr}\xc$. If $\mathcal{M}(i)=\xh$, $\psi_2$ can be given by the following lemma after choosing the appropriate rings and bimodules: $D=\xc$, $A,B\in\{\xc, M_2(\xc), \xc\times\xc\}$, $M\in\{\xc,\begin{bmatrix}
		\xc\\ \xc 
	\end{bmatrix}\}$, $N\in\{\xc,\begin{bmatrix}
		\xc& \xc 
	\end{bmatrix}\}$.
	
	\begin{lem}\label{map 2}
		Let $M$ be a $A$-$D$-bimodule and $N$ be a $D$-$B$-bimodule, where $A$, $B$ and $D$ are rings. We have isomorphisms of $A$-$B$-bimodules.
		$$\begin{bmatrix}
			M & M 
		\end{bmatrix}\otimes_{M_2(D)}\begin{bmatrix}
			N \\ N \end{bmatrix}\simeq \begin{bmatrix}
			M & 0 
		\end{bmatrix}\otimes_{\begin{bsmallmatrix}
				D &0  \\ 0&0 \end{bsmallmatrix}}\begin{bmatrix}
			N\\0 	\end{bmatrix}\simeq M\otimes_{D}N \mbox{, given by}$$
		$$\begin{bmatrix}
			m_1 & m_2 
		\end{bmatrix}\otimes\begin{bmatrix}
			n_1 \\ n_2 \end{bmatrix}\mapsto \begin{bmatrix}
			m_1 & 0 
		\end{bmatrix}\otimes \begin{bmatrix}
			n_1 \\ 0 \end{bmatrix}+\begin{bmatrix}
			m_2 & 0	\end{bmatrix}\otimes \begin{bmatrix}
			n_2 \\ 0 \end{bmatrix} \mapsto m_1\otimes n_1+m_2\otimes n_2.$$
	\end{lem} 
	\begin{proof}
		The inverse map of the first isomorphism is given by	$$\begin{bmatrix}
			m & 0
		\end{bmatrix}\otimes\begin{bmatrix}
			n \\ 0\end{bmatrix}\mapsto \begin{bmatrix}
			m & 0
		\end{bmatrix}\otimes\begin{bmatrix}
			n \\ 0\end{bmatrix}.$$
	\end{proof}

	\section{The locally complexified-gentle algebras}
	
	In this section, we introduce and study $\xr$-algebras which become gentle algebras over $\xc$ after complexification. 
	
	For convenience, when we say ``A xor B", we mean ``A or B, but not both". For each vertex $v$ in a quiver $Q$, set $$v^+:=\{\alpha\in Q_1\,|\,s(\alpha)=v\},\mbox{ }v^-:=\{\alpha\in Q_1\,|\,t(\alpha)=v\}.$$ By $N_v:=\{t(\alpha)\in Q_0\,|\,\alpha\in v^+\}\cup\{s(\alpha)\in Q_0\,|\,\alpha\in v^-\}$ we denote the set of vertices adjacent to $v$. There are loops on $v$ if and only if $v\in N_v$. The number of elements in a set $S$ is denoted by $|S|$. 
	
	\medskip
	\subsection{Locally complexified-gentle algebras}
	To introduce and study $\xr$-algebras which becomes gentle algebras after complexification, we first recall the usual definition of a gentle algebra over a field $k$.
	
	\begin{defn}\label{usual defn}
		A $k$-algebra is called \textbf{locally gentle} if it is Morita equivalent to some algebra $k Q/I$, where $Q$ is a finite quiver and $I$ is an ideal of the path algebra $kQ$, such that each vertex $v$ in $kQ/I$ is a \textbf{gentle vertex}, i.e.,
		\begin{enumerate}
			\item[(G1)]  $|v^+|\leq 2$ and $|v^-|\leq 2$;
			\item[(G2)]  if there are arrows $\alpha\in v^+$ and $\beta, \gamma \in v^-$ with $\beta\neq\gamma$, then $\alpha\beta$ xor $\alpha\gamma$ belongs to $I$; if there are arrows $\alpha\in v^-$ and $\beta, \gamma \in v^+$ with $\beta\neq\gamma$, then $\beta\alpha$ xor $\gamma\alpha$ belongs to $I$;
		\end{enumerate}
		and $I$ satisfies that
		\begin{enumerate}
			\item[(G3)]  $I$ can be generated by some paths of length two. 
		\end{enumerate}
		
		A locally gentle algebra is called \textbf{gentle} if it is finite-dimensional.
	\end{defn}
	
	Now we use the field extension $\xc/\xr$ to extend the above definition.
	\begin{defn}\label{defn of cmp-gen alg}
		An $\xr$-algebra $A$ is called a (locally) complexified-gentle algebra if $A\otimes_{\xr}\xc$ is a (locally) gentle algebra.
	\end{defn}
	
	Locally complexified-gentle algebras have many similar properties to those of locally gentle algebras. We list some as instances.
	
	\begin{prop}\label{prop of cmp-gen alg}
		We have the following facts. \begin{enumerate}
			\item (Locally) complexified-gentle algebras are closed under Morita equivalence.
			\item Complexified-gentle algebras are closed under derived equivalence.
			\item Complexified-gentle algebras are generically tame.
			\item (Locally) complexified-gentle algebras are Gorenstein.
		\end{enumerate}
	\end{prop}
	
	\begin{proof}
		(1) Assume that $A$ is an $\xr$-algebra such that $A\otimes_{\xr}\xc$ is a locally gentle algebra and that $B$ is Morita equivalent to $A$ with $B=\mathrm{End}_A(P)$ for some projective generator $P$ of $A$. Then $P\otimes_{\xr}\xc$ is a projective generator of $A\otimes_{\xr}\xc$ and $B\otimes_{\xr}\xc\simeq\mathrm{End}_{A\otimes_{\xr}\xc}(P\otimes_{\xr}\xc)$ is Morita equivalent to $A\otimes_{\xr}\xc$. Since our definition of locally gentle algebras is up to Morita equivalence, $B\otimes_{\xr}\xc$ is also a locally gentle algebra. Thus, $B$ is a locally complexified-gentle algebra.
		
		(2) Assume that $A$ is a finite-dimensional $\xr$-algebra such that $A\otimes_{\xr}\xc$ is a gentle algebra and that $B$ is derived equivalent to $A$. By \cite[Lemma 2.3]{LJ2021}, $B\otimes_{\xr}\xc$ is derived equivalent to $A\otimes_{\xr}\xc$. Thus, $B\otimes_{\xr}\xc$ is also a gentle algebra since gentle algebras are closed under derived equivalence; see \cite[Corollary 1.2]{SZ2003}.
		
		(3) Gentle algebras over $\xc$ are of tame representation type. Hence, they are generically tame; see \cite[Theorem 2.3]{KS2001} for definitions and more details. By \cite[Theorem 4.3]{KS2001}, generically tameness is compatible with the extension $\xc/\xr$. Hence, a complexified-gentle algebra is generically tame.
		
		(4) Let $A$ be a locally complexified-gentle algebra. Then $A \otimes_{\xr} \xc$ is a locally gentle algebra. By \cite[Theorem 4.28]{FOZ2024} $A \otimes_{\xr} \xc$ is Gorenstein, i.e., the injective dimensions $\mathrm{injdim}({_{A \otimes_{\xr} \xc}}A \otimes_{\xr} \xc)$ and $\mathrm{injdim}(A \otimes_{\xr} \xc{_{A \otimes_{\xr} \xc}})$ are finite. We have $A \otimes_{\xr} \xc\simeq A\oplus A$ as both a left and a right $A$-module. So $\mathrm{injdim}({_A}A)=\mathrm{injdim}({_A}A\otimes_{\xr}\xc)$. Due to \cite[Proposition 2]{ERZ1957}, $\mathrm{injdim}({_A}A\otimes_{\xr}\xc)\leq\mathrm{injdim}({_{A \otimes_{\xr} \xc}}A \otimes_{\xr} \xc)$. So $\mathrm{injdim}({_A}A)$ is finite. Dually, $\mathrm{injdim}(A_A)$ is finite. Therefore, $A$ is Gorenstein.  
	\end{proof}
	
	There are many other properties of locally complexified-gentle algebras that can be investigated. In the rest of this paper, our purpose is to describe locally complexified-gentle algebras in a more concrete form. By (1) in the above proposition and Gabriel's theorem, a reasonable approach is to use modulated quivers. 
	
	In the following subsections, we give two types of locally complexified-gentle algebras and conjecture that they are all locally complexified-gentle algebras. An $\xr$-algebra $A$ is called \textbf{connected} if $A$ is Morita equivalent to some $T(Q,\mathcal{M})/I$ such that $Q$ is a connected quiver. Notice that an algebra is (locally) complexified-gentle if and only if so is each of its connected components.
	
	\subsection{Gentle vertices in \texorpdfstring{$T(Q,\mathcal{M})/I$}{}}
	
	Let $T(Q,\mathcal{M})$ be the tensor algebra of a modulated quiver $(Q,\mathcal{M})$ over $\xr$ and $I$ be an ideal of $T(Q,\mathcal{M})$. We consider which conditions on $T(Q,\mathcal{M})/I$ will ensure that its complexified quiver presentation $\xc\Gamma/J$ is a gentle algebra. We first consider such ``local'' conditions for each vertex in $T(Q,\mathcal{M})/I$.

	\begin{defn}
		A vertex $v$ in $T(Q,\mathcal{M})/I$ is called \textbf{ordinarily gentle} with $\xr$ (with $\xh$) if $\mathcal{M}(u)=\mathcal{M}(v)=\xr \mbox{ (respectively,}= \xh),\forall u\in N_v$, and satisfies the following:
		\begin{enumerate}
			\item[(G1)] $|v^+|\leq 2$ and $|v^-|\leq 2$;
			\item[(G2)] for $\alpha\in v^+$ and $\beta, \gamma \in v^-$ with $\beta\neq\gamma$, $I$ includes $\mathcal{M}(\alpha\beta)$ xor $\mathcal{M}(\alpha\gamma)$;\\ for $\alpha\in v^-$ and $\beta, \gamma \in v^+$ with $\beta\neq\gamma$, $I$ includes $\mathcal{M}(\beta\alpha)$ xor $\mathcal{M}(\gamma\alpha)$.
		\end{enumerate}
	\end{defn}
	
	\begin{defn}
		A vertex $v$ in $T(Q,\mathcal{M})/I$ is called \textbf{ordinarily gentle} with $\xc$ if $\mathcal{M}(v)=\xc$ and satisfies the following:
		\begin{enumerate}
			\item[(G1)] $|v^+|\leq 2$ and $|v^-|\leq 2$;
			\item[(G2)] for $\alpha\in v^+$ and $\beta, \gamma \in v^-$ with $\beta\neq\gamma$, $I$ includes $\mathcal{M}(\alpha\beta)$ xor $\mathcal{M}(\alpha\gamma)$; \\
			for $\alpha\in v^-$ and $\beta, \gamma \in v^+$ with $\beta\neq\gamma$, $I$ includes $\mathcal{M}(\beta\alpha)$ xor $\mathcal{M}(\gamma\alpha)$;
			\item[(G3)] for $\alpha\in v^+$ and $\beta\in v^-$ with  $\mathcal{M}(t(\alpha))\neq\xc$ and $\mathcal{M}(s(\beta))\neq\xc$,\\ $I\cap\mathcal{M}(\alpha\beta)=\mathcal{M}(\alpha\beta)$ or $0$.
		\end{enumerate}
	\end{defn}
	
	Now we consider vertices $v$ such that $\mathcal{M}(v)\neq\xc$ and $v$ are not ordinarily gentle. In fact, we consider 18 cases of modulated quivers formed by $v$, $N_v$, $v^-$ and $v^+$ listed in Table \ref{spgenv}. In the first 10 cases, we denote the other arrow parallel to $\gamma$ by $\gamma^{\im}$, which is convenient for discussion in the rest of this paper. 
	
	In Table \ref{spgenv}, we also write down elements $r_0$, $r_1$, $r_0^\im$ (only for cases $\xr\xr\xr$ and $\xr\xh\xr$) and $r_1^\im$ (only for cases $\xr\xr\xr$ and $\xr\xh\xr$) in $\underset{\alpha\in v^+, \beta\in v^-}{\bigoplus}\mathcal{M}(\alpha\beta)$. They are generators of these bimodules. We denote $a_{\mathcal{M}(\gamma)}$ by $a_{\gamma}$ for short in the table.
	
	\begin{table}[!htb]
		\small{\begin{tabular}{ll|l|l}
				\cline{1-4}
				cases & modulated quivers  &$r_0 (r_0^\im) \in\hspace{-3mm}\underset{\alpha\in v^+, \beta\in v^-}{\bigoplus}\hspace{-3mm}\mathcal{M}(\alpha\beta)$ &$r_1$ ($r_1^\im$)  $\in\hspace{-3mm}\underset{\alpha\in v^+, \beta\in v^-}{\bigoplus}\hspace{-3mm}\mathcal{M}(\alpha\beta)$\\ \hline
				
				$\xr\xr\xr$ & \makecell{$\xymatrix@C=4ex@R=-1ex{\xr_u\ar@<.4ex>[r]^{\xr_\beta,\,\xr_{\beta^\im}}\ar@<-.4ex>[r] & \xr_v \ar@<.4ex>[r]^{\xr_\alpha,\,\xr_{\alpha^\im}}\ar@<-0.4ex>[r] & \xr_w}$} &\begin{tabular}[c]{@{}l@{}}$1_{\alpha}\otimes1_{\beta}-1_{\alpha^\im}\otimes1_{\beta^\im}$\\ \small{($-1_{\alpha}\otimes1_{\beta^\im}-\!1_{\alpha^\im}\otimes1_{\beta}$)}\end{tabular}  & \begin{tabular}[c]{@{}l@{}}$1_{\alpha}\otimes1_{\beta}+1_{\alpha^\im}\otimes1_{\beta^\im}$\\ \small{($1_{\alpha}\otimes1_{\beta^\im}-\!1_{\alpha^\im}\otimes1_{\beta}$)}\end{tabular}  \\ \hline
				
				$\xr\xr\xh$ & \makecell{$\xymatrix@C=4ex@R=-1ex{\xr_u\ar@<.4ex>[r]^{\xr_\beta,\,\xr_{\beta^\im}}\ar@<-.4ex>[r] & \xr_v \ar[r]^{\xh_\alpha} & \xh_w}$}  & $1_{\alpha}\otimes1_{\beta}+\jm_{\alpha}\otimes1_{\beta^\im}$ & $1_{\alpha}\otimes1_{\beta}-\jm_{\alpha}\otimes1_{\beta^\im}$ \\\hline

				$\xr\xr\xc$ & \makecell{$\xymatrix@C=4ex@R=-1ex{\xr_u\ar@<.4ex>[r]^{\xr_\beta,\,\xr_{\beta^\im}}\ar@<-.4ex>[r] & \xr_v \ar[r]^{\xc_\alpha} & \xc_w}$} & $1_{\alpha}\otimes1_{\beta}+\im_{\alpha}\otimes1_{\beta^\im}$ & $1_{\alpha}\otimes1_{\beta}-\im_{\alpha}\otimes1_{\beta^\im}$  \\ \hline
				
				$\xh\xh\xr$ & \makecell{$\xymatrix@C=4ex@R=-1ex{\xh_u\ar@<.4ex>[r]^{\xh_\beta,\,\xh_{\beta^\im}}\ar@<-.4ex>[r] & \xh_v \ar[r]^{\xh_\alpha} & \xr_w}$} & $1_{\alpha}\otimes1_{\beta}+1_{\alpha}\otimes\jm_{\beta^\im}$ & $1_{\alpha}\otimes\km_{\beta}-1_{\alpha}\otimes\lm_{\beta^\im}$ \\ \hline
				
				$\xh\xh\xh$ & \makecell{$\xymatrix@C=4ex@R=-1ex{\xh_u\ar@<.4ex>[r]^{\xh_\beta,\,\xh_{\beta^\im}}\ar@<-.4ex>[r] & \xh_v \ar@<.4ex>[r]^{\xh_\alpha,\,\xh_{\alpha^\im}}\ar@<-0.4ex>[r] & \xh_w}$} & \begin{tabular}[c]{@{}l@{}}$1_{\alpha}\otimes1_{\beta}-1_{\alpha^\im}\otimes1_{\beta^\im}$\\$\quad+1_{\alpha}\otimes\jm_{\beta^\im}+1_{\alpha^\im}\otimes\jm_{\beta}$\end{tabular} & \begin{tabular}[c]{@{}l@{}}$1_{\alpha}\otimes\km_{\beta}+1_{\alpha^\im}\otimes\km_{\beta^\im}$\\$\quad-1_{\alpha}\otimes\lm_{\beta^\im}+1_{\alpha^\im}\otimes\lm_{\beta}$\end{tabular} \\ \hline
				
				$\xh\xh\xc$ & \makecell{$\xymatrix@C=4ex@R=-1ex{\xh_u\ar@<.4ex>[r]^{\xh_\beta,\,\xh_{\beta^\im}}\ar@<-.4ex>[r] & \xh_v \ar[r]^{\xc^2_\alpha} & \xc_w}$} & \begin{tabular}[c]{@{}l@{}}$ \tiny{\begin{bmatrix}1&0\end{bmatrix}}_{\alpha}\otimes1_{\beta}$\\ \qquad $+\tiny{\begin{bmatrix}1&0\end{bmatrix}}_{\alpha}\otimes\jm_{\beta^\im}$\end{tabular} &  \begin{tabular}[c]{@{}l@{}}$\tiny{\begin{bmatrix}1&0\end{bmatrix}}_{\alpha}\otimes\km_{\beta}$\\ \qquad$-\tiny{\begin{bmatrix}1&0\end{bmatrix}}_{\alpha}\otimes\lm_{\beta^\im}$\end{tabular}
				\\ \hline
				
				$\xh\xr\xr$ & \makecell{$\xymatrix@C=4ex@R=-1ex{\xh_u\ar[r]^{\xh_\beta} & \xr_v \ar@<.4ex>[r]^{\xr_\alpha,\,\xr_{\alpha^\im}}\ar@<-0.4ex>[r] & \xr_w}$} & $1_{\alpha}\otimes1_{\beta}+1_{\alpha^\im}\otimes\jm_{\beta}$ &  $1_{\alpha}\otimes1_{\beta}-1_{\alpha^\im}\otimes\jm_{\beta}$    \\ \hline
				
				$\xc\xr\xr$ & \makecell{$\xymatrix@C=4ex@R=-1ex{\xc_u\ar[r]^{\xc_\beta} & \xr_v \ar@<.4ex>[r]^{\xr_\alpha,\,\xr_{\alpha^\im}}\ar@<-0.4ex>[r] & \xr_w}$} & $1_{\alpha}\otimes1_{\beta}+1_{\alpha^\im}\otimes\im_{\beta}$ &  $1_{\alpha}\otimes1_{\beta}-1_{\alpha^\im}\otimes\im_{\beta}$    \\ \hline
				
				$\xr\xh\xh$ & \makecell{$\xymatrix@C=4ex@R=-1ex{\xr_u\ar[r]^{\xh_\beta} & \xh_v \ar@<.4ex>[r]^{\xh_\alpha,\,\xh_{\alpha^\im}}\ar@<-0.4ex>[r] & \xh_w}$} & $1_{\alpha}\otimes1_{\beta}+1_{\alpha^\im}\otimes\jm_{\beta}$ &  $1_{\alpha}\otimes\km_{\beta}+1_{\alpha^\im}\otimes\lm_{\beta}$  \\ \hline
				
				$\xc\xh\xh$ & \makecell{$\xymatrix@C=4ex@R=-1ex{\xc_u\ar[r]^{\xc^2_\beta} & \xh_v \ar@<.4ex>[r]^{\xh_\alpha,\,\xh_{\alpha^\im}}\ar@<-0.4ex>[r] & \xh_w}$} & $1_{\alpha}\otimes\tiny{\begin{bmatrix}1\\0\end{bmatrix}}_{\beta}+\jm_{\alpha^\im}\otimes\tiny{\begin{bmatrix}1\\0\end{bmatrix}}_{\beta}$ &  $\km_{\alpha}\otimes\tiny{\begin{bmatrix}1\\0\end{bmatrix}}_{\beta}+\lm_{\alpha^\im}\otimes\tiny{\begin{bmatrix}1\\0\end{bmatrix}}_{\beta}$
				\\ \hline
				
				$\xh\xr\xh$ & \makecell{$\xymatrix@C=4ex@R=-1ex{\xh_u\ar@<-.4ex>[r]^{\xh_\beta} & \xr_v \ar@<-.4ex>[r]^{\xh_\alpha} & \xh_w}$} & $1_{\alpha}\otimes1_{\beta}-\jm_{\alpha}\otimes\jm_{\beta}$ &  $1_{\alpha}\otimes1_{\beta}+\jm_{\alpha}\otimes\jm_{\beta}$  
				\\  \hline
				
				$\xh\xr\xc$ & \makecell{$\xymatrix@C=4ex@R=-1ex{\xh_u\ar@<-.4ex>[r]^{\xh_\beta} & \xr_v \ar@<-.4ex>[r]^{\xc_\alpha} & \xc_w}$}  & $1_{\alpha}\otimes1_{\beta}-\im_{\alpha}\otimes\jm_{\beta}$ &  $1_{\alpha}\otimes1_{\beta}+\im_{\alpha}\otimes\jm_{\beta}$ 
				\\ \hline
				
				$\xc\xr\xh$ & \makecell{$\xymatrix@C=4ex@R=-1ex{\xc_u\ar@<-.4ex>[r]^{\xc_\beta} & \xr_v \ar@<-.4ex>[r]^{\xh_\alpha} & \xh_w}$}  & $1_{\alpha}\otimes1_{\beta}-\jm_{\alpha}\otimes\im_{\beta}$ &  $1_{\alpha}\otimes1_{\beta}+\jm_{\alpha}\otimes\im_{\beta}$
				\\ \hline
				
				$\xc\xr\xc$ & \makecell{$\xymatrix@C=4ex@R=-1ex{\xc_u\ar@<-.4ex>[r]^{\xc_\beta} & \xr_v \ar@<-.4ex>[r]^{\xc_\alpha} & \xc_w}$}  & $1_{\alpha}\otimes1_{\beta}-\im_{\alpha}\otimes\im_{\beta}$ &  $1_{\alpha}\otimes1_{\beta}+\im_{\alpha}\otimes\im_{\beta}$
				\\ \hline
				
				$\xr\xh\xc$ & \makecell{$\xymatrix@C=4ex@R=-1ex{\xr_u\ar@<-.4ex>[r]^{\xh_\beta} & \xh_v \ar@<-.4ex>[r]^{\xc^2_\alpha} & \xc_w}$}  & $\tiny{\begin{bmatrix}1&0\end{bmatrix}}_{\alpha}\otimes1_{\beta}$ &  $\tiny{\begin{bmatrix}1&0\end{bmatrix}}_{\alpha}\otimes\km_{\beta}$
				\\ \hline
				
				$\xc\xh\xr$ & \makecell{$\xymatrix@C=4ex@R=-1ex{\xc_u\ar@<-.4ex>[r]^{\xc^2_\beta} & \xh_v \ar@<-.4ex>[r]^{\xh_\alpha} & \xr_w}$}  &$1_{\alpha}\otimes\tiny{\begin{bmatrix}1\\0\end{bmatrix}}_{\beta}$ &  $-\km_{\alpha}\otimes\tiny{\begin{bmatrix}1\\0\end{bmatrix}}_{\beta}$
				\\ \hline
				
				$\xc\xh\xc$ & \makecell{$\xymatrix@C=4ex@R=-1ex{\xc_u\ar@<-.4ex>[r]^{\xc^2_\beta} & \xh_v \ar@<-.4ex>[r]^{\xc^2_\alpha} & \xc_w}$}  & $\tiny{\begin{bmatrix}1&0\end{bmatrix}}_{\alpha}\otimes\tiny{\begin{bmatrix}1\\0\end{bmatrix}}_{\beta}$ &  $\tiny{\begin{bmatrix}1&0\end{bmatrix}}_{\alpha}\otimes\tiny{\begin{bmatrix}0\\1\end{bmatrix}}_{\beta}$
				\\ \hline
				
				$\xr\xh\xr$ & \makecell{$\xymatrix@C=4ex@R=-1ex{\xr_u\ar@<-.4ex>[r]^{\xh_\beta} & \xh_v \ar@<-.4ex>[r]^{\xh_\alpha} & \xr_w}$}  & $1_{\alpha}\otimes1_{\beta}$ ($1_{\alpha}\otimes\jm_{\beta}$) & $1_{\alpha}\otimes\km_{\beta}$ ($1_{\alpha}\otimes\lm_{\beta}$)
				\\ \hline
		\end{tabular}} 
		\caption{Local modulated quivers for specially gentle vertex}\label{spgenv}
	\end{table}

	Denote by $I_{v}^0$ the $\mathcal{M}(w)$-$\mathcal{M}(u)$-bimodule generated by $r_0$ (by $r_0$ and $r_0^\im$ for cases $\xr\xr\xr$ and $\xr\xh\xr$), and by $I_{v}^1$ denote the $\mathcal{M}(w)$-$\mathcal{M}(u)$-bimodule generated by $r_1$ (by $r_1$ and $r_1^\im$ for cases $\xr\xr\xr$ and $\xr\xh\xr$). Then one can check (or by Remark \ref{4.1=3.1}) that for each vertex $v$ above, $$\underset{\alpha\in v^+, \beta\in v^-}{\bigoplus}\mathcal{M}(\alpha\beta)=I_{v}^0\oplus I_{v}^1.$$

	\begin{rem}
		Notice that for the above notation, we do not assume that $u$, $v$ and $w$ are pairwise different. If $N_v=\{v\}$, we identify $\alpha$ with $\beta$, and $\alpha^\im$ with $\beta^\im$. 
	\end{rem}
	
	\begin{defn}
		A vertex $v$ in $T(Q,\mathcal{M})/I$ is called \textbf{specially gentle} if its modulation is $\xr$ or $\xh$ and satisfies one of the following conditions:
		\begin{enumerate}
			\item  vertices $v$, $N_v$ and arrows $v^-$, $v^+$ with their modulations form one of the 18 cases in Table \ref{spgenv} and the ideal $I$ includes $I_v^0$ xor $I_v^1$;
			\item vertices $v$, $N_v$ and arrows $v^-$, $v^+$ with their modulations form one of the following degenerated cases in Table \ref{spgenv} (in these cases, $v$ is a source or sink in $Q$):
			$$\xr_v,\;\xr_v\underset{\xr}{\overset{\xr}{\rightrightarrows}}\xr,\;\xr\underset{\xr}{\overset{\xr}{\rightrightarrows}}\xr_v,\;\xr_v\overset{\xh}{\rightarrow}\xh,\;\xh\overset{\xh}{\rightarrow}\xr_v,\;\xr_v\overset{\xc}{\rightarrow}\xc,\;\xc\overset{\xc}{\rightarrow}\xr_v,$$
			$$\xh_v,\;\xh_v\overset{\xh}{\rightarrow}\xr,\;\xr\overset{\xh}{\rightarrow}\xh_v,\;\xh_v\underset{\xh}{\overset{\xh}{\rightrightarrows}}\xh,\;\xh\underset{\xh}{\overset{\xh}{\rightrightarrows}}\xh_v,\;\xh_v\overset{\xc^2}{\rightarrow}\xc,\;\xc\overset{\xc^2}{\rightarrow}\xh_v.$$
		\end{enumerate}
		
	\end{defn}

	The definition of ordinarily gentle vertices is an analogy of condition (G1) and (G2) in Definition \ref{usual defn}. The extra condition (G3) for ordinarily gentle vertices with $\xc$ is to avoid communicative relations; see Example \ref{skgen exm}. The 18 cases in Table \ref{spgenv} for specially gentle vertices can be considered as some special locally gentle parts in $T(Q,\mathcal{M})/I$.  Notice that a vertex can be ordinarily and specially gentle at the same time, which happens only if it is a source or sink in $Q$; see Example~\ref{unimform diff from special} for more information.
	
	\begin{rem}
		We call $v$ a \textbf{gentle} vertex in $T(Q,\mathcal{M})/I$, if $v$ is ordinarily or specially gentle in $T(Q,\mathcal{M})/I$. In this case, we set $$I_v:=\underset{\alpha\in v^+, \beta\in v^-}{\bigoplus}\mathcal{M}(\alpha\beta)\cap I,$$ which is a subbimodule (may be $0$) of $\underset{\alpha\in v^+, \beta\in v^-}{\bigoplus}\mathcal{M}(\alpha\beta)$. Here, if $v$ is a source or sink in $Q$, we set $\underset{\alpha\in v^+, \beta\in v^-}{\bigoplus}\mathcal{M}(\alpha\beta)=0$.
	\end{rem}
	
	\medskip
	
	Let $\xc\Gamma/J$ be the complexified quiver presentation of $T(Q,\mathcal{M})/I$ and $v$ be a gentle vertex in $T(Q,\mathcal{M})/I$. We show that there is an isomorphism of $\xc$-algebras $\xc\Gamma/J\simeq \xc\Gamma/J_v$ such that $v$ and $\overline{v}$ (if exists) are gentle in $\xc\Gamma/J_v$. To this end, we compute certain images in $\Psi(\mathbf{e}(I_v^0\otimes_\xr\xc)\mathbf{e})$ and $\Psi(\mathbf{e}(I_v^1\otimes_\xr\xc)\mathbf{e})$ for a specially gentle vertex $v$, and list them in Table~\ref{psi(ab)}. We compute the case ($\xr\xh\xr$) as an example in the proof of the following proposition.
	
	\begin{table}
		
		\begin{tabular}{l|l|l}
			\cline{1-3}
			\small{cases \begin{tabular}[c]{c@{}c@{}}
					complexified\\ quivers
			\end{tabular}}  & \small{\begin{tabular}[c]{@{}l@{}}
					elements in $\Psi(\mathbf{e}(I_v^0\otimes\xc)\mathbf{e})$
			\end{tabular}}   &\small{\begin{tabular}[c]{@{}l@{}}
					elements in $\Psi(\mathbf{e}(I_v^1\otimes\xc)\mathbf{e})$
			\end{tabular}} \\
			\hline
			
			\small{$\xr\xr\xr$ \makecell{$\xymatrix@C=5ex@R=-1ex{u\ar@<.5ex>[r]^{\beta}\ar@<-0.5ex>[r]_{\beta^\im} & v \ar@<.5ex>[r]^{\alpha}\ar@<-0.5ex>[r]_{\alpha^\im} & w}$}} & \small{\begin{tabular}[c]{@{}l@{}}
					$\quad\Psi(\mathbf{e}(r_0\otimes\!1-r_0^\im\otimes\im)\mathbf{e})$\\$=(\alpha+\im\alpha^\im)(\beta+\im\beta^\im)$\\$\quad\Psi(\mathbf{e}(r_0\otimes\!1+r_0^\im\otimes\im)\mathbf{e})$\\$=(\alpha-\im\alpha^\im)(\beta-\im\beta^\im)$
			\end{tabular}} & \small{\begin{tabular}[c]{@{}l@{}}
					$\quad\Psi(\mathbf{e}(r_1\otimes\!1-r_1^\im\otimes\im)\mathbf{e})$\\$=(\alpha+\im\alpha^\im)(\beta-\im\beta^\im)$\\$\quad\Psi(\mathbf{e}(r_1\otimes\!1+r_1^\im\otimes\im)\mathbf{e})$\\$=(\alpha-\im\alpha^\im)(\beta+\im\beta^\im)$
			\end{tabular}} \\  \hline
			
			\small{$\xr\xr\xh$ \makecell{$\xymatrix@C=5ex@R=-1ex{u\ar@<.5ex>[r]^{\beta}\ar@<-0.5ex>[r]_{\beta^\im} & v \ar@<.5ex>[r]^{\alpha}\ar@<-0.5ex>[r]_{\overline{\alpha}} & w}$}} & \small{\begin{tabular}[c]{@{}l@{}}
					$\Psi(\mathbf{e}(r_0\otimes\!1)\mathbf{e})=\alpha(\beta+\im\beta^\im)$\\$\Psi(\mathbf{e}(\km r_0\otimes\!1)\mathbf{e})=\overline{\alpha}(\beta-\im\beta^\im)$
			\end{tabular}} & \small{\begin{tabular}[c]{@{}l@{}}
					$\Psi(\mathbf{e}(r_1\otimes\!1)\mathbf{e})=\alpha(\beta-\im\beta^\im)$\\$\Psi(\mathbf{e}(\km r_1\otimes\!1)\mathbf{e})=\overline{\alpha}(\beta+\im\beta^\im)$
			\end{tabular}} \\ \hline

			\small{$\xr\xr\xc$ \makecell{$\xymatrix@C=5ex@R=-1ex{ & & w\\ u\ar@<.5ex>[r]^{\beta}\ar@<-0.5ex>[r]_{\beta^\im}& v \ar[ur]^{\alpha}\ar[dr]_{\overline{\alpha}} & \\ & &\overline{w}}$}} & \small{\begin{tabular}[c]{@{}l@{}}
					$\!e_w\Psi(\mathbf{e}(r_0\otimes\!1)\mathbf{e})\!=\!\alpha(\beta+\im\beta^\im)\!$\\$\!e_{\overline{w}}\Psi(\mathbf{e}( r_0\otimes\!1)\mathbf{e})\!=\!\overline{\alpha}(\beta-\im\beta^\im)\!$
			\end{tabular}} & \small{\begin{tabular}[c]{@{}l@{}}
					$\!e_w\Psi(\mathbf{e}(r_1\otimes\!1)\mathbf{e})\!=\!\alpha(\beta-\im\beta^\im)\!$\\$\!e_{\overline{w}}\Psi(\mathbf{e}( r_1\otimes\!1)\mathbf{e})\!=\!\overline{\alpha}(\beta+\im\beta^\im)\!$
			\end{tabular}} \\ \hline
			
			\small{$\xh\xh\xr$ \makecell{$\xymatrix@C=5ex@R=-1ex{u\ar@<.5ex>[r]^{\beta}\ar@<-0.5ex>[r]_{\beta^\im} & v \ar@<.5ex>[r]^{\alpha}\ar@<-0.5ex>[r]_{\overline{\alpha}} & w}$}} & \small{\begin{tabular}[c]{@{}l@{}}
					$\Psi(\mathbf{e}(r_0\otimes\!1)\mathbf{e})\!=\!\alpha(\beta+\im\beta^\im)\!$\\$\!-\Psi(\mathbf{e}( r_0\km\otimes\!1)\mathbf{e})\!=\!\overline{\alpha}(\beta-\im\beta^\im)\!$
			\end{tabular}} & \small{\begin{tabular}[c]{@{}l@{}}
					$\!-\Psi(\mathbf{e}(r_1\km\otimes\!1)\mathbf{e})\!=\!\alpha(\beta-\im\beta^\im)\!$\\$\!-\Psi(\mathbf{e}(r_1\otimes\!1)\mathbf{e})\!=\!\overline{\alpha}(\beta+\im\beta^\im)\!$
			\end{tabular}}\\ \hline
			
			\small{$\xh\xh\xh$ \makecell{$\xymatrix@C=5ex@R=-1ex{u\ar@<.5ex>[r]^{\beta}\ar@<-0.5ex>[r]_{\beta^\im} & v \ar@<.5ex>[r]^{\alpha}\ar@<-0.5ex>[r]_{\alpha^\im} & w}$}} & \small{\begin{tabular}[c]{@{}l@{}}
					$\quad\Psi(\mathbf{e}(r_0\otimes\!1)\mathbf{e})\!$\\$=\!(\alpha+\im\alpha^\im)(\beta+\im\beta^\im)\!$\\$\quad-\Psi(\mathbf{e}( \km r_0\km\otimes\!1)\mathbf{e})\!$\\$=\!(\alpha-\im\alpha^\im)(\beta-\im\beta^\im)\!$
			\end{tabular}} & \small{\begin{tabular}[c]{@{}l@{}}
					$\quad\Psi(\mathbf{e}(r_1\otimes\!1)\mathbf{e})\!$\\$=\!(\alpha+\im\alpha^\im)(\beta-\im\beta^\im)\!$\\$\quad-\Psi(\mathbf{e}(\km r_1\km\otimes\!1)\mathbf{e})\!$\\$=\!(\alpha-\im\alpha^\im)(\beta+\im\beta^\im)\!$
			\end{tabular}} \\ \hline
			
			\small{$\xh\xh\xc$ \makecell{$\xymatrix@C=5ex@R=-1ex{ & & w\\ u\ar@<.5ex>[r]^{\beta}\ar@<-0.5ex>[r]_{\beta^\im}& v \ar[ur]^{\alpha}\ar[dr]_{\overline{\alpha}} & \\ & &\overline{w}}$}} & \small{\begin{tabular}[c]{@{}l@{}}
					$\!\Psi(\mathbf{e}(r_0\otimes\!1)\mathbf{e})\!=\!\alpha(\beta+\im\beta^\im)\!$\\$\!-\Psi(\mathbf{e}( r_0\km\otimes\!1)\mathbf{e})\!=\!\overline{\alpha}(\beta-\im\beta^\im)\!$
			\end{tabular}} & \small{\begin{tabular}[c]{@{}l@{}}
					$\!\Psi(\mathbf{e}(r_1\otimes\!1)\mathbf{e})\!=\!\alpha(\beta-\im\beta^\im)\!$\\$\!-\Psi(\mathbf{e}( r_1\km\otimes\!1)\mathbf{e})\!=\!\overline{\alpha}(\beta+\im\beta^\im)\!$
			\end{tabular}}  \\ \hline
			
			\small{$\xh\xr\xr$ \makecell{$\xymatrix@C=5ex@R=-1ex{u\ar@<.5ex>[r]^{\beta}\ar@<-0.5ex>[r]_{\overline{\beta}} & v \ar@<.5ex>[r]^{\alpha}\ar@<-0.5ex>[r]_{\alpha^\im} & w}$}} &  \small{\begin{tabular}[c]{@{}l@{}}
					$\Psi(\mathbf{e}(r_0\otimes\!1)\mathbf{e})\!=\!(\alpha+\im\alpha^\im)\beta\!$\\$\Psi(\mathbf{e}(r_0\km\otimes\!1)\mathbf{e})\!=\!(\alpha-\im\alpha^\im)\overline{\beta}\!$
			\end{tabular}} & \small{\begin{tabular}[c]{@{}l@{}}
					$\Psi(\mathbf{e}(r_1\otimes\!1)\mathbf{e})\!=\!(\alpha-\im\alpha^\im)\beta\!$\\$\!-\Psi(\mathbf{e}(r_1\km\otimes\!1)\mathbf{e})\!=\!(\alpha+\im\alpha^\im)\overline{\beta}\!$
			\end{tabular}}  \\  \hline
			
			\small{$\xc\xr\xr$ \makecell{$\xymatrix@C=5ex@R=-1ex{u\ar[dr]^{\beta} & & \\ & v \ar@<.5ex>[r]^{\alpha}\ar@<-0.5ex>[r]_{\alpha^\im} & w\\ \overline{u}\ar[ur]_{\overline{\beta}}& &}$}} &  \small{\begin{tabular}[c]{@{}l@{}}
					$\Psi(\mathbf{e}(r_0\otimes\!1)\mathbf{e})e_u\!=\!(\alpha+\im\alpha^\im)\beta\!$\\$\Psi(\mathbf{e}(r_0\otimes\!1)\mathbf{e})e_{\overline{u}}\!=\!(\alpha-\im\alpha^\im)\overline{\beta}\!$
			\end{tabular}} &   \small{\begin{tabular}[c]{@{}l@{}}
					$\Psi(\mathbf{e}(r_1\otimes\!1)\mathbf{e})e_u\!=\!(\alpha-\im\alpha^\im)\beta\!$\\$\Psi(\mathbf{e}(r_1\otimes\!1)\mathbf{e})e_{\overline{u}}\!=\!(\alpha+\im\alpha^\im)\overline{\beta}\!$
			\end{tabular}}   \\  \hline
			
			\small{$\xr\xh\xh$ \makecell{$\xymatrix@C=5ex@R=-1ex{u\ar@<.5ex>[r]^{\beta}\ar@<-0.5ex>[r]_{\overline{\beta}} & v \ar@<.5ex>[r]^{\alpha}\ar@<-0.5ex>[r]_{\alpha^\im} & w}$}} &\small{\begin{tabular}[c]{@{}l@{}}
					$\Psi(\mathbf{e}(r_0\otimes\!1)\mathbf{e})\!=\!(\alpha+\im\alpha^\im)\beta\!$\\$\Psi(\mathbf{e}(\km r_0\otimes\!1)\mathbf{e})\!=\!(\alpha-\im\alpha^\im)\overline{\beta}\!$
			\end{tabular}} & \small{\begin{tabular}[c]{@{}l@{}}
					$\Psi(\mathbf{e}(r_1\otimes\!1)\mathbf{e})\!=\!(\alpha+\im\alpha^\im)\overline{\beta}\!$\\$\!-\Psi(\mathbf{e}(\km r_1\otimes\!1)\mathbf{e})\!=\!(\alpha-\im\alpha^\im)\beta\!$
			\end{tabular}}\\ \hline
			
			\small{$\xc\xh\xh$ \makecell{$\xymatrix@C=5ex@R=-1ex{u\ar[dr]^{\beta} & & \\ & v \ar@<.5ex>[r]^{\alpha}\ar@<-0.5ex>[r]_{\alpha^\im} & w\\ \overline{u}\ar[ur]_{\overline{\beta}}& &}$}} & \small{\begin{tabular}[c]{@{}l@{}}
					$\Psi(\mathbf{e}(r_0\otimes\!1)\mathbf{e})\!=\!(\alpha+\im\alpha^\im)\beta\!$\\$\Psi(\mathbf{e}(\km r_0\otimes\!1)\mathbf{e})\!=\!(\alpha-\im\alpha^\im)\overline{\beta}\!$
			\end{tabular}} & \small{\begin{tabular}[c]{@{}l@{}}
					$\Psi(\mathbf{e}(r_1\otimes\!1)\mathbf{e})\!=\!(\alpha+\im\alpha^\im)\overline{\beta}\!$\\$\!-\Psi(\mathbf{e}(\km r_1\otimes\!1)\mathbf{e})\!=\!(\alpha-\im\alpha^\im)\beta\!$
			\end{tabular}} \\ \hline
			
			\small{$\xh\xr\xh$ \makecell{$\xymatrix@C=5ex@R=-1ex{u\ar@<.5ex>[r]^{\beta}\ar@<-0.5ex>[r]_{\overline{\beta}} & v \ar@<.5ex>[r]^{\alpha}\ar@<-0.5ex>[r]_{\overline{\alpha}} & w}$}} &\small{\begin{tabular}[c]{@{}l@{}}
					$\Psi(\mathbf{e}(r_0\otimes\!1)\mathbf{e})/2\!=\!\alpha\beta\!$\\$\!-\Psi(\mathbf{e}(\km r_0\km\otimes\!1)\mathbf{e})/2\!=\!\overline{\alpha}\overline{\beta}\!$
			\end{tabular}} & \small{\begin{tabular}[c]{@{}l@{}}
					$\Psi(\mathbf{e}(\km r_1\otimes\!1)\mathbf{e})/2\!=\!\overline{\alpha}\beta\!$\\$\!-\Psi(\mathbf{e}(r_1\km\otimes\!1)\mathbf{e})/2\!=\!\alpha\overline{\beta}\!$
			\end{tabular}} \\ \hline
			
			\small{$\xh\xr\xc$ \makecell{$\xymatrix@C=5ex@R=-1ex{ & & w\\ u\ar@<.5ex>[r]^{\beta}\ar@<-0.5ex>[r]_{\overline{\beta}}& v \ar[ur]^{\alpha}\ar[dr]_{\overline{\alpha}} & \\ & &\overline{w}}$}}  & \small{\begin{tabular}[c]{@{}l@{}}
					$\Psi(\mathbf{e}(r_0\otimes\!1)\mathbf{e})/2\!=\!\alpha\beta\!$\\$\!-\Psi(\mathbf{e}(r_0\km\otimes\!1)\mathbf{e})/2\!=\!\overline{\alpha}\overline{\beta}\!$
			\end{tabular}} & \small{\begin{tabular}[c]{@{}l@{}}
					$\Psi(\mathbf{e}(r_1\otimes\!1)\mathbf{e})/2\!=\!\overline{\alpha}\beta\!$\\$\!-\Psi(\mathbf{e}(r_1\km\otimes\!1)\mathbf{e})/2\!=\!\alpha\overline{\beta}\!$
			\end{tabular}} \\  \hline
			
			\small{$\xc\xr\xh$ \makecell{$\xymatrix@C=5ex@R=-1ex{u\ar[dr]^{\beta} & & \\ & v \ar@<.5ex>[r]^{\alpha}\ar@<-0.5ex>[r]_{\overline{\alpha}} & w\\ \overline{u}\ar[ur]_{\overline{\beta}}& &}$}}  & \small{\begin{tabular}[c]{@{}l@{}}
					$\Psi(\mathbf{e}(r_0\otimes\!1)\mathbf{e})/2\!=\!\alpha\beta\!$\\$\Psi(\mathbf{e}(\km r_0\otimes\!1)\mathbf{e})/2\!=\!\overline{\alpha}\overline{\beta}\!$
			\end{tabular}} & \small{\begin{tabular}[c]{@{}l@{}}
					$\Psi(\mathbf{e}(r_1\otimes\!1)\mathbf{e})/2\!=\!\alpha\overline{\beta}\!$\\$\Psi(\mathbf{e}(\km r_1\otimes\!1)\mathbf{e})/2\!=\!\overline{\alpha}\beta\!$
			\end{tabular}}\\ \hline
			
			\small{$\xc\xr\xc$ \makecell{$\xymatrix@C=5ex@R=-1ex{u\ar[dr]^{\beta} & & w\\ & v \ar[ur]^{\alpha}\ar[dr]_{\overline{\alpha}} & \\ \overline{u}\ar[ur]_{\overline{\beta}}& &\overline{w}}$}}  & \small{\begin{tabular}[c]{@{}l@{}}
					$\!e_w\Psi(\mathbf{e}(r_0\otimes\!1)\mathbf{e})e_u/2\!=\!\alpha\beta\!$\\$\!e_{\overline{w}}\Psi(\mathbf{e}(r_0\otimes\!1)\mathbf{e})e_{\overline{u}}/2\!=\!\overline{\alpha}\overline{\beta}\!$
			\end{tabular}} & \small{\begin{tabular}[c]{@{}l@{}}
					$\!e_w\Psi(\mathbf{e}(r_1\otimes\!1)\mathbf{e})e_{\overline{u}}/2\!=\!\alpha\overline{\beta}\!$\\$\!e_{\overline{w}}\Psi(\mathbf{e}(r_1\otimes\!1)\mathbf{e})e_u/2\!=\!\overline{\alpha}\beta\!$
			\end{tabular}} \\  \hline
			
			\small{$\xr\xh\xc$ \makecell{$\xymatrix@C=5ex@R=-1ex{ & & w\\ u\ar@<.5ex>[r]^{\beta}\ar@<-0.5ex>[r]_{\overline{\beta}}& v \ar[ur]^{\alpha}\ar[dr]_{\overline{\alpha}} & \\ & &\overline{w}}$}}  & \small{\begin{tabular}[c]{@{}l@{}}
					$\!e_w\Psi(\mathbf{e}(r_0\otimes\!1)\mathbf{e})\!=\!\alpha\beta\!$\\$\!e_{\overline{w}}\Psi(\mathbf{e}(r_0\otimes\!1)\mathbf{e})\!=\!\overline{\alpha}\overline{\beta}\!$
			\end{tabular}} & \small{\begin{tabular}[c]{@{}l@{}}
					$\!e_w\Psi(\mathbf{e}(r_1\otimes\!1)\mathbf{e})\!=\!\alpha\overline{\beta}\!$\\$\!-e_{\overline{w}}\Psi(\mathbf{e}(r_1\otimes\!1)\mathbf{e})\!=\!\overline{\alpha}\beta\!$
			\end{tabular}} \\  \hline
			
			\small{$\xc\xh\xr$ \makecell{$\xymatrix@C=5ex@R=-1ex{u\ar[dr]^{\beta} & & \\ & v \ar@<.5ex>[r]^{\alpha}\ar@<-0.5ex>[r]_{\overline{\alpha}} & w\\ \overline{u}\ar[ur]_{\overline{\beta}}& &}$}}& \small{\begin{tabular}[c]{@{}l@{}}
					$\!\Psi(\mathbf{e}(r_0\otimes\!1)\mathbf{e})e_u\!=\!\alpha\beta\!$\\$\!\Psi(\mathbf{e}(r_0\otimes\!1)\mathbf{e})e_{\overline{u}}\!=\!\overline{\alpha}\overline{\beta}\!$
			\end{tabular}} & \small{\begin{tabular}[c]{@{}l@{}}
					$\!-\Psi(\mathbf{e}(r_1\otimes\!1)\mathbf{e})e_u\!=\!\alpha\overline{\beta}\!$\\$\!\Psi(\mathbf{e}(r_1\otimes\!1)\mathbf{e})e_{\overline{u}}\!=\!\overline{\alpha}\beta\!$
			\end{tabular}}\\ \hline
			
			\small{$\xc\xh\xc$ \makecell{$\xymatrix@C=5ex@R=-1ex{u\ar[dr]^{\beta} & & w\\ & v \ar[ur]^{\alpha}\ar[dr]_{\overline{\alpha}} & \\ \overline{u}\ar[ur]_{\overline{\beta}}& &\overline{w}}$}}  & \small{\begin{tabular}[c]{@{}l@{}}
					$\!e_w\Psi(\mathbf{e}(r_0\otimes\!1)\mathbf{e})e_u\!=\!\alpha\beta\!$\\$\!e_{\overline{w}}\Psi(\mathbf{e}(r_0\otimes\!1)\mathbf{e})e_{\overline{u}}\!=\!\overline{\alpha}\overline{\beta}\!$
			\end{tabular}} & \small{\begin{tabular}[c]{@{}l@{}}
					$\!e_w\Psi(\mathbf{e}(r_1\otimes\!1)\mathbf{e})e_{\overline{u}}\!=\!\alpha\overline{\beta}\!$\\$\!e_{\overline{w}}\Psi(\mathbf{e}(r_1\otimes\!1)\mathbf{e})e_u\!=\!\overline{\alpha}\beta\!$
			\end{tabular}} \\  \hline
			
			\small{$\xr\xh\xr$ \makecell{$\xymatrix@C=5ex@R=-1ex{u\ar@<.5ex>[r]^{\beta}\ar@<-0.5ex>[r]_{\overline{\beta}} & v \ar@<.5ex>[r]^{\alpha}\ar@<-0.5ex>[r]_{\overline{\alpha}} & w}$}}  & \small{\begin{tabular}[c]{@{}l@{}}
					$\!\Psi(\mathbf{e}(r_0\otimes\!1-r_0^\im\otimes\im)\mathbf{e})/2\!=\!\alpha\beta\!$\\$\!\Psi(\mathbf{e}(r_0\otimes\!1+r_0^\im\otimes\im)\mathbf{e})/2\!=\!\overline{\alpha}\overline{\beta}\!$
			\end{tabular}} & \small{\begin{tabular}[c]{@{}l@{}}
					$\!\Psi(\mathbf{e}(r_1\otimes\!1-r_1^\im\otimes\im)\mathbf{e})/2\!=\!\alpha\overline{\beta}\!$\\$\!\Psi(\mathbf{e}(r_1\otimes\!1+r_1^\im\otimes\im)\mathbf{e})/2\!=\!\overline{\alpha}\beta\!$
			\end{tabular}} \\  \hline
		\end{tabular}\medskip
		\caption{Local complexified quivers and elements in relations}
		\label{psi(ab)}
	\end{table}

	\begin{prop}\label{gentle-ver}
		Let $v$ be a gentle vertex in $T(Q,\mathcal{M})/I$ and $\xc\Gamma/J$ be its complexified quiver presentation. Then there is a $\xc$-algebra isomorphism $\xc\Gamma/J\simeq\xc\Gamma/J_v$ such that $v$ and $\overline{v}$ (if exists) are gentle vertices in $\xc\Gamma/J_v$.
	\end{prop}
	\begin{proof}
		If $v$ is gentle in $T(Q,\mathcal{M})/I$, we give an isomorphism $d_v\colon\xc\Gamma\rightarrow\xc\Gamma$.
		
		\begin{enumerate}
			\item If $v$ is ordinarily gentle or specially gentle in the last 8 cases in Table~\ref{psi(ab)}, set $d_v=\mathrm{Id}$;
			\item If $v$ is specially gentle in the first 10 cases in Table \ref{psi(ab)} (i.e., $v^+=\{\gamma,\gamma^\im\}$ or $v^-=\{\gamma,\gamma^\im\}$ for some parallel arrows $\gamma,\gamma^\im$), $d_v$ is given as follows.\\
			For arrows $\gamma,\gamma^\im$ such that $v^+=\{\gamma,\gamma^\im\}$ or $v^-=\{\gamma,\gamma^\im\}$, set $$d_v(\gamma)=(\gamma+\gamma^\im)/2\mbox{ and }d_v(\gamma^\im)=(\gamma-\gamma^\im)/2\im.$$ $$(\mbox{So   }d_v(\gamma+\im\gamma^\im)=\gamma\mbox{ and }
			d_v(\gamma-\im\gamma^\im)=\gamma^\im.)$$ 
			For other vertices and arrows, set $d_v$ as identity.
		\end{enumerate}
		
		To prove the statement, we set $$J_v:=d_v(J).$$
		
		Assume that $v\in Q_0$ is an ordinarily gentle vertex with $\xr$ or $\xh$. Then $\mathcal{M}(u)=\mathcal{M}(v)=\xr \mbox{ or }\xh\mbox{, }\forall u\in N_v$. The vertex $v\in Q_0$ has only one fiber $v$ in $\Gamma$ and each path $u\overset{\beta}{\rightarrow}v\overset{\alpha}{\rightarrow}w$ in $Q$ also has only one fiber $u\overset{\beta}{\rightarrow}v\overset{\alpha}{\rightarrow}w$ in $\Gamma$. Hence $|v^+|\leq 2$ and $|v^-|\leq 2$ in $\Gamma$. By Proposition \ref{ios after complexification}, $I$ includes $\mathcal{M}(\alpha\beta)$ if and only if $\alpha\beta\in J$. Then condition (G2) for $v$ being gentle in $T(Q,\mathcal{M})/I$ implies condition (G2) for $v$ in $\xc\Gamma/J_v=\xc\Gamma/J$. 
		
		Assume that $v\in Q_0$ is ordinarily gentle with $\xc$. By the construction of $\Gamma$, there are two fibers $v$ and $\overline{v}$ in $\Gamma$. Each arrow $\alpha\in Q_1$ starting at $v$ gives two fibers $\alpha$ and $\overline{\alpha}$ in $\Gamma$ with one stating at $v$ and the other at $\overline{v}$. Each arrow $\beta\in Q_1$ ending at $v$ gives two fibers $\beta$ and $\overline{\beta}$ in $\Gamma$ with one ending at $v$ and the other at $\overline{v}$. We have $|v^+|=|\overline{v}^+|$ (in $\Gamma$)$=|v^+|$ (in $Q$) and $|v^-|=|\overline{v}^-|$ (in $\Gamma$) $=|v^-|$ (in $Q$). Hence in $\Gamma$, $|v^+|=|\overline{v}^+|\leq 2$ and $|v^-|=|\overline{v}^-|\leq 2$.
		
		As above, for each $u\overset{\beta}{\rightarrow}v\overset{\alpha}{\rightarrow}w$ in $Q$, there are exactly two fibers in $\Gamma$, which can be denoted by $p$ (the one passing $v$) and $\tau(p)$ (the one passing $\overline{v}$). By \cite[Remark 3.6]{LJ2023} and \cite[Lemma 4.6]{LJ2023}, $I$ includes $\mathcal{M}(\alpha\beta)$ if and only if $p\in J$ if and only if $\tau(p)\in J$. Then condition (G2) for $v$ being gentle in $T(Q,\mathcal{M})/I$ implies condition (G2) for $v$ and $\overline{v}$ in $\xc\Gamma/J_v=\xc\Gamma/J$.
		
		Assume that $v\in Q_0$ is a specially gentle vertex satisfying condition (2). So $v$ is a source or sink in both $Q$ and $\Gamma$ ($\overline{v}$ does not exist). We only need to confirm that $v$ satisfies (G1) in $\xc\Gamma/J_v=\xc\Gamma/J$. This can be checked case by case in Table~\ref{psi(ab)}. 
		
		Assume that $v\in Q_0$ is a specially gentle vertex satisfying condition (1). It gives only one fiber $v$ in $\Gamma$. And $|v^+|=|v^-|=2$ as shown in Table \ref{psi(ab)}. Hence, (G1) holds for $v$. In $\xc\Gamma$, $$\underset{\alpha\in v^+, \beta\in v^-}{\bigoplus}\xc\alpha\beta=\Psi(\mathbf{e}(I_v^0\otimes_\xr\xc)\mathbf{e})\oplus\Psi(\mathbf{e}(I_v^1\otimes_\xr\xc)\mathbf{e}).$$ By the results in Table \ref{psi(ab)} and the isomorphism $d_v$, we have $I_v^0\subseteq I$ if and only if $\alpha\beta\in J_v$ if and only if $\overline{\alpha}\overline{\beta}$, $\alpha^\im\overline{\beta}$, $\overline{\alpha}\beta^\im$ or $\alpha^\im\beta^\im\in J_v$, according to the case in which $v$ is. Also, $I_v^1\subseteq I$ if and only if $\alpha^\im\beta$ or
		$\overline{\alpha}\beta\in J_v$ if and only if $\alpha\overline{\beta}$ or $\alpha\beta^\im\in J_v$, according to the case in which $v$ is. Therefore, $v$ satisfies (G2) in $\xc\Gamma/J_v$.
		
		Finally, to obtain the results in Table \ref{psi(ab)}, we perform the computation for the case $(\xr\xh\xr)$ as an example. Other cases can also be checked. Recall from Section~2.3 the isomorphism $\Psi\colon \mathbf{e}(T(Q, \mathcal{M})\otimes_{\xr}\xc)\mathbf{e}\longrightarrow\xc \Gamma$ used here.
		
		In the case $(\xr\xh\xr)$, we have
		\begin{align*}
			&\mathbf{e}((\begin{bmatrix}1&0\\0&1\end{bmatrix}\otimes\begin{bmatrix}a&b\\-\overline{b}&\overline{a}\end{bmatrix}\otimes1_{\xc})\mathbf{e}\overset{\psi_1}{\mapsto}(\begin{bmatrix}1&0\\0&1\end{bmatrix}\otimes1_{\xc})\otimes(\begin{bmatrix}a&b\\-\overline{b}&\overline{a}\end{bmatrix}\otimes1_{\xc})\\
			\overset{\theta}{\mapsto}&\begin{bmatrix}1&0\\0&1\end{bmatrix}\otimes\begin{bmatrix}a&b\\-\overline{b}&\overline{a}\end{bmatrix}\overset{\psi_2}{\mapsto}\begin{bmatrix}1&0\\0&0\end{bmatrix}\otimes\begin{bmatrix}a&b\\0&0\end{bmatrix}+\begin{bmatrix}0&0\\1&0\end{bmatrix}\otimes\begin{bmatrix}-\overline{b}&\overline{a}\\0&0\end{bmatrix}\\
			\overset{\psi_3}{\mapsto}&\alpha(a\beta+b\overline{\beta})+\overline{\alpha}(-\overline{b}\beta+\overline{a}\overline{\beta})=a\alpha\beta+\overline{a}\overline{\alpha}\overline{\beta}+b\alpha\overline{\beta}-\overline{b}\overline{\alpha}\beta.
		\end{align*} Hence
		$$\Psi(\mathbf{e}((1_{\xh}\otimes1_\xh)\otimes_{\xr}1_\xc)\mathbf{e})=\alpha\beta+\overline{\alpha}\overline{\beta}\mbox{ , } \Psi(\mathbf{e}((1_{\xh}\otimes \jm)\otimes_{\xr}1_\xc)\mathbf{e})=\im\alpha\beta-\im\overline{\alpha}\overline{\beta},$$
		$$\Psi(\mathbf{e}((1_{\xh}\otimes \km)\otimes_{\xr}1_\xc)\mathbf{e})=\alpha\overline{\beta}-\overline{\alpha}\beta\mbox{ , } \Psi(\mathbf{e}((1_{\xh}\otimes \lm)\otimes_{\xr}1_\xc)\mathbf{e})=\im\alpha\overline{\beta}+\im\overline{\alpha}\beta.$$ 
		Then we obtain the results in Table \ref{psi(ab)}. It is easy to see that if $I_v^0\subseteq I$, then $\alpha\beta, \overline{\alpha}\overline{\beta}\in J$; and that if $I_v^1\subseteq I$, then $\alpha\overline{\beta}, \overline{\alpha}\beta\in J$. Since $d_v(\alpha\beta)=\alpha\beta$, $v$ satisfies (G2) in $\xc\Gamma/J_v=\xc\Gamma/J$.
	\end{proof}

	\medskip
	
	\subsection{Locally complexified-gentle algebras of uniform type and of special type}
	In this subsection, we introduce two types of locally complexified-gentle algebras. 
	
	\begin{thm}\label{main}
		Let $T(Q,\mathcal{M})/I$ be an $\xr$-algebra satisfying one of the following conditions: \begin{enumerate}
			\item each vertex is ordinarily gentle with $\xr$ in $T(Q,\mathcal{M})/I$ and the      ideal $I$ is generated by $I_v, v\in Q_0$.
			\item each vertex is ordinarily gentle with $\xh$ in $T(Q,\mathcal{M})/I$ and the ideal $I$ is generated by $I_v, v\in Q_0$.
			\item each vertex is ordinarily gentle with $\xc$ or specially gentle in         $T(Q,\mathcal{M})/I$ and the ideal $I$ is generated by $I_v, v\in Q_0$.
		\end{enumerate} Then $T(Q,\mathcal{M})/I$ is locally complexified-gentle.
	\end{thm}
	\begin{proof}
		Let $\xc\Gamma/J$ be the complexified quiver presentation of $T(Q,\mathcal{M})/I$.
		
		If $T(Q,\mathcal{M})/I$ satisfies condition (1) (or (2)), that is, $\mathcal{M}(v)=\xr$ (respectively, $=\xh$) for each vertex $v$, $\mathcal{M}(\alpha)=\xr$ (respectively, $=\xh$) for each arrow $\alpha$ and $I=\langle\mathcal{M}(p_1), \mathcal{M}(p_2),\dots,\mathcal{M}(p_n)\rangle$ for some paths $p_1,p_2,\dots,p_n$ of length two in $Q$. By Proposition \ref{gentle-ver}, each vertex is a gentle vertex in the modulated quiver presentation $\xc\Gamma/J=\xc Q/J$. By Proposition \ref{ios after complexification}, $J$ is generated by $p_1, p_2, \dots, p_n $. Hence $\xc\Gamma /J$ is a locally gentle algebra.
		
		Now assume that $T(Q,\mathcal{M})/I$ satisfies condition (3). We give an isomorphism of algebras $d:\xc\Gamma\rightarrow\xc\Gamma$. For parallel arrows $\gamma,\gamma^\im$ in the first 10 cases in Table~\ref{psi(ab)} (i.e. $t(\gamma)=t(\gamma^\im)=v$ or $s(\gamma)=s(\gamma^\im)=v$ with $\mathcal{M}(s(\gamma))=\mathcal{M}(t(\gamma))$ for each specially gentle vertex $v$ in $T(Q,\mathcal{M})/I$), set $$d(\gamma+\im\gamma^\im)=\gamma \mbox{ and } d(\gamma-\im\gamma^\im)=\gamma^\im.$$ For other vertices and arrows, set $d$ as identity. This map is well-defined, since the endpoints of these parallel arrows in $Q$ are both specially gentle vertices. 
		
		The construction of $d\colon\xc\Gamma\rightarrow\xc\Gamma$ coincides with $d_v$ in the proof of Proposition~\ref{gentle-ver} when restricted to $v$, $N_v$, $v^-$ and $v^+$. So, each vertex $v$ in $\xc\Gamma/d(J)$ satisfies (G1) and (G2) of Definition~\ref{usual defn} by Proposition~\ref{gentle-ver}. 
		
		By assumption, $I$ is generated by $I_v$, $v\in Q_0$. If $v$ is specially gentle and $I^0_v\subset I$ (or $I^1_v\subset I$) then $\alpha\beta, \alpha^\circ\beta^\bullet\in d(J)$ (respectively, $\alpha^\circ\beta, \alpha\beta^\bullet\in d(J)$), where $\alpha^\circ=\overline{\alpha}$ or $\alpha^\im$, and $\beta^\bullet=\overline{\beta}$ or $\beta^\im$, depending on the situation in which $v$ is. If $v$ is ordinarily gentle with $\xc$, all fibers of $\{\alpha\beta\,|\,s(\alpha)=t(\beta)=v, \mathcal{M}(\alpha\beta)=I_v\}$ in $\Gamma$ belong to $d(J)$. Thus, $d(J)$ can be generated by these paths of length two.
		
		Hence, $T(Q,\mathcal{M})/I\otimes_{\xr}\xc$ is Morita equivalent to the gentle algebra $\xc\Gamma/d(J)$.
	\end{proof}
	
	Since Morita equivalence is closed under complexification, we introduce two types of locally complexified-gentle algebra due to the above theorem.
	
	\begin{defn}\label{two types of cmp-gentle}
		Let $A$ be an $\xr$-algebra.
		\begin{enumerate}
			\item We call $A$ \textbf{locally complexified-gentle of uniform type} (with $\xr$) if it is Morita equivalent to some $T(Q,\mathcal{M})/I$ satisfying condition (1) in Theorem~\ref{main}.
			\item We call $A$ \textbf{locally complexified-gentle of uniform type} (with $\xh$) if it is Morita equivalent to some $T(Q,\mathcal{M})/I$ satisfying condition (2) in Theorem~\ref{main}.
			\item We call $A$ \textbf{locally complexified-gentle of special type} if it is Morita equivalent to some $T(Q,\mathcal{M})/I$ satisfying condition (3) in Theorem \ref{main}.
		\end{enumerate}
		Furthermore, if $A$ is finite-dimensional, we call it complexified-gentle of uniform type (with $\xr$ or with $\xh$) or of special type, respectively.
	\end{defn}

	Now we show that locally complexified-gentle algebras of uniform type are nothing but locally gentle algebras in the usual sense. We note that an algebra is called locally gentle over a division ring $D$ if it is Morita equivalent to $DQ/I$ for some quiver $Q$ and some ideal $I$ such that (G1), (G2), and (G3) in Definition~\ref{usual defn} hold, where $DQ$ can be viewed as the tensor algebra $T(\prod_{i\in Q_0} D e_i, \bigoplus_{\alpha\in Q_1}D \alpha)$ with $de_i=e_id$ and $d\alpha=\alpha d$, $\forall d\in D, i\in Q_0, \alpha\in Q_1$.
	
	As in \cite[Section 3.1]{LJ2023}, a modulation $\mathcal{M}$ on a quiver $Q$ over $\xr$ is called \textbf{v-uniform} with $D$ if $\mathcal{M}(i)=\mathcal{M}(j)=D$, where $D$ is $\xr$, $\xh$ or $\xc$ for any $i,j\in Q_0$. 
	
	Assume that $(Q,\mathcal{M})$ is a v-uniform modulated quiver with $D=\xr$ or $\xh$, and $p_i,i=1,2,\dots,n$ are paths of length two in $Q$. By \cite[Lemma 3.2]{LJ2023}, we have $$T(Q,\mathcal{M})/I\simeq DQ/\langle R \rangle,$$ for $I=\langle\mathcal{M}(p_i)\,|\, i=1,2,\dots,n\rangle$ and $R=\{p_i\,|\,i=1,2,\dots,n\}$. We have the following description about locally complexified-gentle algebras of uniform type.
	
	\begin{prop}\label{cpm-gen of uni}
		An $\xr$-algebra is a locally complexified-gentle algebra of uniform type if and only if it is a locally gentle algebra over $D$, where $D=\xr$ or $\xh$.
	\end{prop}

	\medskip
	\subsection{Examples}
	In this subsection, we give some examples and conjecture that each complexified-gentle algebra is either of uniform type or of special type.
	
	The next example shows that a complexified-gentle algebra $T(Q,\mathcal{M})/I$ of uniform type cannot be of special type once $I\neq 0$.
	
	\begin{exm}\label{unimform diff from special}
		Assume that $Q=u\underset{\beta^\im}{\overset{\beta}{\rightrightarrows}} v\underset{\alpha^\im}{\overset{\alpha}{\rightrightarrows}} w$ and $(Q,\mathcal{M})$ is v-uniform with $\xr$. Let $$I=\langle \alpha\beta, \alpha^\im\beta^\im \rangle \mbox{ and } I'=\langle \alpha\beta\mp\alpha^\im\beta^\im, \alpha\beta^\im\pm\alpha^\im\beta \rangle$$ be two ideals of $\xr Q$. Then, via the canonical isomorphism $\xr Q\simeq T(Q,\mathcal{M})$ (see \cite[Lemma 3.2]{LJ2023}), $\xr Q/I$ is complexified-gentle of uniform type and $\xr Q/I'$ is complexified-gentle of special type. Moreover, their complexifications are both isomorphic to $\xc Q/I$ as $\xc$-algebras.
		
		Notice that $\xr Q/I$ is not Morita equivalent to $\xr Q/I'$. In fact, in $\xr Q/I$, there is an element $\lceil\alpha\rceil$ in $\lceil e_w\rceil\mathrm{rad}(\xr Q/I)\lceil e_v\rceil$ and an element $\lceil\beta\rceil$ in $\lceil e_v\rceil\mathrm{rad}(\xr Q/I)\lceil e_u\rceil$ such that $\lceil\alpha\beta\rceil=0$, where $\lceil x\rceil$ denotes $x+I$ or $x+I'$. If $\xr Q/I$ is Morita equivalent to $\xr Q/I'$,  they are isomorphic. Then there are non-zero elements in $\lceil e_w\rceil\mathrm{rad}(\xr Q/I')\lceil e_v\rceil$ and $\lceil e_v\rceil\mathrm{rad}(\xr Q/I')\lceil e_u\rceil$, say $\lceil r_1\alpha+r_2\alpha^\im\rceil$ and $\lceil r_3\beta+r_4\beta^\im\rceil$ respectively, such that $$ (r_1\alpha+r_2\alpha^\im)(r_3\beta+r_4\beta^\im)=r_5(\alpha\beta\mp\alpha^\im\beta^\im)+r_6 (\alpha\beta^\im\pm\alpha^\im\beta)$$
        in $\xr Q$ for some $r_1, \dots, r_6\in\xr$. However, each of these two equations with variables $r_1,\dots, r_6$ has no solution in $\xr$ such that $r_1r_2\neq 0$ and $r_3r_4\neq 0$.
		
		If we assume that $(Q,\mathcal{M})$ is v-uniform with $\xh$, then $\xh Q\simeq T(Q,\mathcal{M})$. Take $I$ as above and $I'=\langle \alpha\beta\mp\alpha^\im\beta^\im\pm\alpha\beta^\im+\alpha^\im\beta\rangle$. Then $\xh Q/I$ is complexified-gentle of uniform type and $\xh Q/I'$ is complexified-gentle of special type. We can show in a similar way that they are not Morita equivalent.
    \end{exm}

    For a specially gentle vertex $v$, conditions $I_v=I_v^0$ and $I_v=I_v^1$ sometimes give the same algebra up to isomorphism. In the above example, $\xr Q/\langle \alpha\beta+\alpha^\im\beta^\im, \alpha\beta^\im-\alpha^\im\beta \rangle$ is isomorphic to $\xr Q/\langle \alpha\beta-\alpha^\im\beta^\im, \alpha\beta^\im+\alpha^\im\beta \rangle$. But in some cases, these two conditions make an essential difference.

    \begin{exm}
        Assume that $$(Q,\mathcal{M})=\makecell{\xymatrix@C=4ex@R=1ex{ & \xc_w \ar[dl]_{\xc_{\gamma}} & \\ \xr_u\ar[rr]_{\xh_{\alpha}} & & \xh_v\ar[ul]_{\xc^2_{\beta}}}}, I=\langle \mathcal{M}(\gamma\beta), I_u^0, I_v^0\rangle, I'=\langle \mathcal{M}(\gamma\beta), I_u^0, I_v^1\rangle.$$
        Then $T(Q,\mathcal{M})/I$ and $T(Q,\mathcal{M})/I'$ are both complexified-gentle of special type. We claim that they are not isomorphic as $\xr$-algebras. Otherwise their complexified quiver presentations $\xc\Gamma/J$ and $\xc\Gamma/J'$ should be isomorphic as $\xc$-algebras. However, $\xc\Gamma/J$ and $\xc\Gamma/J'$ are given as follows:
        $$\Gamma=\makecell{\xymatrix@C=4ex@R=2ex{ & w \ar[dl]_{\gamma} & \\ u\ar@<.4ex>[rr]^{\alpha}\ar@<-0.4ex>[rr]_{\overline{\alpha}} & & v\ar[ul]_{\beta}\ar[dl]^{\overline{\beta}}\\ & \overline{w}\ar[ul]^{\overline{\gamma}}&}},J=\langle\gamma\beta,\overline{\gamma}\overline{\beta},\alpha\gamma,\overline{\alpha}\overline{\gamma},\beta\alpha,\overline{\beta}\overline{\alpha}\rangle, J'=\langle\gamma\beta,\overline{\gamma}\overline{\beta},\alpha\gamma,\overline{\alpha}\overline{\gamma},\overline{\beta}\alpha,\beta\overline{\alpha}\rangle.$$ They should not be isomorphic since there is a non-zero non-trivial circle $\beta\overline{\alpha}\gamma$ in $\xc\Gamma/J'$ (which gives a non-zero non-isomorphic endomorphism for some projective module), while there are no such circles in $\xc\Gamma/J$.
    \end{exm}
	
	In the following proposition, we give some examples of $\xr$-algebras that are not complexified-gentle. 
	
	\begin{prop}\label{notgentle}
		Assume that $T(Q,\mathcal{M})/I$ is in one of the following situations: 
		
		\begin{enumerate}
			\item[\rm (1)] $(Q,\mathcal{M})=\xr\overset{\xr}{\rightarrow}\xr\overset{\xh}{\rightarrow}\xh,\;\xr\overset{\xr}{\rightarrow}\xr\overset{\xc}{\rightarrow}\xc,\;\xh\overset{\xh}{\rightarrow}\xh\overset{\xh}{\rightarrow}\xr\mbox{ or }\xh\overset{\xh}{\rightarrow}\xh\overset{\xc^2}{\rightarrow}\xc,$

			\item[\rm (2)] 
			$(Q,\mathcal{M})=\xr\overset{\xr}{\rightarrow} \xr_u\underset{\xr}{\overset{\xr}{\rightrightarrows}} \xr_v\underset{\xr}{\overset{\xr}{\rightrightarrows}} \xr,\;
			\xr\overset{\xr}{\rightarrow} \xr_u\underset{\xr}{\overset{\xr}{\rightrightarrows}} \xr_v\overset{\xh}{\rightarrow} \xh,\; \xr\overset{\xr}{\rightarrow} \xr_u\underset{\xr}{\overset{\xr}{\rightrightarrows}} \xr_v\overset{\xc}{\rightarrow} \xc,\; $
			
			\hspace{20pt}$\xh\overset{\xh}{\rightarrow} \xh_u\underset{\xh}{\overset{\xh}{\rightrightarrows}} \xh_v\overset{\xh}{\rightarrow} \xr,\; \xh\overset{\xh}{\rightarrow} \xh_u\underset{\xh}{\overset{\xh}{\rightrightarrows}} \xh_v\underset{\xh}{\overset{\xh}{\rightrightarrows}} 
			\xh\mbox{ or }\xh\overset{\xh}{\rightarrow} \xh_u\underset{\xh}{\overset{\xh}{\rightrightarrows}} \xh_v\overset{\xc^2}{\rightarrow} \xc;$\\
			with $u$ ordinarily gentle and $v$ specially gentle, 
		\end{enumerate}
		then its complexification $\xc\Gamma/J$ cannot be a gentle algebra.
	\end{prop}
	\begin{proof}
		(1). Assume that $Q=u\overset{\beta
		}{\rightarrow}v\overset{\alpha}{\rightarrow}w$. For each case, in the complexifed quiver $\Gamma$ there are exactly two paths passing $v$. One is $\alpha\beta$ and the other is $\tau(\alpha\beta)$, which is $\alpha\overline{\beta}$ or $\overline{\alpha}\beta$. One can check that in each case, $\mathcal{M}(\alpha\beta)$ is a $\mathcal{M}(w)$-$\mathcal{M}(u)$-bimodule generated by any non-zero element. Hence, either $\mathcal{M}(\alpha\beta)\subseteq I$ or $\mathcal{M}(\alpha\beta)\cap I=0$. By \cite[Remark 3.6]{LJ2023}, $\alpha\beta$ and $\tau(\alpha\beta)$ both belong to $J$, xor both do not belong to $J$. That is, $\dim_{\xc}\mathrm{rad}^2(\xc\Gamma/J)=0$ or $2$. However, if $v$ satisfies the condition (G2) for some $\xc\Gamma/J'$ with $\xc\Gamma/J'\simeq\xc\Gamma/J$, then $\dim_{\xc}\mathrm{rad}^2(\xc\Gamma/J)=\dim_{\xc}\mathrm{rad}^2(\xc\Gamma/J')=1$, which is a contradiction.
		
		(2). Consider the case $\xh_x\overset{\xh_\gamma}{\rightarrow} \xh_u\underset{\xh_{\beta^\im}}{\overset{\xh_\beta}{\rightrightarrows}} \xh_v\overset{\xh_\alpha}{\rightarrow} \xr_w.$ Then $$\Gamma=x\overset{\gamma}{\rightarrow} u\underset{\beta^\im}{\overset{\beta}{\rightrightarrows}} v\underset{\overline{\alpha}}{\overset{\alpha}{\rightrightarrows}} w.$$
		If $\xc\Gamma/J$ is (isomorphic to) a gentle algebra, then there is a non-zero non-isomorphic morphism from the projective module corresponding with $w$ to the projective module corresponding with $x$.
		
		On the other hand, since $u$ is ordinarily gentle and $v$ is specially gentle, we assume that $I$ includes $1_{\alpha}\otimes1_{\beta}+1_{\alpha}\otimes\jm_{\beta^\im}$ and $1_{\beta}\otimes1_{\gamma}$. By calculation, $J$ includes $\alpha(\beta+\im\beta^\im), \overline{\alpha}(\beta-\im\beta^\im)$ and $\beta\gamma$; see Table \ref{psi(ab)}. Then $J$ includes all the (non-trivial) paths from $x$ to $w$. Then there is no non-zero non-isomorphic homomorphism from the projective module corresponding with $w$, to the projective module corresponding with $x$, which is a contradiction. Other assumptions on $I$ cause a similar contradiction.
		
		The statement for other modulated quivers listed above can be proved in a similar way.
	\end{proof}
	
	Notice that the statement holds for modulated quivers opposite to those in the above proposition, since the opposite algebra of a gentle algebra is gentle.
	
	Proposition \ref{notgentle} (1) partially shows that, if we only consider modulated quivers, the definition of gentle vertex in $T(Q,\mathcal{M})/I$ includes all the possible local situations that can occur in a complexified-gentle algebra. The proposition also shows that we cannot ``glue'' an ordinarily gentle vertex and a specially gentle vertex (not a source or sink) next to each other to get a complexified-gentle algebra. Therefore, we propose the following conjecture.
	
	\begin{conj}
		A connected $\xr$-algebra is locally complexified-gentle if and only if it is locally complexified-gentle either of uniform type or of special type.
	\end{conj}
	
	If we consider which $\xr$-algebras become string algebras (see \cite{BR1987} for the definition of string algebras) after complexification, then locally each vertex may be associated with a larger ideal. In this situation, the modulated quivers in the above proposition may not be excluded. Thus, defining ``complexified-string algebras'' by modulated quivers seems more complicated.
	
	\begin{exm}\label{r-string alg}
		Let $T(Q,\mathcal{M})/I$ be an $\xr$-algebra given by the modulated quiver $$\xh_x\overset{\xh_\gamma}{\rightarrow} \xh_u\underset{\xh_{\beta^\im}}{\overset{\xh_\beta}{\rightrightarrows}} \xh_v\overset{\xh_\alpha}{\rightarrow} \xr_w,$$ which is in case {\rm(2)} of the above proposition. And consider the ideal $$I=\langle 1_{\beta}\otimes1_{\gamma}, 1_{\alpha}\otimes1_{\beta},1_{\alpha}\otimes \jm_{\beta^\im}\rangle.$$ 
		Then its complexified quiver presentation $\xc\Gamma/J$ is given by $$\Gamma=x\overset{\gamma}{\rightarrow} u\underset{\beta^\im}{\overset{\beta}{\rightrightarrows}} v\underset{\overline{\alpha}}{\overset{\alpha}{\rightrightarrows}} w,\qquad J=\langle\beta\gamma,\alpha\beta,\alpha\beta^{\im},\overline{\alpha}\beta,\overline{\alpha}\beta^{\im}\rangle,$$ which is a string algebra.
	\end{exm}

	\section{The \texorpdfstring{$\xc$}{}-semilinear clannish algebras of gentle type}
	
	In this section, we study a special type of semilinear clannish algebras over $\xc$ in the sense of \cite{BC2024}. We show that these algebras and locally complexified-gentle algebras of special type coincide up to Morita equivalences.
	
	\medskip
	\subsection{The algebra \texorpdfstring{$\xc_{\sigma}Q/\langle S\rangle$}{}}
	Let $Q=(Q_0, Q_1, s, t)$ be a finite quiver and $$\sigma: Q_1\rightarrow \mathrm{Gal}(\xc/\xr)=\{\mathrm{Id},\overline{(?)}\}\mbox{, }\alpha\mapsto\sigma_{\alpha}$$ be twists on arrows, where $\overline{(?)}$ is taking conjugation. 
	
	Let $\xc_{\sigma}Q$ be the semilinear path algebra given by $Q$ and $\sigma$, that is, $\xc_{\sigma}Q$ is a $\xc$-ring generated by trivial paths $e_i$, $i\in Q_0$ and arrows $\alpha\in Q_1$ subject to the relations: for $ i,j\in Q_0,\; \alpha\in Q_1,\; c\in\xc$, $$e_ie_j=\begin{cases}
		e_i, & i=j\\0& i\neq j
	\end{cases}, \;\sum_{i\in Q_0}e_i=1, \; e_{t(\alpha)}\alpha=\alpha, \; \alpha e_{s(\alpha)}=\alpha, \; e_i c=c e_i, \; \alpha c=\sigma_{\alpha}(c)\alpha;$$ see \cite[Section 2.1]{BC2024}. Notice that $\xc_{\sigma}Q$ is an $\xr$-algebra.
	
	Let $\mathbb{S}$ be a subset of loops in $Q_1$ whose elements are called \textbf{special loops}. Arrows not in $\mathbb{S}$ are called \textbf{ordinary arrows}. We assume that:
	\begin{enumerate}
		\item $\sigma_{s}=\overline{(?)}, \forall s\in\mathbb{S}$;
		\item each $s\in\mathbb{S}$ is associated with a skew polynomial $x^2-1$ or $x^2+1$ in the skew polynomial ring $\xc[x;\overline{(?)}]$;
		\item each vertex has at most one special loop on it, and we denote this loop by $s_i$ if $i\in Q_0$ has one.
	\end{enumerate}
	Hence, there are three types of vertices in $Q$, that is, $$Q_0=Q_{\xr}\sqcup Q_{\xh}\sqcup Q_{\xc},$$ where $Q_{\xr}$ contains endpoints of special loops associated with $x^2-1$, $Q_{\xh}$ contains endpoints of special loops associated with $x^2+1$ and $Q_{\xc}$ is the set of vertices with no special loops on them.
	
	We consider the quotient algebra $$\xc_{\sigma}Q/\langle S\rangle \mbox{, where }  S=\{s_i^2-e_i\,|\,i\in Q_{\xr}\}\cup\{s_j^2+e_j\,|\, j\in Q_{\xh}\}.$$ It is an $\xr$-algebra. Denoted by $\lceil p\rceil$ the element $p+\langle S\rangle$ in $\xc_{\sigma}Q/\langle S\rangle$, for $p\in \xc_{\sigma}Q$.
	
	For each vertex $i\in Q$, let $A_i$ be the $\xc$-subring of $\xc_{\sigma}Q/\langle S\rangle$ generated by the trivial path $\lceil e_i\rceil $ and the special loop $\lceil s_i\rceil $ (if exists). That is,
	$$A_i=\begin{cases}
		\xc \lceil e_i\rceil \oplus \xc \lceil s_i\rceil,  (\simeq \xc[x;\overline{(?)}]/\langle x^2-1\rangle\simeq M_2(\xr)) &\mbox{ if } i\in Q_{\xr}\\
		\xc \lceil e_i\rceil \oplus \xc \lceil s_i\rceil, (\simeq \xc[x;\overline{(?)}]/\langle x^2+1\rangle\simeq \xh) &\mbox{ if } i\in Q_{\xh}\\
		\xc \lceil e_i\rceil, (\simeq \xc) &\mbox{ if } i\in Q_{\xc}
	\end{cases}.$$ In each case, $A_i$ is a simple $\xr$-algebra.
	
	For each ordinary arrow $\alpha\in Q_1\setminus\mathbb{S}$, let $M_{\alpha}$ be the $A_{t(\alpha)}$-$A_{s(\alpha)}$-bimodule generated by $\lceil \alpha\rceil $ in $\xc_{\sigma}Q/\langle S\rangle$. That is,
	$$M_{\alpha}=\begin{cases}
		\xc \lceil \alpha\rceil \oplus \xc \lceil s_{t(\alpha)}\alpha\rceil  \oplus \xc\lceil \alpha s_{s(\alpha)}\rceil \oplus \xc\lceil s_{t(\alpha)}\alpha s_{s(\alpha)}\rceil   &\mbox{if } t(\alpha)\notin Q_{\xc}, s(\alpha)\notin Q_{\xc}\\
		\xc \lceil \alpha\rceil \oplus \xc \lceil s_{t(\alpha)}\alpha\rceil    &\mbox{if } t(\alpha)\notin Q_{\xc}, s(\alpha)\in Q_{\xc}\\
		\xc \lceil \alpha\rceil \oplus \xc\lceil \alpha s_{s(\alpha)}\rceil  &\mbox{if } t(\alpha)\in Q_{\xc}, s(\alpha)\notin Q_{\xc}\\
		\xc \lceil \alpha\rceil   &\mbox{if } t(\alpha)\in Q_{\xc}, s(\alpha)\in Q_{\xc}
	\end{cases}$$
	
	Set $\textbf{A}=\underset{i\in Q_0}{\prod} A_i$ and $\textbf{M}=\underset{\alpha\in Q_1\setminus\mathbb{S}}{\bigoplus}M_{\alpha}$. The algebra $\xc_{\sigma}Q/\langle S\rangle$ is freely generated by $\textbf{A}$ and $\textbf{M}$; see \cite[Section 1]{BSZ2009}. Hence, we have the following result.
	
	\begin{prop}\label{CQS is t alg}
		As $\xr$-algebras, the algebra $\xc_{\sigma}Q/\langle S\rangle$ is isomorphic to the tensor algebra $T(\textbf{A}, \textbf{M})$.
	\end{prop}
	
	We identify $\xc_{\sigma}Q/\langle S\rangle$ with $T(\textbf{A}, \textbf{M})$ since the isomorphism can be given as the identity restricted on $\textbf{A}$ and $\textbf{M}$.
	
	\begin{rem}
		As in \cite[Section 2]{BC2024}, when we consider the ring $\xc[x;\sigma]/\langle p(x)\rangle$, where $p(x)$ is a quadratic polynomial, it is reasonable to require $p(x)$ to be normal. In this case, according to \cite[Lemma 2.3]{BC2024}, $\xc[x;\sigma]/\langle p(x)\rangle$ is isomorphic to one of the following four cases:
		
		\begin{tabular}{ll}
			(1) $\xc[x;\overline{(?)}]/\langle x^2+1\rangle\simeq\xh$; & (2) $\xc[x;\overline{(?)}]/\langle x^2-1\rangle\simeq M_2(\xr)$;\\
			(3) $\xc[x;\mathrm{Id}]/\langle x^2-x\rangle\simeq \xc\times\xc$; &(4) $\xc[x;\mathrm{Id}]/\langle x^2\rangle$.
		\end{tabular}\\
		Loops associated with case (3) are of ``skew-gentle type''. The loops associated with case (4) can be viewed as ordinary loops. In this paper, the special loops in $\mathbb{S}$ are associated with only case (1) or (2) since we only consider ``gentle type''.
	\end{rem}
	
	\medskip
	
	\subsection{The basic algebra of \texorpdfstring{$\xc_{\sigma}Q/\langle S\rangle$}{}} Let $\xc_{\sigma}Q/\langle S\rangle$ be an $\xr$-algebra as considered in Section 4.1.
	We give a modulated quiver presentation $T(Q^{\mathrm{b}},\mathcal{M})$ of $\xc_{\sigma}Q/\langle S\rangle$.
	
	\textbf{The construction of the quiver $Q^{\mathrm{b}}=(Q^{\mathrm{b}}_0, Q^{\mathrm{b}}_1, s^{\mathrm{b}}, t^{\mathrm{b}})$:}
	\begin{enumerate}
		\item $Q^{\mathrm{b}}_0=Q_0$.
		\item $Q^{\mathrm{b}}_1=\{\alpha\;|\; \alpha\in Q_1\setminus\mathbb{S}\mbox{ with } s(\alpha)\in Q_D \mbox{ such that } t(\alpha)\notin Q_D \mbox{ or } D=\xc \}\\ \mbox{ }\qquad\sqcup\, \{\alpha, \alpha^{\im}\;|\; \alpha\in Q_1\setminus\mathbb{S}\mbox{ with } s(\alpha), t(\alpha)\in Q_{\xr} \mbox{ or }s(\alpha), t(\alpha)\in Q_{\xh} \}$.\\ That is, each ordinary arrow $\alpha$ in $Q$ gives two parallel arrows $\alpha$ and $\alpha^\im$ in $Q^{\mathrm{b}}$ if its endpoints both belong to $Q_{\xr}$ or both belong to $Q_{\xh}$; otherwise, it gives one arrow $\alpha$ in $Q^{\mathrm{b}}$.
		\item $\forall \alpha \in Q_1\setminus\mathbb{S}$, $s^{\mathrm{b}}(\alpha)=s(\alpha)$, $t^{\mathrm{b}}(\alpha)=t(\alpha)$; and if $\alpha^{\im}$ exists, $s^{\mathrm{b}}(\alpha^{\im})=s(\alpha)$, $t^{\mathrm{b}}(\alpha^{\im})=t(\alpha)$.
	\end{enumerate}
	
	\textbf{The modulation $\mathcal{M}$ on $Q^{\mathrm{b}}$:}
	\begin{enumerate}
		\item For $i\in Q_{D}\subseteq Q_0=Q^{\mathrm{b}}_0$, where $D\in\{\xr,\xh, \xc\}$, set $\mathcal{M}(i)=D$.
		\item For $\alpha\in Q^{\mathrm{b}}_1$, set $$\mathcal{M}(\alpha)=\begin{cases}
			\xc, \qquad\qquad\qquad\qquad\mbox{ if } t(\alpha), s(\alpha)\in Q_{\xc} \mbox{ and } \sigma_{\alpha}=\mathrm{Id};\\
			\overline{\xc}, \qquad\qquad\qquad\qquad\mbox{ if } t(\alpha), s(\alpha)\in Q_{\xc} \mbox{ and } \sigma_{\alpha}=\overline{(?)};\\
			\mbox{the simple $\mathcal{M}(t(\alpha))$-$\mathcal{M}(s(\alpha))$-bimodule,}\qquad\quad\; \mbox{ else.}
		\end{cases}$$  
	\end{enumerate}
	
	For explicitness, sometimes by $D_x$ we denote the ring or bimodule $D$ that is related to $x\in Q^{\mathrm{b}}_0\sqcup Q^{\mathrm{b}}_1$ (or is the modulation on $x$) and by $k_x$ we denote the element $k$ in $D_x$. 
	
	\begin{thm}\label{CQS and TQM}
		The $\xr$-algebra $\xc_{\sigma}Q/\langle S\rangle$ is Morita equivalent to $T(Q^{\mathrm{b}},\mathcal{M})$. That is, there is an isomorphism $\Phi:\epsilon\xc_{\sigma}Q/\langle S\rangle\epsilon\simeq T(Q^{\mathrm{b}},\mathcal{M})$ for some full idempotent $\epsilon$ of $\xc_{\sigma}Q/\langle S\rangle$.
	\end{thm}
	
	\begin{proof}
		By Proposition \ref{CQS is t alg}, we need to give a full idempotent $\epsilon$ of $T(\textbf{A},\textbf{M})$ and an isomorphism from $\epsilon T(\textbf{A},\textbf{M})\epsilon$ to $T(Q^{\mathrm{b}},\mathcal{M})$. To this end, we first give some isomorphisms on $\textbf{A}$ and $\textbf{M}$.
		
		\textbf{(1)} The map $\phi_1$. 
		
		For $i\in Q_0=Q^\mathrm{b}_0$, we have the following isomorphisms of $\xr$-algebras.
		\begin{enumerate}
			\item[($A_\xr$)] If $i\in Q_{\xr}$, $A_i\simeq M_2(\xr)_i,$ given by 
			$$\lceil e_i\rceil \mapsto \begin{bmatrix} 1 & 0\\ 0& 1 \end{bmatrix}_i\mbox{, }\im\mapsto \begin{bmatrix} 0& -1 \\ 1& 0 \end{bmatrix}_i\mbox{ and }\lceil s_i\rceil \mapsto\begin{bmatrix} 1 & 0\\ 0& -1 \end{bmatrix}_i.$$ 
			The idempotent $\epsilon_i=\dfrac{\lceil e_i\rceil +\lceil s_i\rceil }{2}\in A_i$ is mapped to $\begin{bmatrix} 1 & 0\\ 0& 0 \end{bmatrix}_i$. So $$\epsilon_i A_i \epsilon_i\simeq \xr_i.$$
			\item[($A_\xh$)] If $i\in Q_{\xh}$, $$A_i\simeq \xh_i \mbox{, given by } \lceil e_i\rceil \mapsto 1_i \mbox{, }\im\mapsto \jm_i \mbox{ and }\lceil s_i\rceil \mapsto \km_i.$$
			\item[($A_\xc$)] If $i\in Q_{\xc}$, $$A_i\simeq \xc_i \mbox{, given by } \lceil e_i\rceil \mapsto 1_i.$$
		\end{enumerate}
		
		For $\alpha\in Q_1\setminus\mathbb{S}\subseteq Q^\mathrm{b}_1$, we have the following isomorphisms of $A_{t(\alpha)}$-$A_{s(\alpha)}$-bimodules via isomorphisms of algebras above.
		\begin{enumerate}
			\item[($_{\xr}M_{\xr}$)] If $t(\alpha)\in Q_{\xr}$ and $s(\alpha)\in Q_{\xr}$, via $(A_{\xr})$,
			$$M_{\alpha}\simeq M_2(\xr)_{\alpha}\oplus M_2(\xr)_{\alpha^\im} \mbox{, given by }\lceil \alpha\rceil \mapsto \begin{cases}
				\small{\begin{bmatrix} 1 & 0\\ 0& 1 \end{bmatrix}_{\alpha}\!+\begin{bmatrix} 0 & -1\\ 1& 0 \end{bmatrix}_{\alpha^\im}}  & \mbox{if } \sigma_{\alpha}=\mathrm{Id}\\
				\small{\begin{bmatrix} 1 & 0\\ 0& -1 \end{bmatrix}_{\alpha}\!+\begin{bmatrix} 0 & 1\\ 1& 0 \end{bmatrix}_{\alpha^\im}}  & \mbox{if } \sigma_{\alpha}=\overline{(?)}
			\end{cases}.$$ 
			After multiplying idempotents, we have $\epsilon_{t(\alpha)}M_{\alpha}\epsilon_{s(\alpha)}\simeq {\xr_{\alpha}}\oplus {\xr_{\alpha^\im}}.$
			
			\item[($_{\xr}M_{\xh}$)] If $t(\alpha)\in Q_{\xr}$ and $s(\alpha)\in Q_{\xh}$, via $(A_{\xr})$ and $(A_{\xh})$,
			$$M_{\alpha}\simeq \begin{bmatrix}
				{_{\xr}\xh_{\xh}} \\ _{\xr}\xh_{\xh}
			\end{bmatrix}_{\alpha}  \mbox{, given by }\lceil \alpha\rceil \mapsto \begin{cases}
				\small{\begin{bmatrix} 1 \\ -\jm \end{bmatrix}_{\alpha}}  & \mbox{if } \sigma_{\alpha}=\mathrm{Id}\\
				\small{\begin{bmatrix} 1 \\ \;\jm\;\; \end{bmatrix}_{\alpha}}  & \mbox{if } \sigma_{\alpha}=\overline{(?)}
			\end{cases}.$$ After multiplying idempotent, we have $\epsilon_{t(\alpha)}M_{\alpha}\simeq {\xh_{\alpha}}$.
			
			\item[($_{\xr}M_{\xc}$)] If $t(\alpha)\in Q_{\xr}$ and $s(\alpha)\in Q_{\xc}$, via $(A_{\xr})$ and $(A_{\xc})$,
			$$M_{\alpha}\simeq \begin{bmatrix}
				{_{\xr}\xc_{\xc}} \\ {_{\xr}\xc_{\xc}}
			\end{bmatrix}_{\alpha}  \mbox{, given by }\lceil \alpha\rceil \mapsto \begin{cases}
				\small{\begin{bmatrix} 1 \\ -\im \end{bmatrix}_{\alpha}}  & \mbox{if } \sigma_{\alpha}=\mathrm{Id}\\
				\small{\begin{bmatrix} \;1\;\; \\ \;\im\;\; \end{bmatrix}_{\alpha}}  & \mbox{if } \sigma_{\alpha}=\overline{(?)}
			\end{cases}.$$ After multiplying idempotent, we have $\epsilon_{t(\alpha)}M_{\alpha}\simeq {\xc_{\alpha}}$.
			
			\item[($_{\xh}M_{\xr}$)] If $t(\alpha)\in Q_{\xh}$ and $s(\alpha)\in Q_{\xr}$, via $(A_{\xh})$ and $(A_{\xr})$,
			$$M_{\alpha}\simeq \begin{bmatrix}
				{_{\xh}\xh_{\xr}} & {_{\xh}\xh_{\xr}}
			\end{bmatrix}_{\alpha}  \mbox{, given by }\lceil \alpha\rceil \mapsto \begin{cases}
				\small{\begin{bmatrix} 1 &\;\; \jm\end{bmatrix}_{\alpha}}  & \mbox{if } \sigma_{\alpha}=\mathrm{Id}\\
				\small{\begin{bmatrix} 1& -\jm \end{bmatrix}_{\alpha}}  & \mbox{if } \sigma_{\alpha}=\overline{(?)}
			\end{cases}.$$ After multiplying idempotent, we have $M_{\alpha}\epsilon_{s(\alpha)}\simeq {\xh_{\alpha}}$.
			
			\item[($_{\xh}M_{\xh}$)]If $t(\alpha)\in Q_{\xh}$ and $s(\alpha)\in Q_{\xh}$, via $(A_{\xh})$,
			$$M_{\alpha}\simeq {\xh_{\alpha}}\oplus{\xh_{\alpha^\im}} \mbox{, given by }\lceil \alpha\rceil \mapsto \begin{cases}
				1_{\alpha}+\jm_{\alpha^\im}  & \mbox{if } \sigma_{\alpha}=\mathrm{Id}\\
				\km_{\alpha}+\lm_{\alpha^\im} & \mbox{if } \sigma_{\alpha}=\overline{(?)}
			\end{cases}.$$ 
			
			\item[($_{\xh}M_{\xc}$)]If $t(\alpha)\in Q_{\xh}$ and $s(\alpha)\in Q_{\xc}$, via $(A_{\xh})$ and $(A_{\xc})$,
			$$M_{\alpha}\simeq \begin{bmatrix}
				\xc \\ \xc
			\end{bmatrix}_{\alpha}, \mbox{ given by }\lceil \alpha\rceil \mapsto \begin{cases}
				\small{\begin{bmatrix} 1 \\ 0 \end{bmatrix}_{\alpha}}  & \mbox{if } \sigma_{\alpha}=\mathrm{Id}\\
				\small{\begin{bmatrix} 0 \\ 1 \end{bmatrix}_{\alpha}}  & \mbox{if } \sigma_{\alpha}=\overline{(?)}
			\end{cases}.$$ 
			
			\item[($_{\xc}M_{\xr}$)]If $t(\alpha)\in Q_{\xc}$ and $s(\alpha)\in Q_{\xr}$, via $(A_{\xc})$ and $(A_{\xr})$,
			$$M_{\alpha}\simeq \begin{bmatrix}
				{_{\xc}\xc_{\xr}} & {_{\xc}\xc_{\xr}}
			\end{bmatrix}_{\alpha}  \mbox{, given by }\lceil \alpha\rceil \mapsto \begin{cases}
				\small{\begin{bmatrix} 1 &\;\;\im \end{bmatrix}_{\alpha}}  & \mbox{if } \sigma_{\alpha}=\mathrm{Id}\\
				\small{\begin{bmatrix} 1& -\im \end{bmatrix}_{\alpha}}  & \mbox{if } \sigma_{\alpha}=\overline{(?)}
			\end{cases}.$$ After multiplying idempotent, we have $M_{\alpha}\epsilon_{s(\alpha)}\simeq {\xc_{\alpha}}$.
			
			\item[($_{\xc}M_{\xh}$)]If $t(\alpha)\in Q_{\xc}$ and $s(\alpha)\in Q_{\xh}$, via $(A_{\xc})$ and $(A_{\xh})$,
			$$M_{\alpha}\simeq \begin{bmatrix}
				\xc & \xc
			\end{bmatrix}_{\alpha}, \mbox{ given by }\lceil \alpha\rceil \mapsto \begin{cases}
				\small{\begin{bmatrix} 1 &0 \end{bmatrix}_{\alpha}}  & \mbox{if } \sigma_{\alpha}=\mathrm{Id}\\
				\small{\begin{bmatrix} 0& 1 \end{bmatrix}_{\alpha}}  & \mbox{if } \sigma_{\alpha}=\overline{(?)}
			\end{cases}.$$ 
			
			\item[($_{\xc}M_{\xc}$)]If $t(\alpha)\in Q_{\xc}$ and $s(\alpha)\in Q_{\xc}$, via $(A_{\xc})$,
			$$M_{\alpha}\simeq \begin{cases} \xc_{\alpha}  & \mbox{if } \sigma_{\alpha}=\mathrm{Id}\\
				\overline{\xc}_{\alpha}  & \mbox{if } \sigma_{\alpha}=\overline{(?)}
			\end{cases}\mbox{, both given by }\lceil \alpha\rceil \mapsto 1_{\alpha}.$$
		\end{enumerate}
		
		We denote these isomorphisms by $\phi_1$. An idempotent of $T(\textbf{A}, \textbf{M})$ is given by $$\epsilon=(\epsilon_i)_{i\in Q_0}, \mbox{ where }\epsilon_i=\begin{cases}
			\dfrac{\lceil e_i\rceil +\lceil s_i\rceil }{2}& \mbox{if } i\in Q_{\xr}\\ \lceil e_i\rceil & \mbox{if }i\in Q_{\xh}\cup Q_{\xc}
		\end{cases}.$$

		\textbf{(2)} The map $\phi_2$.
		
		The map $\phi_1$ induces
		$$\epsilon T(\textbf{A},\textbf{M})\epsilon\simeq \phi_1(\epsilon) T(\phi_1(\textbf{A}),\phi_1(\textbf{M}))\phi_1(\epsilon).$$
		By the construction of $Q^\mathrm{b}$ and $\mathcal{M}$, we have: $$T(\phi_1(\epsilon)\phi_1(\textbf{A})\phi_1(\epsilon),\phi_1(\epsilon)\phi_1(\textbf{M})\phi_1(\epsilon))=T(\phi_1(\epsilon\mathbf{A} \epsilon), \phi_1(\epsilon \textbf{M} \epsilon))= T(Q^{\mathrm{b}},\mathcal{M}).$$
		Hence we only need to give $$\phi_2\colon\phi_1(\epsilon) T(\phi_1(\textbf{A}),\phi_1(\textbf{M}))\phi_1(\epsilon)\simeq T(\phi_1(\epsilon)\phi_1(\textbf{A})\phi_1(\epsilon),\phi_1(\epsilon)\phi_1(\textbf{M})\phi_1(\epsilon)).$$ This is a well-known result. Here we give a concrete map, which is needed later.
		
		To give isomorphisms
		\begin{align*}
			&\phi_1(\epsilon\textbf{M})\otimes_{\phi_1(\textbf{A})} \phi_1(\textbf{M})\otimes_{\phi_1(\textbf{A})}\cdots\otimes_{\phi_1(\textbf{A})}\phi_1(\textbf{M}\epsilon)\\
			\overset{\phi_2}{\simeq}& \phi_1(\epsilon\textbf{M}\epsilon)\otimes_{\phi_1(\epsilon\textbf{A}\epsilon)}\phi_1(\epsilon\textbf{M}\epsilon)\otimes_{\phi_1(\epsilon\textbf{A}\epsilon)}\cdots\otimes_{\phi_1(\epsilon\textbf{A}\epsilon)}\phi_1(\epsilon\textbf{M}\epsilon),
		\end{align*} we only need to give isomorphisms $$\phi_2:\phi_1(M_{\alpha})\otimes_{\phi_1(A_v)} \phi_1(M_{\beta})\simeq \phi_1(M_{\alpha}\epsilon_v)\otimes_{\phi_1(\epsilon_v A_v\epsilon_v)}\phi_1(\epsilon_v M_{\beta}),$$ for $\alpha,\beta\in Q_1\setminus\mathbb{S}$ with $s(\alpha)=t(\beta)=v$. 
		
		If $v\in Q_{\xh}$ or $Q_{\xc}$, then $\epsilon_v A_v\epsilon_v=A_v$ and set $\phi_2$ as the identity map. If $v\in Q_{\xr}$, that is, $\phi_1(A_v)=M_2(\xr)$, $\phi_2$ can be given by the isomorphisms in Lemma \ref{map 2} after taking appropriate rings and bimodules: $D=\xr$, $A,B\in\{M_2(\xr),\xh,\xc\}$, $M,N\in\{M_2(\xr),\xh,\xc\}$.

		Therefore, we have 
		\begin{align*}
			\Phi\colon\epsilon\xc_{\sigma}Q/\langle S\rangle\epsilon&=\epsilon T(\textbf{A},\textbf{M})\epsilon\overset{\phi_1}{\simeq}\phi_1(\epsilon) T(\phi_1(\textbf{A}),\phi_1(\textbf{M}))\phi_1(\epsilon)\\&\overset{\phi_2}{\simeq}T(\phi_1(\epsilon\textbf{A} \epsilon), \phi_1(\epsilon \textbf{M} \epsilon))= T(Q^{\mathrm{b}},\mathcal{M}).
		\end{align*}
	\end{proof}
	
	\begin{rem}
		(1) There are other isomorphisms from $\epsilon\xc_{\sigma}Q/\langle S\rangle\epsilon$ to $T(Q^\mathrm{b},\mathcal{M})$. To do computation, we fix the isomorphism given in the proof above: $$\Phi=\phi_2\circ\phi_1\colon\epsilon\xc_{\sigma}Q/\langle S\rangle\epsilon\simeq T(Q^\mathrm{b},\mathcal{M}).$$
		
		(2) Set $$\mathcal{M}(\alpha)^{\mathrm{b}}=\begin{cases}
			\mathcal{M}(\alpha) & \mbox{ if } \alpha^{\im} \mbox{ does not exist}\\ \mathcal{M}(\alpha)\oplus\mathcal{M}(\alpha^{\im}) & \mbox{ if } \alpha^{\im} \mbox{ exists}
		\end{cases}, \forall \alpha\in Q_1\setminus\mathbb{S}\subset Q^{\mathrm{b}}_1.$$
		
		Then $\Phi(\epsilon A_{t(\alpha)}\lceil\alpha\rceil A_{s(\alpha)}\epsilon)=\mathcal{M}(\alpha)^{\mathrm{b}}$.
	\end{rem}
	\begin{exm}
		Assume that $\xc_{\sigma}Q/\langle S\rangle$ is given as below: $$Q=\xymatrix{1\ar@(ul,ur)_{s_1}\ar[r]^{\alpha_1}& 2\ar@(ul,ur)_{s_2}\ar[r]^{\alpha_2}& 3\ar@(ul,ur)_{s_3}\ar[r]^{\alpha_3}& 4\ar[r]^{\alpha_4} &5}$$ with $1,2\in Q_{\xh}$, $3\in Q_{\xr}$, $4,5\in Q_{\xc}$, and $\sigma_{\alpha_4}=\overline{(?)}.$ Then $T(Q^\mathrm{b},\mathcal{M})$ can be depicted as:
		$$\xymatrix{\xh\ar@<0.5ex>[r]^{\xh}\ar@<-0.5ex>[r]_{\xh}& \xh\ar[r]^{\xh}& \xr\ar[r]^{\xc}& \xc\ar[r]^{\overline{\xc}} &\xc}.$$
	\end{exm}
	
	\medskip
	\subsection{The \texorpdfstring{$\xc$}{}-semilinear clannish algebras of gentle type}
	Let $\xc_{\sigma}Q/\langle S\rangle$ be an algebra as in the former subsections and $Z$ be a subset of paths in $Q$ with length of at least two. We require that no path in $Z$ starts or ends with a special loop or contains $s\circ s$ as a subpath for some special loop $s$; see \cite[Section 1]{BC2024}.

	\begin{defn}\label{csca of gentle}
		An algebra $\xc_{\sigma}Q/\langle S\cup Z\rangle$ given above is called \textbf{$\xc$-semilinear clannish of gentle type} if it satisfies the following conditions. \begin{enumerate}
			\item[(G1)] $\forall v\in Q_0$, $|v^+|\leq 2$ and $|v^-|\leq 2$.
			\item[(G2)] $\forall \alpha\in Q_1\setminus\mathbb{S}$, there is at most one arrow $\beta$ with $\alpha\beta$ a path not in $Z$ and at most one arrow $\gamma$ with $\gamma\alpha$ a path not in $Z$. 
			\\ $\forall \alpha\in Q_1\setminus\mathbb{S}$, there is at most one arrow $\beta$ with $\alpha\beta$ a path in $Z$ and at most one arrow $\gamma$ with $\gamma\alpha$ a path in $Z$.
			\item[(G3)] The paths in $Z$ are of length two.
		\end{enumerate}
	\end{defn}
    
	\begin{rem} (1) Let $v$ be a vertex in a $\xc$-semilinear clannish algebra $\xc_{\sigma}Q/\langle S\cup Z\rangle$ of gentle type such that $v\in Q_{\xr}\cup Q_{\xh}$. Then there is at most one ordinary arrow starting at $v$ and at most one ending at $v$. If $v^+=\{\alpha,s_v\}$ and $v^-=\{\beta,s_v\}$ with $\alpha\neq s_v\neq\beta$, then $\alpha\beta\in Z$.
    
		(2) A $\xc$-similinear clannish algebra of gentle type is a semilinear clannish algebra over $\xc$ which is normally bounded, non-singular and of (semi)simple type; see \cite[Section 1]{BC2024}.	
		
		(3) The algebra $\xc_{\sigma}Q/\langle S\cup Z\rangle$ can be infinite-dimensional over $\xc$ (and $\xr$).
	\end{rem}
	
	Given a $\xc$-semilinear clannish algebra of gentle type $\xc_{\sigma}Q/\langle S\cup Z\rangle$, we have already given a modulated quiver presentation 
	$$\Phi\colon\epsilon\xc_{\sigma}Q/\langle S\rangle\epsilon\simeq T(Q^\mathrm{b},\mathcal{M}).$$
	Now let $$I:=\Phi(\epsilon\langle\lceil p\rceil\,|\,p\in Z\rangle\epsilon)\subset\Phi(\epsilon\xc_{\sigma}Q/\langle S\rangle\epsilon)=T(Q^\mathrm{b},\mathcal{M}).$$
	It is obvious that $I$ is an ideal of $T(Q^\mathrm{b},\mathcal{M})$. It is an admissible ideal if $\xc_{\sigma}Q/\langle S\cup Z\rangle$ is finite-dimensional over $\xc$, or equivalently if $T(Q^\mathrm{b},\mathcal{M})/I$ is a finite-dimensional $\xr$-algebra. For convenience, we identify $\xc_{\sigma}Q/\langle S\cup Z\rangle$ with $(\xc_{\sigma}Q/\langle S\rangle)/ \langle\lceil p\rceil\,|\,p\in Z\rangle$ by the canonical isomorphism between them. Immediately, we have the following theorem.
	
	\begin{thm}\label{basic form of cq/sz}
		Given the notation as above, the isomorphism $\Phi$ induces $$\epsilon\xc_{\sigma}Q/\langle S\cup Z \rangle\epsilon\simeq T(Q^\mathrm{b},\mathcal{M})/I.$$ That is, $\xc_{\sigma}Q/\langle S\cup Z \rangle$ is Morita equivalent to $T(Q^\mathrm{b},\mathcal{M})/I$.
	\end{thm}

	Now we describe generators in $I$. Assume that $Z=\{\alpha_1\beta_1,\alpha_2\beta_2,\dots,\alpha_r\beta_r\}$, where $\alpha_i:v_i\rightarrow w_i$ and $\beta_i:u_i\rightarrow v_i$ are arrows in $Q_1\setminus \mathbb{S}$. According to the construction of $Q^{\mathrm{b}}$, the quiver $Q^{\mathrm{b}}$ contains vertices $u_i$, $v_i$, $w_i$; arrows $\alpha_i$, $\beta_i$, $\alpha_i^{\im}$ (exists if $\mathcal{M}(v_i)=\mathcal{M}(w_i)\neq\xc)$ and $\beta_i^{\im}$ (exists if $\mathcal{M}(u_i)=\mathcal{M}(v_i)\neq\xc)$. 
	Then the ideal $I$ can be generated by some $\mathcal{M}(w_i)$-$\mathcal{M}(u_i)$-submodules $I_{\alpha_i\beta_i}$ of $$\mathcal{M}(\alpha_i\beta_i)^{\mathrm{b}}:=\mathcal{M}(\alpha_i)^{\mathrm{b}}\otimes_{\mathcal{M}(v_i)}\mathcal{M}(\beta_i)^{\mathrm{b}},$$ with $\Phi(\epsilon\lceil \alpha_i\beta_i\rceil\epsilon)\in I_{\alpha_i\beta_i}, i=1,2,\dots,r$.
	
	Consider $\alpha\beta:u\rightarrow v\rightarrow w$ in $Z$. We describe $I_{\alpha\beta}$.
	
	(1) If $v\in Q_{\xc}$, that is, there is no special loop on $v$ and $\mathcal{M}(v)=\xc$ in $(Q^\mathrm{b},\mathcal{M})$, then $A_w\lceil\alpha\beta\rceil A_u=M_{\alpha}\otimes_{A_v}M_{\beta}$ as an $A_w$-$A_u$-bimodule. Hence $$\Phi(\epsilon\langle\lceil \alpha\beta\rceil\rangle\epsilon)=\Phi(\epsilon M_{\alpha}\otimes_{A_v}M_{\beta}\epsilon)= \mathcal{M}(\alpha\beta)^{\mathrm{b}}=\mathcal{M}(\alpha\beta).$$ Therefore $$I_{\alpha\beta}=\mathcal{M}(\alpha\beta), \mbox{ if }\mathcal{M}(s(\alpha))=\xc. $$
	
	(2) If $v\notin Q_{\xc}$, that is, there is a special loop $s$ on $v$, then as an $A_w$-$A_u$-bimodule, $M_{\alpha}\otimes_{A_v}M_{\beta}=A_w\lceil\alpha\beta\rceil A_u\oplus A_w\lceil\alpha s\beta\rceil A_u$. According to the construction of $T(Q^\mathrm{b},\mathcal{M})$, the modulated quivers given by $\alpha\beta$ are listed in Table \ref{phi(ab)}. 
	
	Via the map $\Phi$, we have a decomposition of $\mathcal{M}(w)$-$\mathcal{M}(u)$-bimodule, $$\mathcal{M}(\alpha\beta)^{\mathrm{b}}=\Phi(\epsilon\langle\lceil \alpha\beta\rceil\rangle\epsilon)\oplus\Phi(\epsilon\langle\lceil \alpha s \beta\rceil\rangle\epsilon).$$
	
	The values of $\Phi(\epsilon\lceil \alpha\beta\rceil\epsilon)$ with all the modulations and twists are listed in Table \ref{phi(ab)}, including the values of $\Phi(\epsilon\lceil \alpha s \beta\rceil\epsilon)$ with twists $\sigma_{\alpha}=\sigma_{\beta}=\mathrm{Id}$. Some detailed computations are shown later as examples. We denote the element $k$ in $\mathcal{M}(\gamma)$ by $k_{\gamma}$, $\gamma\in Q^{\mathrm{b}}_1$.
	
	Take any twist. Except for cases $\xr\xr\xr$ and $\xr\xh\xr$, as a $\mathcal{M}(w)$-$\mathcal{M}(u)$-bimodule, $$\mathcal{M}(\alpha\beta)^{\mathrm{b}}=\mathcal{M}(w)\Phi(\epsilon\lceil \alpha\beta\rceil\epsilon)\mathcal{M}(u)\oplus\mathcal{M}(w)\Phi(\epsilon\lceil \alpha s\beta\rceil\epsilon)\mathcal{M}(u).$$ For cases $\xr\xr\xr$ and $\xr\xh\xr$, \begin{align*}
		\mathcal{M}(\alpha\beta)^{\mathrm{b}}=&\mathcal{M}(w)\Phi(\epsilon\lceil \alpha\beta\rceil\epsilon)\mathcal{M}(u)\oplus\mathcal{M}(w)\Phi(\epsilon\lceil \alpha s\beta\rceil\epsilon)\mathcal{M}(u) \\ &\oplus\mathcal{M}(w)\Phi(\epsilon\im\lceil \alpha\beta\rceil\epsilon)\mathcal{M}(u)\oplus\mathcal{M}(w)\Phi(\epsilon\im\lceil \alpha s\beta\rceil\epsilon)\mathcal{M}(u);
	\end{align*} see examples below.
	
	Denote by $I_{\alpha\beta}^0$ the $\mathcal{M}(w)$-$\mathcal{M}(u)$-bimodule generated by $\Phi(\epsilon\lceil \alpha\beta\rceil\epsilon)$ (by $\Phi(\epsilon\lceil \alpha\beta\rceil\epsilon)$ and $\Phi(\epsilon\im\lceil \alpha\beta\rceil\epsilon)$ for cases $\xr\xr\xr$ and $\xr\xh\xr$) with $\sigma_{\alpha}=\sigma_{\beta}=\mathrm{Id}$, and by $I_{\alpha\beta}^1$ denote the $\mathcal{M}(w)$-$\mathcal{M}(u)$-bimodule generated by $\Phi(\epsilon\lceil \alpha s\beta\rceil\epsilon)$ (by $\Phi(\epsilon\lceil \alpha s\beta\rceil\epsilon)$ and $\Phi(\epsilon\im\lceil \alpha s\beta\rceil\epsilon)$ for cases $\xr\xr\xr$ and $\xr\xh\xr$) with $\sigma_{\alpha}=\sigma_{\beta}=\mathrm{Id}$. Then $$\mathcal{M}(\alpha\beta)^{\mathrm{b}}=I_{\alpha\beta}^0\oplus I_{\alpha\beta}^1.$$ Following the results in Table \ref{phi(ab)}, we see that for arbitrary twist, $$\Phi(\epsilon\langle\lceil \alpha\beta\rceil\rangle\epsilon)\cap\mathcal{M}(\alpha\beta)^{\mathrm{b}}=I_{\alpha\beta}^0 \mbox{ xor }I_{\alpha\beta}^1.$$ Hence $$I_{\alpha\beta}=I_{\alpha\beta}^0 \mbox{ xor }I_{\alpha\beta}^1, \mbox{ if }\mathcal{M}(s(\alpha))\neq\xc.$$
	
	\begin{table}
		
		\small{\begin{tabular}{ll|l|l}
				\cline{1-4}
				cases & \begin{tabular}[c]{@{}l@{}}modulated quivers\\ for $\mathcal{M}(\alpha\beta)^{\mathrm{b}}$ \end{tabular}  &\small{\begin{tabular}[c]{@{}l@{}}$\Phi(\epsilon\lceil\alpha\beta\rceil\epsilon)$ ($\Phi(\epsilon\im\lceil\alpha\beta\rceil\epsilon)$)\! \\with $\sigma_{\alpha}=\mathrm{Id},\sigma_{\beta}=\mathrm{Id}$\end{tabular}} & \small{\begin{tabular}[c]{@{}l@{}}$\Phi(\epsilon\lceil\alpha s\beta\rceil\epsilon)$ ($\Phi(\epsilon\im\lceil\alpha s\beta\rceil\epsilon)$)\!\\ with $\sigma_{\alpha}=\mathrm{Id},\sigma_{\beta}=\mathrm{Id}$\end{tabular}}\\ \cline{2-4} 
				&\tiny{$\Phi(\epsilon\lceil\alpha\beta\rceil\epsilon): \sigma_{\alpha}=\mathrm{Id},\sigma_{\beta}=\overline{(?)}$} &\tiny{$\Phi(\epsilon\lceil\alpha\beta\rceil\epsilon): \sigma_{\alpha}=\overline{(?)},\sigma_{\beta}=\mathrm{Id}$}& \tiny{$\Phi(\epsilon\lceil\alpha\beta\rceil\epsilon): \sigma_{\alpha}=\overline{(?)},\sigma_{\beta}=\overline{(?)}$}\\ \hline
				
				$\xr\xr\xr$ & \makecell{$\xymatrix@C=4ex@R=-1ex{\xr_u\ar@<.4ex>[r]^{\xr_\beta,\,\xr_{\beta^\im}}\ar@<-.4ex>[r] & \xr_v \ar@<.4ex>[r]^{\xr_\alpha,\,\xr_{\alpha^\im}}\ar@<-0.4ex>[r] & \xr_w}$} &\begin{tabular}[c]{@{}l@{}}$1_{\alpha}\otimes1_{\beta}-1_{\alpha^\im}\otimes1_{\beta^\im}$\\ \small{($-1_{\alpha}\otimes1_{\beta^\im}-\!1_{\alpha^\im}\otimes1_{\beta}$)}\end{tabular}  & \begin{tabular}[c]{@{}l@{}}$1_{\alpha}\otimes1_{\beta}+1_{\alpha^\im}\otimes1_{\beta^\im}$\\ \small{($1_{\alpha}\otimes1_{\beta^\im}-\!1_{\alpha^\im}\otimes1_{\beta}$)}\end{tabular} \\ \cline{2-4} 
				&  $1_{\alpha}\otimes1_{\beta}-1_{\alpha^\im}\otimes1_{\beta^\im}$ & $1_{\alpha}\otimes1_{\beta}+1_{\alpha^\im}\otimes1_{\beta^\im}$ & $1_{\alpha}\otimes1_{\beta}+1_{\alpha^\im}\otimes1_{\beta^\im}$ \\ \hline
				
				$\xr\xr\xh$ & \makecell{$\xymatrix@C=4ex@R=-1ex{\xr_u\ar@<.4ex>[r]^{\xr_\beta,\,\xr_{\beta^\im}}\ar@<-.4ex>[r] & \xr_v \ar[r]^{\xh_\alpha} & \xh_w}$} & $1_{\alpha}\otimes1_{\beta}+\jm_{\alpha}\otimes1_{\beta^\im}$ & $1_{\alpha}\otimes1_{\beta}-\jm_{\alpha}\otimes1_{\beta^\im}$ \\ \cline{2-4} 
				& $1_{\alpha}\otimes1_{\beta}+\jm_{\alpha}\otimes1_{\beta^\im}$ & $1_{\alpha}\otimes1_{\beta}-\jm_{\alpha}\otimes1_{\beta^\im}$ & $1_{\alpha}\otimes1_{\beta}-\jm_{\alpha}\otimes1_{\beta^\im}$ \\\hline

				$\xr\xr\xc$ & \makecell{$\xymatrix@C=4ex@R=-1ex{\xr_u\ar@<.4ex>[r]^{\xr_\beta,\,\xr_{\beta^\im}}\ar@<-.4ex>[r] & \xr_v \ar[r]^{\xc_\alpha} & \xc_w}$} & $1_{\alpha}\otimes1_{\beta}+\im_{\alpha}\otimes1_{\beta^\im}$ & $1_{\alpha}\otimes1_{\beta}-\im_{\alpha}\otimes1_{\beta^\im}$ \\ \cline{2-4} 
				& $1_{\alpha}\otimes1_{\beta}+\im_{\alpha}\otimes1_{\beta^\im}$ & $1_{\alpha}\otimes1_{\beta}-\im_{\alpha}\otimes1_{\beta^\im}$ & $1_{\alpha}\otimes1_{\beta}-\im_{\alpha}\otimes1_{\beta^\im}$ \\ \hline
				
				$\xh\xh\xr$ & \makecell{$\xymatrix@C=4ex@R=-1ex{\xh_u\ar@<.4ex>[r]^{\xh_\beta,\,\xh_{\beta^\im}}\ar@<-.4ex>[r] & \xh_v \ar[r]^{\xh_\alpha} & \xr_w}$} & $1_{\alpha}\otimes1_{\beta}+1_{\alpha}\otimes\jm_{\beta^\im}$ & $1_{\alpha}\otimes\km_{\beta}-1_{\alpha}\otimes\lm_{\beta^\im}$ \\ \cline{2-4} 
				&  $1_{\alpha}\otimes\km_{\beta}+1_{\alpha}\otimes\lm_{\beta^\im}$ & $1_{\alpha}\otimes1_{\beta}+1_{\alpha}\otimes\jm_{\beta^\im}$ & $1_{\alpha}\otimes\km_{\beta}+1_{\alpha}\otimes\lm_{\beta^\im}$ \\ \hline
				
				$\xh\xh\xh$ & \makecell{$\xymatrix@C=4ex@R=-1ex{\xh_u\ar@<.4ex>[r]^{\xh_\beta,\,\xh_{\beta^\im}}\ar@<-.4ex>[r] & \xh_v \ar@<.4ex>[r]^{\xh_\alpha,\,\xh_{\alpha^\im}}\ar@<-0.4ex>[r] & \xh_w}$} & \begin{tabular}[c]{@{}l@{}}$1_{\alpha}\otimes1_{\beta}-1_{\alpha^\im}\otimes1_{\beta^\im}$\\$\quad+1_{\alpha}\otimes\jm_{\beta^\im}+1_{\alpha^\im}\otimes\jm_{\beta}$\end{tabular} & \begin{tabular}[c]{@{}l@{}}$1_{\alpha}\otimes\km_{\beta}+1_{\alpha^\im}\otimes\km_{\beta^\im}$\\$\quad-1_{\alpha}\otimes\lm_{\beta^\im}+1_{\alpha^\im}\otimes\lm_{\beta}$\end{tabular} \\ \cline{2-4} 
				&  \begin{tabular}[c]{@{}l@{}}$1_{\alpha}\otimes\km_{\beta}-1_{\alpha^\im}\otimes\km_{\beta^\im}$\\$\quad+1_{\alpha}\otimes\lm_{\beta^\im}+1_{\alpha^\im}\otimes\lm_{\beta}$\end{tabular} & \begin{tabular}[c]{@{}l@{}}$1_{\alpha}\otimes\km_{\beta}+1_{\alpha^\im}\otimes\km_{\beta^\im}$\\$\quad-1_{\alpha}\otimes\lm_{\beta^\im}+1_{\alpha^\im}\otimes\lm_{\beta}$\end{tabular} & \begin{tabular}[c]{@{}l@{}}$-1_{\alpha}\otimes1_{\beta}-1_{\alpha^\im}\otimes1_{\beta^\im}$\\$\quad+1_{\alpha}\otimes\jm_{\beta^\im}-1_{\alpha^\im}\otimes\jm_{\beta}$\end{tabular}  \\ \hline
				
				$\xh\xh\xc$ & \makecell{$\xymatrix@C=4ex@R=-1ex{\xh_u\ar@<.4ex>[r]^{\xh_\beta,\,\xh_{\beta^\im}}\ar@<-.4ex>[r] & \xh_v \ar[r]^{\xc^2_\alpha} & \xc_w}$} & $\tiny{\begin{bmatrix}1&0\end{bmatrix}}_{\alpha}\otimes(1_{\beta}+\jm_{\beta^\im})$ &  $\tiny{\begin{bmatrix}1&0\end{bmatrix}}_{\alpha}\otimes(\km_{\beta}-\lm_{\beta^\im})$  \\ \cline{2-4} 
				&  $\tiny{\begin{bmatrix}1&0\end{bmatrix}}_{\alpha}\otimes(\km_{\beta}+\lm_{\beta^\im})$ & $\tiny{\begin{bmatrix}1&0\end{bmatrix}}_{\alpha}\otimes(\km_{\beta}-\lm_{\beta^\im})$ & $\tiny{\begin{bmatrix}1&0\end{bmatrix}}_{\alpha}\otimes(-1_{\beta}+\jm_{\beta^\im})$ \\ \hline
				
				$\xh\xr\xr$ & \makecell{$\xymatrix@C=4ex@R=-1ex{\xh_u\ar[r]^{\xh_\beta} & \xr_v \ar@<.4ex>[r]^{\xr_\alpha,\,\xr_{\alpha^\im}}\ar@<-0.4ex>[r] & \xr_w}$} & $1_{\alpha}\otimes1_{\beta}+1_{\alpha^\im}\otimes\jm_{\beta}$ &  $1_{\alpha}\otimes1_{\beta}-1_{\alpha^\im}\otimes\jm_{\beta}$   \\ \cline{2-4} 
				&  $1_{\alpha}\otimes1_{\beta}-1_{\alpha^\im}\otimes\jm_{\beta}$ & $1_{\alpha}\otimes1_{\beta}-1_{\alpha^\im}\otimes\jm_{\beta}$ & $1_{\alpha}\otimes1_{\beta}+1_{\alpha^\im}\otimes\jm_{\beta}$  \\ \hline
				
				$\xc\xr\xr$ & \makecell{$\xymatrix@C=4ex@R=-1ex{\xc_u\ar[r]^{\xc_\beta} & \xr_v \ar@<.4ex>[r]^{\xr_\alpha,\,\xr_{\alpha^\im}}\ar@<-0.4ex>[r] & \xr_w}$} & $1_{\alpha}\otimes1_{\beta}+1_{\alpha^\im}\otimes\im_{\beta}$ &  $1_{\alpha}\otimes1_{\beta}-1_{\alpha^\im}\otimes\im_{\beta}$   \\ \cline{2-4} 
				&  $1_{\alpha}\otimes1_{\beta}-1_{\alpha^\im}\otimes\im_{\beta}$ & $1_{\alpha}\otimes1_{\beta}-1_{\alpha^\im}\otimes\im_{\beta}$ & $1_{\alpha}\otimes1_{\beta}+1_{\alpha^\im}\otimes\im_{\beta}$  \\ \hline
				
				$\xr\xh\xh$ & \makecell{$\xymatrix@C=4ex@R=-1ex{\xr_u\ar[r]^{\xh_\beta} & \xh_v \ar@<.4ex>[r]^{\xh_\alpha,\,\xh_{\alpha^\im}}\ar@<-0.4ex>[r] & \xh_w}$} & $1_{\alpha}\otimes1_{\beta}+1_{\alpha^\im}\otimes\jm_{\beta}$ &  $1_{\alpha}\otimes\km_{\beta}+1_{\alpha^\im}\otimes\lm_{\beta}$ \\ \cline{2-4} 
				& $1_{\alpha}\otimes1_{\beta}+1_{\alpha^\im}\otimes\jm_{\beta}$ & $1_{\alpha}\otimes\km_{\beta}+1_{\alpha^\im}\otimes\lm_{\beta}$ & $1_{\alpha}\otimes\km_{\beta}+1_{\alpha^\im}\otimes\lm_{\beta}$  \\ \hline
				
				$\xc\xh\xh$ & \makecell{$\xymatrix@C=4ex@R=-1ex{\xc_u\ar[r]^{\xc^2_\beta} & \xh_v \ar@<.4ex>[r]^{\xh_\alpha,\,\xh_{\alpha^\im}}\ar@<-0.4ex>[r] & \xh_w}$} & $1_{\alpha}\otimes\tiny{\begin{bmatrix}1\\0\end{bmatrix}}_{\beta}+\jm_{\alpha^\im}\otimes\tiny{\begin{bmatrix}1\\0\end{bmatrix}}_{\beta}$ &  $\km_{\alpha}\otimes\tiny{\begin{bmatrix}1\\0\end{bmatrix}}_{\beta}+\lm_{\alpha^\im}\otimes\tiny{\begin{bmatrix}1\\0\end{bmatrix}}_{\beta}$  \\ \cline{2-4} 
				&  $-\km_{\alpha}\otimes\tiny{\begin{bmatrix}1\\0\end{bmatrix}}_{\beta}\!-\lm_{\alpha^\im}\otimes\tiny{\begin{bmatrix}1\\0\end{bmatrix}}_{\beta}$ & $\km_{\alpha}\otimes\tiny{\begin{bmatrix}1\\0\end{bmatrix}}_{\beta}+\lm_{\alpha^\im}\otimes\tiny{\begin{bmatrix}1\\0\end{bmatrix}}_{\beta}$ & $1_{\alpha}\otimes\tiny{\begin{bmatrix}1\\0\end{bmatrix}}_{\beta}+\jm_{\alpha^\im}\otimes\tiny{\begin{bmatrix}1\\0\end{bmatrix}}_{\beta}$  \\ \hline
				
				$\xh\xr\xh$ & \makecell{$\xymatrix@C=4ex@R=-1ex{\xh_u\ar@<-.4ex>[r]^{\xh_\beta} & \xr_v \ar@<-.4ex>[r]^{\xh_\alpha} & \xh_w}$} & $1_{\alpha}\otimes1_{\beta}-\jm_{\alpha}\otimes\jm_{\beta}$ &  $1_{\alpha}\otimes1_{\beta}+\jm_{\alpha}\otimes\jm_{\beta}$  \\ \cline{2-4} 
				&$1_{\alpha}\otimes1_{\beta}+\jm_{\alpha}\otimes\jm_{\beta}$ & $1_{\alpha}\otimes1_{\beta}+\jm_{\alpha}\otimes\jm_{\beta}$ & $1_{\alpha}\otimes1_{\beta}-\jm_{\alpha}\otimes\jm_{\beta}$  \\ \hline
				
				$\xh\xr\xc$ & \makecell{$\xymatrix@C=4ex@R=-1ex{\xh_u\ar@<-.4ex>[r]^{\xh_\beta} & \xr_v \ar@<-.4ex>[r]^{\xc_\alpha} & \xc_w}$}  & $1_{\alpha}\otimes1_{\beta}-\im_{\alpha}\otimes\jm_{\beta}$ &  $1_{\alpha}\otimes1_{\beta}+\im_{\alpha}\otimes\jm_{\beta}$ \\ \cline{2-4} 
				& $1_{\alpha}\otimes1_{\beta}+\im_{\alpha}\otimes\jm_{\beta}$ & $1_{\alpha}\otimes1_{\beta}+\im_{\alpha}\otimes\jm_{\beta}$ & $1_{\alpha}\otimes1_{\beta}-\im_{\alpha}\otimes\jm_{\beta}$  \\ \hline
				
				$\xc\xr\xh$ & \makecell{$\xymatrix@C=4ex@R=-1ex{\xc_u\ar@<-.4ex>[r]^{\xc_\beta} & \xr_v \ar@<-.4ex>[r]^{\xh_\alpha} & \xh_w}$}  & $1_{\alpha}\otimes1_{\beta}-\jm_{\alpha}\otimes\im_{\beta}$ &  $1_{\alpha}\otimes1_{\beta}+\jm_{\alpha}\otimes\im_{\beta}$\\ \cline{2-4} 
				& $1_{\alpha}\otimes1_{\beta}+\jm_{\alpha}\otimes\im_{\beta}$ & $1_{\alpha}\otimes1_{\beta}+\jm_{\alpha}\otimes\im_{\beta}$ & $1_{\alpha}\otimes1_{\beta}-\jm_{\alpha}\otimes\im_{\beta}$ \\ \hline
				
				$\xc\xr\xc$ & \makecell{$\xymatrix@C=4ex@R=-1ex{\xc_u\ar@<-.4ex>[r]^{\xc_\beta} & \xr_v \ar@<-.4ex>[r]^{\xc_\alpha} & \xc_w}$}  & $1_{\alpha}\otimes1_{\beta}-\im_{\alpha}\otimes\im_{\beta}$ &  $1_{\alpha}\otimes1_{\beta}+\im_{\alpha}\otimes\im_{\beta}$ \\ \cline{2-4} 
				& $1_{\alpha}\otimes1_{\beta}+\im_{\alpha}\otimes\im_{\beta}$ & $1_{\alpha}\otimes1_{\beta}+\im_{\alpha}\otimes\im_{\beta}$ & $1_{\alpha}\otimes1_{\beta}-\im_{\alpha}\otimes\im_{\beta}$ \\ \hline
				
				$\xr\xh\xc$ & \makecell{$\xymatrix@C=4ex@R=-1ex{\xr_u\ar@<-.4ex>[r]^{\xh_\beta} & \xh_v \ar@<-.4ex>[r]^{\xc^2_\alpha} & \xc_w}$}  & $\tiny{\begin{bmatrix}1&0\end{bmatrix}}_{\alpha}\otimes1_{\beta}$ &  $\tiny{\begin{bmatrix}1&0\end{bmatrix}}_{\alpha}\otimes\km_{\beta}$\\ \cline{2-4} 
				& $\tiny{\begin{bmatrix}1&0\end{bmatrix}}_{\alpha}\otimes1_{\beta}$ & $\tiny{\begin{bmatrix}1&0\end{bmatrix}}_{\alpha}\otimes\km_{\beta}$ & $\tiny{\begin{bmatrix}1&0\end{bmatrix}}_{\alpha}\otimes\km_{\beta}$ \\ \hline
				
				$\xc\xh\xr$ & \makecell{$\xymatrix@C=4ex@R=-1ex{\xc_u\ar@<-.4ex>[r]^{\xc^2_\beta} & \xh_v \ar@<-.4ex>[r]^{\xh_\alpha} & \xr_w}$}  &$1_{\alpha}\otimes\tiny{\begin{bmatrix}1\\0\end{bmatrix}}_{\beta}$ &  $-\km_{\alpha}\otimes\tiny{\begin{bmatrix}1\\0\end{bmatrix}}_{\beta}$ \\ \cline{2-4} 
				& $-\km_{\alpha}\otimes\tiny{\begin{bmatrix}1\\0\end{bmatrix}}_{\beta}$ & $1_{\alpha}\otimes\tiny{\begin{bmatrix}1\\0\end{bmatrix}}_{\beta}$ & $-\km_{\alpha}\otimes\tiny{\begin{bmatrix}1\\0\end{bmatrix}}_{\beta}$\\ \hline
				
				$\xc\xh\xc$ & \makecell{$\xymatrix@C=4ex@R=-1ex{\xc_u\ar@<-.4ex>[r]^{\xc^2_\beta} & \xh_v \ar@<-.4ex>[r]^{\xc^2_\alpha} & \xc_w}$}  & $\tiny{\begin{bmatrix}1&0\end{bmatrix}}_{\alpha}\otimes\tiny{\begin{bmatrix}1\\0\end{bmatrix}}_{\beta}$ &  $\tiny{\begin{bmatrix}1&0\end{bmatrix}}_{\alpha}\otimes\tiny{\begin{bmatrix}0\\1\end{bmatrix}}_{\beta}$ \\ \cline{2-4} 
				& $\tiny{\begin{bmatrix}1&0\end{bmatrix}}_{\alpha}\otimes\tiny{\begin{bmatrix}0\\1\end{bmatrix}}_{\beta}$ & $\tiny{-\begin{bmatrix}1&0\end{bmatrix}}_{\alpha}\otimes\tiny{\begin{bmatrix}0\\1\end{bmatrix}}_{\beta}$ & $\tiny{\begin{bmatrix}1&0\end{bmatrix}}_{\alpha}\otimes\tiny{\begin{bmatrix}1\\0\end{bmatrix}}_{\beta}$ \\ \hline
				
				$\xr\xh\xr$ & \makecell{$\xymatrix@C=4ex@R=-1ex{\xr_u\ar@<-.4ex>[r]^{\xh_\beta} & \xh_v \ar@<-.4ex>[r]^{\xh_\alpha} & \xr_w}$}  & $1_{\alpha}\otimes1_{\beta}$ ($1_{\alpha}\otimes\jm_{\beta}$) & $1_{\alpha}\otimes\km_{\beta}$ ($1_{\alpha}\otimes\lm_{\beta}$)\\ \cline{2-4} 
				& $1_{\alpha}\otimes1_{\beta}$ & $1_{\alpha}\otimes1_{\beta}$ & $1_{\alpha}\otimes1_{\beta}$  \\ \hline
		\end{tabular}}
		\caption{Generators in $I$ of $T(Q^\mathrm{b},\mathcal{M})$}\label{phi(ab)}
	\end{table}
	
	\begin{rem}\label{4.1=3.1}
		Quivers with modulations in Table \ref{phi(ab)} are exactly the 18 cases in Table \ref{spgenv}, with $I_{\alpha\beta}^0=I^0_v$ and $I_{\alpha\beta}^1=I^1_v$.
	\end{rem}

	\begin{exm}
		We do computation and check the above conclusions for two cases. We use the notation above and in Table \ref{phi(ab)}.
		
		Case $\xr\xr\xh$. Assume that $u,v\in Q_\xr$, $w\in Q_\xh$ and $\sigma_{\alpha}=\sigma_{\beta}=\mathrm{Id}$. Then
		\begin{align*}
			&\Phi(\epsilon\lceil \alpha\beta\rceil\epsilon)=\phi_2\phi_1(\epsilon_w\lceil \alpha\rceil\otimes\lceil\beta\rceil\epsilon_u)\\
			=&\phi_2(\begin{bmatrix}1&\jm\end{bmatrix}_{\alpha}\otimes_{M_2(\xr)}\begin{bmatrix}1&0\\0&1\end{bmatrix}_{\beta}+\begin{bmatrix}1&\jm\end{bmatrix}_{\alpha}\otimes_{M_2(\xr)}\begin{bmatrix}0&-1\\1&0\end{bmatrix}_{\beta^\im})\begin{bmatrix}1 &0\\0&0\end{bmatrix}_u\\
			=&(1_{\alpha}\otimes_{\xr}\begin{bmatrix}1&0\end{bmatrix}_{\beta}+\jm_{\alpha}\otimes_{\xr}\begin{bmatrix}0&1\end{bmatrix}_{\beta}+1_{\alpha}\otimes_{\xr}\begin{bmatrix}0&-1\end{bmatrix}_{\beta^\im}+\jm_{\alpha}\otimes_{\xr}\begin{bmatrix}1&0\end{bmatrix}_{\beta^\im})\begin{bmatrix}1 &0\\0&0\end{bmatrix}_u\\
			=&1_{\alpha}\otimes1_{\beta}+\jm_{\alpha}\otimes 1_{\beta^\im}.
		\end{align*} 
		\begin{align*}
			&\Phi(\epsilon\lceil \alpha s \beta\rceil\epsilon)=\phi_2\phi_1(\epsilon_w\lceil \alpha s\rceil\otimes\lceil\beta\rceil\epsilon_u)\\
			=&\phi_2(\begin{bmatrix}1&\jm\end{bmatrix}\begin{bmatrix}1&0\\0&-1\end{bmatrix}_v\otimes\begin{bmatrix}1&0\\0&1\end{bmatrix}_{\beta}+\begin{bmatrix}1&\jm\end{bmatrix}\begin{bmatrix}1&0\\0&-1\end{bmatrix}_v\otimes\begin{bmatrix}0&-1\\1&0\end{bmatrix}_{\beta^\im})\begin{bmatrix}1&0\\0&0\end{bmatrix}\\
			=&\phi_2(\begin{bmatrix}1&-\jm\end{bmatrix}\otimes\begin{bmatrix}1&0\\0&1\end{bmatrix}_{\beta}+\begin{bmatrix}1&-\jm\end{bmatrix}\otimes\begin{bmatrix}0&-1\\1&0\end{bmatrix}_{\beta^\im})\begin{bmatrix}1 &0\\0&0\end{bmatrix}\\
			=&(1\otimes\begin{bmatrix}1&0\end{bmatrix}_{\beta}-\jm\otimes\begin{bmatrix}0&1\end{bmatrix}+1\otimes\begin{bmatrix}0&-1\end{bmatrix}-\jm\otimes\begin{bmatrix}1&0\end{bmatrix}_{\beta^\im})\begin{bmatrix}1 &0\\0&0\end{bmatrix}\\
			=&1_{\alpha}\otimes1_{\beta}-\jm_{\alpha}\otimes 1_{\beta^\im}.
		\end{align*} 
		In this case, $I_{\alpha\beta}=\xh_w(1_{\alpha}\otimes1_{\beta}+\jm_{\alpha}\otimes 1_{\beta^\im})\xr_u=I_{\alpha\beta}^0.$
		
		Case $\xr\xh\xr$. Assume that $u,w\in Q_\xr$, $v\in Q_\xh$ and $\sigma_{\alpha}=\sigma_{\beta}=\mathrm{Id}$. Then
		\begin{align*}
			&\Phi(\epsilon\lceil \alpha\beta\rceil\epsilon)=\phi_2\phi_1(\epsilon_w\lceil \alpha\rceil\otimes\lceil\beta\rceil\epsilon_u)\\
			=&\begin{bmatrix}1&0\\0&0\end{bmatrix}_w\phi_2(\begin{bmatrix}1\\-\jm\end{bmatrix}_{\alpha}\otimes_\xh\begin{bmatrix}1&\jm\end{bmatrix}_{\beta})\begin{bmatrix}1&0\\0&0\end{bmatrix}_u=\begin{bmatrix}1&0\\0&0\end{bmatrix}_w(\begin{bmatrix}1\\-\jm\end{bmatrix}_{\alpha}\otimes_\xh\begin{bmatrix}1&\jm\end{bmatrix}_{\beta})\begin{bmatrix}1&0\\0&0\end{bmatrix}_u\\
			=&1_{\alpha}\otimes1_{\beta},\\
			&\Phi(\epsilon\lceil \alpha s \beta\rceil\epsilon)=\phi_2\phi_1(\epsilon_w\lceil \alpha s\rceil\otimes\lceil\beta\rceil\epsilon_u)&\\
			=&\begin{bmatrix}1&0\\0&0\end{bmatrix}\phi_2((\begin{bmatrix}1\\-\jm\end{bmatrix}\km_v)\otimes\begin{bmatrix}1&\jm\end{bmatrix})\begin{bmatrix}1&0\\0&0\end{bmatrix}=\begin{bmatrix}1&0\\0&0\end{bmatrix}(\begin{bmatrix}\km\\-\lm\end{bmatrix}\otimes\begin{bmatrix}1&\jm\end{bmatrix})\begin{bmatrix}1&0\\0&0\end{bmatrix}\\
			=&\km_{\alpha}\otimes1_{\beta},\\
			&\Phi(\epsilon\im\lceil \alpha\beta\rceil\epsilon)=\phi_2\phi_1(\epsilon_w\lceil \im\alpha\rceil\otimes\lceil\beta\rceil\epsilon_u)\\
			=&\begin{bmatrix}1&0\\0&0\end{bmatrix}\phi_2(\begin{bmatrix}0&-1\\1&0\end{bmatrix}\begin{bmatrix}1\\-\jm\end{bmatrix}\otimes\begin{bmatrix}1&\jm\end{bmatrix})\begin{bmatrix}1&0\\0&0\end{bmatrix}=\begin{bmatrix}1&0\\0&0\end{bmatrix}(\begin{bmatrix}\jm\\1\end{bmatrix}\otimes\begin{bmatrix}1&\jm\end{bmatrix})\begin{bmatrix}1&0\\0&0\end{bmatrix}\\
			=&\jm_{\alpha}\otimes1_{\beta},\\
			&\Phi(\epsilon\im\lceil \alpha s\beta\rceil\epsilon)=\phi_2\phi_1(\epsilon_w\lceil \im\alpha s \rceil\otimes\lceil\beta\rceil\epsilon_u)\\
            =&\begin{bmatrix}1&0\\0&0\end{bmatrix}\phi_2(\begin{bmatrix}0&-1\\1&0\end{bmatrix}\begin{bmatrix}1\\-\jm\end{bmatrix}\km\otimes\begin{bmatrix}1&\jm\end{bmatrix})\begin{bmatrix}1&0\\0&0\end{bmatrix}=\begin{bmatrix}1&0\\0&0\end{bmatrix}(\begin{bmatrix}\lm\\1\end{bmatrix}\otimes\begin{bmatrix}1&\jm\end{bmatrix})\begin{bmatrix}1&0\\0&0\end{bmatrix}\\
			=&\lm_{\alpha}\otimes1_{\beta}.
		\end{align*} 
		As a $\mathcal{M}(w)$-$\mathcal{M}(u)$-bimodule, $\mathcal{M}(\alpha\beta)^{\mathrm{b}}=\mathcal{M}(\alpha)\otimes\mathcal{M}(\beta)$ can be generated by these images of $\Phi$. In this case, $$I_{\alpha\beta}=\xr_w(1_{\alpha}\otimes1_{\beta})\xr_u\oplus\xr_w(\jm_{\alpha}\otimes1_{\beta})\xr_u=I_{\alpha\beta}^0.$$
	\end{exm}

	\begin{rem}\label{asb to I1}
		(1) Let $\xc_{\sigma}Q/\langle S\cup Z\rangle$ be a $\xc$-semilinear clannish algebras of gentle type. The results in Table \ref{phi(ab)} depend on the choice of isomorphisms $\phi_1$ and $\Phi$. But $I_{\alpha\beta}^0=I_v^0$ and $I_{\alpha\beta}^1=I_v^1$ are subbimodules of $\mathcal{M}(\alpha\beta)^\mathrm{b}$, which do not depend on these isomorphisms and twists. 
		
		(2) For the cases with $\sigma_\alpha=\sigma_\beta=\mathrm{Id}$ in Table \ref{phi(ab)}, we always have that $\Phi(\epsilon\lceil \alpha\beta\rceil\epsilon)$ belongs to $I_{\alpha\beta}^0$ and $\Phi(\epsilon\lceil \alpha s\beta\rceil\epsilon)$ to $I_{\alpha\beta}^1$. 
        
        (3) Consider $\xc_{\sigma}Q/\langle S\rangle$ whose modulated quiver presentation is in one of the 18 cases in Table \ref{phi(ab)}. That is, $$Q=\xymatrix{u\ar@(ul,ur)_{s_u}\ar[r]^{\beta}& v\ar@(ul,ur)_{s_v}\ar[r]^{\alpha}& w\ar@(ul,ur)_{s_w}} \mbox{ with } \mathbb{S}=\{s_u,s_v,s_w\},$$
		where $s_u$ or $s_w$ can be omitted if $u$ or $w$ has no special loop on it, respectively.
		
		Let $m_{\alpha}, t_{\alpha}, m_{\beta}, t_{\beta}, t$ be five integers with $t=m_{\alpha}+m_{\beta}-t_{\alpha}-t_{\beta}$. Let $\sigma$ be the twist given by $\sigma_{\alpha}=\sigma_{\beta}=\mathrm{Id}$ and $\delta$ be another twist given by $\delta_{\alpha}=\overline{(?)}^{t_{\alpha}}$ and $\delta_{\beta}=\overline{(?)}^{t_{\beta}}$. Consider a $\xc$-linear map $$f:\xc_{\sigma}Q\rightarrow \xc_{\delta}Q\mbox{ induced by }f(x)=\begin{cases}
			x &\mbox{ if }x\in\{e_u,e_v,e_w,s_u,s_v,s_w\}\\
			s_w^{m_{\alpha}}\alpha s_v^{t_{\alpha}-m_{\alpha}} &\mbox{ if }x=\alpha\\
			s_v^{t_{\beta}-m_{\beta}}\beta s_u^{m_{\beta}} &\mbox{ if }x=\beta
		\end{cases}.$$ Here, we take $m_{\alpha}=0$ (or $m_{\beta}=0$) and $s_w^0=\mathrm{Id}$ (resp. $s_u^0=\mathrm{Id}$) if $s_w$ (resp. $s_u$) is omitted. Then $f$ induces an isomorphism of $\xc$-algebras
		$$\tilde{f}\colon\xc_{\sigma}Q/\langle S\rangle\rightarrow\xc_{\delta}Q/\langle S\rangle.$$
		We have $\tilde{f}(\lceil\alpha s_v^t\beta\rceil)=\lceil s_w^m\alpha\beta s_u^n\rceil$. And it induces an isomorphism
		$$\xc_{\sigma}Q/\langle S\cup \{\alpha s_v^t\beta\}\rangle\rightarrow\xc_{\delta}Q/\langle S\cup \{\alpha\beta\}\rangle.$$
		In particular, if we fix $t=1$, one can choose a proper twist $\delta$ and an isomorphism $\Phi\circ \tilde{f}^{-1}\colon\xc_{\delta}Q/\langle S\rangle\rightarrow T(Q^\mathrm{b},\mathcal{M})$ which maps $\epsilon\lceil \alpha\beta\rceil\epsilon$ into $I^1_{\alpha\beta}$.
	\end{rem}
	\medskip
	
	\subsection{The \texorpdfstring{$\xc$}{}-semilinear clannish algebras of gentle type and the locally complexified-gentle algebras}
	
	We investigate the relation between $\xc$-semilinear clannish algebras of gentle type and locally complexified-gentle algebras. 
	
	\begin{prop}\label{semicgenissprgen}
		Any $\xc$-semilinear clannish algebra of gentle type is locally complexified-gentle of special type.
	\end{prop}
	\begin{proof}
		Assume that $\xc_{\sigma}Q/\langle S\cup Z\rangle$ is a $\xc$-semilinear clannish algebra of gentle type, where $Z=\{ \alpha_1\beta_1, \alpha_2\beta_2,\dots,\alpha_r\beta_r \}$. By Theorem \ref{basic form of cq/sz} and the description of $I$ in the former subsection, $\xc_{\sigma}Q/\langle S\cup Z\rangle$ is Morita equivalent to $T(Q^\mathrm{b}, \mathcal{M})/I$, where $I=\langle I_{\alpha_1\beta_1},I_{\alpha_2\beta_2},\dots, I_{\alpha_r\beta_r}\rangle$. 
		
		By the construction of $Q^\mathrm{b}$ and the assumption that $\xc_{\sigma}Q/\langle S\cup Z\rangle$ is a $\xc$-semilinear clannish algebra of gentle type, we have the following observations. For each $v\in Q_\xc$, $v$ is ordinarily gentle in $T(Q^\mathrm{b}, \mathcal{M})/I$, with $I_v=I_{\alpha_i\beta_i}$ if $v=s(\alpha_i)$ in $Q$ for some $\alpha_i\beta_i\in Z$ and $I_v=0$ otherwise. For each $v\in Q_\xr\cup Q_\xh$, $v$ appears in Table \ref{spgenv} and Table \ref{phi(ab)}. So $v$ is specially gentle in $T(Q^\mathrm{b}, \mathcal{M})/I$ with $I_v=I_{\alpha_i\beta_i}$ for some $\alpha_i\beta_i\in Z$. 
		
		Hence $\xc_{\sigma}Q/\langle S\cup Z\rangle$ is locally complexified-gentle of special type.
	\end{proof}
	
	We prove that a connected locally complexified-gentle algebra of uniform type or special type is Morita equivalent to some semilinear clannish algebra.
	
	Let $A$ be a locally complexified-gentle algebra of uniform type. By Proposition~\ref{cpm-gen of uni}, $A$ is Morita equivalent to some locally gentle algebra $\xr Q/\langle R \rangle$ or $\xh Q/\langle R \rangle$. So, we have the following result.
	
	\begin{prop}\label{oga is slca}
		Any locally complexified-gentle algebra of uniform type is a semilinear clannish algebra over $\xr$ or $\xh$. In particular, they are normally bounded, non-singular, and of semisimple type in the sense of \cite{BC2024}.
	\end{prop}
	
	Now we assume that $A$ is Morita equivalent to some locally complexified-gentle algebra $T(Q,\mathcal{M})/I$ of special type. From $T(Q,\mathcal{M})/I$, we construct two semilinear clannish algebras $$\xc_{\sigma}Q^{\mathrm{s}}/\langle S\cup Z\rangle\mbox{ and }\xc_{\sigma'}Q^{\mathrm{s}}/\langle S\cup Z'\rangle.$$
	As in Table \ref{spgenv} and Table \ref{phi(ab)}, we use the notation $\gamma$ and $\gamma^\im$ if $s(\gamma)=s(\gamma^\im)$, $t(\gamma)=t(\gamma^\im)$ and $\mathcal{M}(s(\gamma))=\mathcal{M}(t(\gamma))\neq\xc$ for some arrow $\gamma$ and we denote $Q_1^\im$ the set of all these $\gamma^\im$ in $Q_1$.

	\textbf{Construction of $\xc_{\sigma}Q^{\mathrm{s}}/\langle S\cup Z\rangle$ and $\xc_{\sigma'}Q^{\mathrm{s}}/\langle S\cup Z'\rangle$ from $T(Q,\mathcal{M})/I$}:

	\begin{enumerate}
		\item[(1)] Set $Q^{\mathrm{s}}_0=Q_0$ with $Q^s_{D}=\{i\in Q_0\,|\,\mathcal{M}(i)=D\}, D\in\{\xr,\xh,\xc\}$; \\
		and set $Q^{\mathrm{s}}_1=(Q_1\setminus Q_1^\im)\sqcup\{s_i:i\rightarrow i\;|\;i\in Q_0, \mathcal{M}(i)\neq\xc\}$.\\
		Hence $\mathbb{S}=\{s_i\,|\,i\in Q^s_{\xr}\cup Q^s_{\xh}\}$ and $S=\{s_i^2-e_i \,|\, i\in Q^s_{\xr}\}\cup\{s_i^2+e_i\,|\, i\in Q^s_{\xh}\}$.
		\item[(2)] Set $Z=\{\alpha\beta\,|\,\alpha,\beta\in Q_1, s(\alpha)=t(\beta)=v\in Q^s_{\xc} \mbox{ and } 0\neq I_v\subset \mathcal{M}(\alpha\beta)\}\\ \mbox{ }\qquad\quad\cup\{\alpha\beta\,|\,\alpha,\beta\in Q_1\setminus Q_1^\im, s(\alpha)=t(\beta)\in Q^s_{\xr}\cup Q^s_{\xh}\}.$
		\item[(2')] Set $Z'=\{\alpha\beta\,|\,\alpha,\beta\in Q_1, s(\alpha)=t(\beta)=v\in Q^s_{\xc} \mbox{ and } 0\neq I_v\subset \mathcal{M}(\alpha\beta)\}\\ \mbox{ }\cup\{\alpha s^p_v\beta\,|\,\alpha,\beta\in Q_1\setminus Q_1^\im, s(\alpha)=t(\beta)=v\notin Q^s_{\xc} \mbox{, and } I_v=I_v^p,p\in\{0,1\}\}.$
	\end{enumerate}

	We call a path $\alpha_1\alpha_2\cdots\alpha_n$ with arrows in $Q^\mathrm{s}_1\setminus\mathbb{S}$ a special path if $\alpha_i\neq\alpha_j$ for $i\neq j$ and $s(\alpha_i)\notin Q^{\mathrm{s}}_{\xc}$ for $1\leq i\leq n-1$. A special path is called maximal if it is not properly contained in another special path. For those maximal special paths whose underlying quivers are the same oriented cycle up to rotation, we fix one as a representative. Since $T(Q,\mathcal{M})/I$ is locally complexified-gentle of special type, each ordinary arrow $\alpha\in Q_1^{\mathrm{s}}\setminus\mathbb{S}$ either is given by an arrow in $T(Q,\mathcal{M})/I$ with $\mathcal{M}(s(\alpha))=\mathcal{M}(t(\alpha))=\xc$ or is contained uniquely in one maximal special path.

	\begin{enumerate}	
		\item[(3)] We give a twist $\sigma$ as below: \\
		For each special loop $s$, $\sigma_{s}=\overline{(?)}$.\\
		For each ordinary arrow $\alpha$ with $\mathcal{M}(\alpha)=\xc$, $\sigma_{\alpha}=\mathrm{Id}$.\\
		For each ordinary arrow $\alpha$ with $\mathcal{M}(\alpha)=\overline{\xc}$, $\sigma_{\alpha}=\overline{(?)}$.\\
		For each maximal special path $\alpha_1\alpha_2\cdots\alpha_n$ which is an oriented cycle, set $\sigma_{\alpha_i}=\overline{(?)}^{p_i}$, where $p_i\in\{0,1\}$ with $I^{p_i}_{s(\alpha_i)}\subseteq I$, for $1\leq i\leq n$.\\ 
		For each maximal special path $\alpha_1\alpha_2\cdots\alpha_n$ which is not an oriented cycle, set $\sigma_{\alpha_i}=\overline{(?)}^{p_i}$, where $p_i\in\{0,1\}$ such that $I^{p_i}_{s(\alpha_i)}\subseteq I$ for $1\leq i\leq n-1$ and set $\sigma_{\alpha_n}=\mathrm{Id}$.
		
		\item[(3')] We give a twist $\sigma'$ as below: \\
		For each special loop $s$, $\sigma'_{s}=\overline{(?)}$.\\
		For each ordinary arrow $\alpha$ with $\mathcal{M}(\alpha)=\xc$, $\sigma'_{\alpha}=\mathrm{Id}$.\\
		For each ordinary arrow $\alpha$ with $\mathcal{M}(\alpha)=\overline{\xc}$, $\sigma'_{\alpha}=\overline{(?)}$.\\
		For each ordinary arrow $\alpha$ contained in some special path, set $\sigma'_{\alpha}=\mathrm{Id}$.
	\end{enumerate}
	
	\begin{prop}\label{Qsissemicgentle}
		If $T(Q,\mathcal{M})/I$ is locally complexified-gentle of special type, then 
		\begin{enumerate}
			\item[\rm (1)] the algebra $T(Q,\mathcal{M})/I$ is Morita equivalent to $\xc_{\sigma'}Q^{\mathrm{s}}/\langle S\cup Z'\rangle$;
			\item[\rm (2)] the algebras $\xc_{\sigma'}Q^{\mathrm{s}}/\langle S\cup Z'\rangle$ and $\xc_{\sigma}Q^{\mathrm{s}}/\langle S\cup Z\rangle$ are isomorphic;
			\item[\rm (3)] the algebra $\xc_{\sigma}Q^{\mathrm{s}}/\langle S\cup Z\rangle$ is a $\xc$-semilinear clannish algebra of gentle type.
		\end{enumerate}
	\end{prop}
	
	\begin{proof}
		(1). From $\xc_{\sigma'}Q^{\mathrm{s}}/\langle S\rangle$ we construct $T((Q^{\mathrm{s}})^{\mathrm{b}},\mathcal{M}')$ as in Section 3.1. Immediately, we have $((Q^{\mathrm{s}})^{\mathrm{b}},\mathcal{M}')=(Q,\mathcal{M})$. By Theorem~\ref{CQS and TQM},  $\Phi:\epsilon(\xc_{\sigma'}Q^{\mathrm{s}}/\langle S\rangle)\epsilon\rightarrow T(Q,\mathcal{M})$ induces $$\epsilon(\xc_{\sigma'}Q^{\mathrm{s}}/\langle S\cup Z'\rangle)\epsilon\simeq T(Q,\mathcal{M})/\Phi(\langle \lceil p\rceil\,|\,p\in Z'\rangle).$$ 
		By the construction of $Z'$ and Remark \ref{asb to I1} (2), we have $\Phi(\langle \lceil p\rceil\,|\,p\in Z'\rangle)=I$. Hence $\xc_{\sigma'}Q^{\mathrm{s}}/\langle S\cup Z'\rangle$ is Morita equivalent to $T(Q,\mathcal{M})/I$.
		
		(2). Consider two $\xc$-linear maps $f_1:\xc_{\sigma}Q^{\mathrm{s}}\rightarrow\xc_{\sigma'}Q^{\mathrm{s}}$ and $f_2:\xc_{\sigma'}Q^{\mathrm{s}}\rightarrow\xc_{\sigma}Q^{\mathrm{s}}$, both induced by 
		$$f(x)=\begin{cases}
			x &\mbox{ if }x\in \{e_i\,|\,i\in Q^{\mathrm{s}}_0\}\cup\mathbb{S}\cup\{x\in Q^{\mathrm{s}}_1\,|\, s(x)\in Q^{\mathrm{s}}_{\xc}\}\\
			x s_{s(x)}^{p} &\mbox{ if }x\in Q^{\mathrm{s}}_1\setminus\mathbb{S}, s(x)\notin Q^{\mathrm{s}}_{\xc}\mbox{ with }\sigma_{x}=\overline{(?)}^p, p\in\{0,1\}
		\end{cases}.$$
		They are well-defined because $\sigma$ and $\sigma'$ are different only on those arrows $\alpha\in Q^{\mathrm{s}}_1\setminus\mathbb{S}$ with $\sigma_{\alpha}=\overline{(?)}$. Further, $f_1$ and $f_2$ induce two isomorphisms of $\xc$-semilinear algebras $g_1:\xc_{\sigma}Q^{\mathrm{s}}/\langle S\rangle\rightarrow\xc_{\sigma'}Q^{\mathrm{s}}/\langle S\rangle$ and $g_2:\xc_{\sigma'}Q^{\mathrm{s}}/\langle S\rangle\rightarrow\xc_{\sigma}Q^{\mathrm{s}}/\langle S\rangle$, which are inverse to each other.
		
		For each $\alpha\beta$ in $Z$, we have $g_1(\lceil\alpha\beta\rceil)=\lceil\alpha s_{s(\alpha)}^p\beta s_{s(\beta)}^q\rceil$, where 
		$$p=\begin{cases}
			1 &\mbox{ if } s(\alpha)\notin Q^{\mathrm{s}}_{\xc}\mbox{ with }\sigma_{\alpha}=\overline{(?)}\\
			0 &\mbox{ otherwise }		
		\end{cases}.$$
		Then $g_1(\langle Z\rangle)=\langle Z'\rangle$. Therefore, $g_1$ induces an isomorphism $$\xc_{\sigma}Q^{\mathrm{s}}/\langle S\cup Z\rangle\simeq\xc_{\sigma'}Q^{\mathrm{s}}/\langle S\cup Z'\rangle.$$
		
		(3). By the assumption on $T(Q,\mathcal{M})/I$ and the construction of $\xc_{\sigma}Q^{\mathrm{s}}/\langle S\cup Z\rangle$, it is routine to check that $\xc_{\sigma}Q^{\mathrm{s}}/\langle S\cup Z\rangle$ is a $\xc$-semilinear clannish algebra of gentle type.
	\end{proof}
	
	By Proposition \ref{semicgenissprgen} and Proposition \ref{Qsissemicgentle}, we have 
	\begin{thm}\label{main 2}
		An $\xr$-algebra is locally complexified-gentle of special type if and only if it is Morita equivalent to a $\xc$-semilinear clannish algebra of gentle type. Therefore, locally complexified-gentle algebras of uniform or special type are semilinear clannish algebras up to Morita equivalence.
	\end{thm}

    \begin{exm}
        We take $u=v=w$ in Example \ref{unimform diff from special}. In this case, we identify $\alpha=\beta$ and $\alpha^\im=\beta^\im$. Let $I'=\langle \alpha^2+(\alpha^\im)^2,\alpha\alpha^\im-\alpha^\im\alpha\rangle$. Then $$T(Q,\mathcal{M})/\langle 1_\alpha\otimes1_\alpha+1_{\alpha^\im}\otimes1_{\alpha^\im},1_\alpha\otimes1_{\alpha^\im}-1_{\alpha^\im}\otimes1_\alpha\rangle\simeq\xr Q/I'$$ is a locally complexified-gentle algebra of special type. By the constructions in the proof of Proposition \ref{Qsissemicgentle}, it is Morita equivalent to the algebra $\xc_{\sigma'}Q^{\mathrm{s}}/\langle S\cup Z'\rangle$ given by $$Q^\mathrm{s}=\makecell{\xymatrix{ v\ar@(dl,ul)^\alpha\ar@(dr,ur)_{s_v}}}, S=\{s_v^2-e_v\}, \sigma'_{\alpha}=\mathrm{Id} \mbox{ and } Z'=\{\alpha s_v\alpha\}.$$ And $\xc_{\sigma'}Q^{\mathrm{s}}/\langle S\cup Z'\rangle$ is isomorphic to a $\xc$-semilinear clannish algebra $\xc_{\sigma}Q^{\mathrm{s}}/\langle S\cup Z\rangle$ given by $$Q^\mathrm{s}=\makecell{\xymatrix{ v\ar@(dl,ul)^\alpha\ar@(dr,ur)_{s_v}}}, S=\{s_v^2-e_v\}, \sigma_{\alpha}=\overline{(?)} \mbox{ and } Z=\{\alpha^2\}.$$

        Notice that $\xr Q/I'\simeq\xr[x,y]/\langle x^2+y^2\rangle$, which was mentioned as an example in the penultimate paragraph of \cite[Introduction]{BC2024}. We point out that the condition ``\,$t^2=-1$'' there should be replaced by ``\,$t^2=1$'', according to the above discussion. The condition $t^2=-1$ gives a $\xc$-semilinear clannish algebra $\xc_{\sigma}Q^{\mathrm{s}}/\langle S'\cup Z\rangle$ provided by $$Q^\mathrm{s}=\makecell{\xymatrix{ v\ar@(dl,ul)^\alpha\ar@(dr,ur)_{s_v}}}, S'=\{s_v^2+e_v\}, \sigma_{\alpha}=\overline{(?)} \mbox{ and } Z=\{\alpha^2\}.$$ It is Morita equivalent (even isomorphic) to $\xh Q/\langle \km\alpha^2+\km(\alpha^\im)^2-\lm\alpha\alpha^\im+\lm\alpha^\im\alpha\rangle$; see $r_1$ for case $\xh\xh\xh$ in Table \ref{spgenv} and Example \ref{unimform diff from special}. This is an algebra equals $\xh Q/\langle \alpha^2+(\alpha^\im)^2, \alpha\alpha^\im-\alpha^\im\alpha\rangle$, which is isomorphic to $\xh[x,y]/\langle x^2+y^2\rangle$.
    \end{exm}

	\begin{rem}
		One can analogously define $\xc$-semilinear clannish algebras of other types, say string type. But the situation becomes different. Example \ref{r-string alg} gives an $\xr$-algebra which becomes a string algebra after complexification. But it is not a modulated quiver presentation of any $\xc$-semilinear clannish algebra, otherwise there should be an extra arrow $\gamma^\im$ in $Q$ by the construction in Section 4.2.
	\end{rem}
	
	\medskip
	
	\subsection{Representations of locally complexified-gentle algebras of uniform and special types}
	
	According to Proposition \ref{oga is slca} and Proposition \ref{Qsissemicgentle}, locally complexified-gentle algebras of uniform or special type are semilinear clannish algebras up to Morita equivalence. Moreover, they are normally bounded, non-singular, and of semisimple type, whose finite-dimensional modules are classified in \cite{BC2024}. Hence, we can classify modules over locally complexified-gentle algebras of these two types.
	
	If $T(Q,\mathcal{M})/I$ is locally complexified-gentle of uniform type with $\xr$ or $\xh$, it is a semilinear clannish (even gentle) algebra via a simple isomorphism. So we can directly apply the main theorem in \cite{BC2024} to classify finite-dimensional modules via ``strings and bands''.  
	
	If $T(Q,\mathcal{M})/I$ is locally complexified-gentle of special type, we have equivalences of module categories
	$$T(Q,\mathcal{M})/I\mbox{-}\mathrm{mod}\simeq\xc_{\sigma'}Q^{\mathrm{s}}/\langle S\cup Z'\rangle\mbox{-}\mathrm{mod}\simeq\xc_{\sigma}Q^{\mathrm{s}}/\langle S\cup Z\rangle\mbox{-}\mathrm{mod},$$
	where $\xc_{\sigma'}Q^{\mathrm{s}}/\langle S\cup Z'\rangle$ and $\xc_{\sigma}Q^{\mathrm{s}}/\langle S\cup Z\rangle$ are constructed as in Proposition \ref{Qsissemicgentle}. We apply the main theorem in \cite{BC2024} to give all finite-dimensional $\xc_{\sigma}Q^{\mathrm{s}}/\langle S\cup Z\rangle$-modules. Through the isomorphism $\xc_{\sigma'}Q^{\mathrm{s}}/\langle S\cup Z'\rangle\simeq\xc_{\sigma}Q^{\mathrm{s}}/\langle S\cup Z\rangle$ in Proposition~\ref{Qsissemicgentle} and the isomorphism $\epsilon\xc_{\sigma}Q/\langle S\cup Z \rangle\epsilon\simeq T(Q^\mathrm{b},\mathcal{M})/I$ in Theorem~\ref{basic form of cq/sz}, the isomorphisms between the module categories can be described concretely. Therefore, it is possible to classify all finite-dimensional $T(Q,\mathcal{M})/I$-modules. However, it may be cumbersome, as shown in \cite{DR1978}. 
	
	\subsection*{Acknowledgements}
	The authors thank Professor Xiao-Wu Chen for his encouragement and advice. This work is supported by the National Natural Science Foundation of China (Grant Nos. 12171297, 12201166, 12461006), the Guizhou Provincial Basic Research Program (Grant No. ZD[2025]085), the Fundamental Research Funds for the Central Universities and the Guizhou Provincial Key Laboratory of Applied Mathematics and Computing Power \& Algorithms  (No. QianKeHePingTai ZSYS[2025]039).

	\bibliographystyle{plain}
	\bibliography{ref}

\end{document}